\documentclass[twoside,11pt]{article}  %the ``twoside'' allows title/author to alternate between even (E) and odd (O) pages

\usepackage{amsmath,amssymb,amsfonts,amsthm,enumerate,titlesec,amscd,tikz,calc}
\usepackage{hyperref}
\usepackage{xcolor}  %the next few lines set the colors of the hyperlinks, and remove the color boxes
\hypersetup{
    colorlinks,
    linkcolor={red!50!black},
    citecolor={blue!70!black},
    urlcolor={blue!80!black}
}

\usepackage[centering,top=2.5cm,bottom=2.5cm,left=2.73cm,right=2.73cm,headsep=.15in]{geometry} %sets margins

\usepackage{fancyhdr}
\makeatletter
\newcommand\Author{\small Y. LOIZIDES}
\newcommand\Title{\small NORM-SQUARE LOCALIZATION}
\makeatother

\usepackage[runin]{abstract}  %``ABSTRACT'' appears in line with text of the abstract (the runin command)
\titleformat{\section}{\normalfont\scshape\centering}{\thesection}{1em}{}  %formats section headings, scshape=small caps
\titleformat{\subsection}[runin]{\bfseries}{\thesubsection}{1em}{}[]  %runin=no extra space, text appears inline with title of subsection
\titleformat{\subsubsection}[runin]{\bfseries}{\thesubsubsection}{1em}{}[]

\usepackage[titles]{tocloft}  %allows formatting in table of contents (toc)
\theoremstyle{plain} \newtheorem{theorem}{Theorem}[section] \newtheorem{proposition}[theorem]{Proposition} \newtheorem{lemma}[theorem]{Lemma} \newtheorem{corollary}[theorem]{Corollary}
\theoremstyle{definition} \newtheorem{definition}[theorem]{Definition}  
\theoremstyle{remark} \newtheorem{remark}[theorem]{Remark}

\def\lieg{\ensuremath{\mathfrak{g}}}
\def\lieh{\ensuremath{\mathfrak{h}}}
\def\liet{\ensuremath{\mathfrak{t}}}

\def\Eul{\ensuremath{\textnormal{Eul}}}
\def\Ad{\ensuremath{\textnormal{Ad}}}

\def\pr{\ensuremath{\textnormal{pr}}} %projection
\def\Pol{\ensuremath{\textnormal{Pol}}}
\def\vol{\ensuremath{\textnormal{vol}}}
\def\im{\ensuremath{\textnormal{im}}}

\def\ker{\ensuremath{\textnormal{ker}}}

\def\supp{\ensuremath{\textnormal{supp}}}

\def\comp{\ensuremath{\textnormal{comp}}}

\def\ann{\ensuremath{\textnormal{ann}}}

\def\anti{\ensuremath{\textnormal{anti}}}

\def\hor{\ensuremath{\textnormal{hor}}}
\def\sgn{\ensuremath{\textnormal{sgn}}}

\def\comp{\ensuremath{\textnormal{comp}}}
\def\red{\ensuremath{\textnormal{red}}}

\def\Hom{\ensuremath{\textnormal{Hom}}}

\def\Hol{\ensuremath{\textnormal{Hol}}}
\def\DH{\ensuremath{\textnormal{DH}}}

\def\a{\ensuremath{\textnormal{a}}}
\def\Id{\ensuremath{\textnormal{Id}}}

\def\min{\ensuremath{\textnormal{min}}}
\def\Umod{\ensuremath{\textnormal{U}_{\textnormal{mod}}}}

\def\frakm{\ensuremath{\mathfrak{m}}}  %caligraphic m

\def\fraku{\ensuremath{\mathfrak{u}}}

\def\calM{\ensuremath{\mathcal{M}}} %loop group space
\def\calR{\ensuremath{\mathcal{R}}}
\def\calC{\ensuremath{\mathcal{C}}}

\def\calW{\ensuremath{\mathcal{W}}}

\def\calF{\ensuremath{\mathcal{F}}}
\def\calV{\ensuremath{\mathcal{V}}} %volume function

\def\calB{\ensuremath{\mathcal{B}}}

\def\calD{\ensuremath{\mathcal{D}}}
\def\calY{\ensuremath{\mathcal{Y}}}
\def\calU{\ensuremath{\mathcal{U}}}
\def\calT{\ensuremath{\mathcal{T}}}
\def\calN{\ensuremath{\mathcal{N}}}

\def\calZ{\ensuremath{\mathcal{Z}}}

\def\calS{\ensuremath{\mathcal{S}}}
\def\calX{\ensuremath{\mathcal{X}}}

\def\hcalX{\ensuremath{\hat{\mathcal{X}}}}

\def\hPhi{\ensuremath{\hat{\Phi}}}
\def\hX{\ensuremath{\hat{X}}}

\def\bC{\ensuremath{\mathbb{C}}}

\def\bbeta{\ensuremath{\overline{\beta}}}
\def\balpha{\ensuremath{\overline{\alpha}}}
\def\tN{\ensuremath{\tilde{N}}}
\def\talpha{\ensuremath{\tilde{\alpha}}}

\newcommand{\pair}[2]{\langle #1, #2 \rangle}
\renewcommand{\i}{{\mathrm{i}}}

\usepackage[shortlabels]{enumitem}
\usetikzlibrary{patterns,calc,intersections}
\def\tomega{\ensuremath{\tilde{\omega}}}
\def\tphi{\ensuremath{\tilde{\phi}}}
\title{\normalsize \bfseries NORM-SQUARE LOCALIZATION FOR HAMILTONIAN $LG$-SPACES}
\author{\normalsize \textsc{Yiannis Loizides}}
\date{\vspace{-7ex}} %no date

\usepackage{verbatim}   %this next paragraph allows document to be compiled without pictures!
\newif\ifshowtikz
\showtikztrue
\let\oldtikzpicture\tikzpicture
\let\oldendtikzpicture\endtikzpicture
\renewenvironment{tikzpicture}{%
    \ifshowtikz\expandafter\oldtikzpicture%
    \else\comment%   
    \fi
}{%
    \ifshowtikz\oldendtikzpicture%
    \else\endcomment%
    \fi
}

\newlength{\bibitemsep}
\newlength{\bibparskip}
\let\oldthebibliography\thebibliography
\renewcommand\thebibliography[1]{%
  \oldthebibliography{#1}%
  \setlength{\parskip}{\bibitemsep}%
  \setlength{\itemsep}{\bibparskip}%
}

\begin{document}
\maketitle
\thispagestyle{empty} %remove page number from the first page (must appear after maketitle)
\vspace{2ex}
\begin{abstract}
We prove a formula for twisted Duistermaat-Heckman distributions associated to a Hamiltonian $LG$-space.  The terms of the formula are localized at the critical points of the norm-square of the moment map, and can be computed in cross-sections.  Our main tools are the theory of quasi-Hamiltonian $G$-spaces, as well as the Hamiltonian cobordism approach to norm-square localization introduced recently by Harada and Karshon.
\end{abstract}

\section{Introduction}
Let $M$ be a Hamiltonian $G$-space with moment map $\phi:M \rightarrow \lieg^\ast$, and equip $\lieg^\ast$ with an inner product.  The function $\|\phi\|^2:M\rightarrow \mathbb{R}$ has been studied extensively.  An early paper of Atiyah-Bott \cite{AtiyahBottYangMills} studied an infinite dimensional example, the Yang-Mills functional on the space of connections on a compact Riemann surface.  In the setting where $M$ and $G$ are compact, Kirwan \cite{Kirwan} extended techniques of Morse theory to $\|\phi\|^2$, using this to prove \emph{Kirwan surjectivity}.  Closer to the subject of this paper are the early results of Witten \cite{Witten}, who studied certain integrals on $\lieg^\ast \times M$, and found that they localize to the critical set of $\|\phi\|^2$, with the dominant contribution coming from the $0$-level set.  In \cite{Paradan97,Paradan98}, Paradan proved a norm-square localization formula for twisted Duistermaat-Heckman measures, in the general setting where $M$ can be non-compact, but $\phi$ is proper.  Another approach to these norm-square localization formulas was developed by Woodward \cite{Woodward}, and more recently an approach based on Hamiltonian cobordism techniques was introduced by Harada and Karshon \cite{KarshonHarada}.  In their approach, $M$ is found to be cobordant (in a suitable sense) to a small open neighbourhood of the critical set, once the moment map on the neighbourhood has been suitably \emph{polarized} and \emph{completed}.  A norm-square localization formula for twisted Duistermaat-Heckman measures then follows from Stokes' theorem.  We give an overview of the Harada-Karshon Theorem in Section 2 and Appendix B.

Many well-known results on compact Hamiltonian $G$-spaces have parallels for proper Hamiltonian $LG$-spaces ($LG$ is the loop group).  Examples include the cross-section theorem, the convexity theorem, and Duistermaat-Heckman formulas, cf. \cite{MWVerlindeFactorization,AlekseevMalkinMeinrenken,
AMWDuistermaatHeckman}.  The objects of study in this paper are twisted Duistermaat-Heckman (DH) distributions for Hamiltonian $LG$-spaces that carry information about cohomology pairings on symplectic quotients.  For example, the (untwisted) DH distribution that we study is a signed measure on $\liet$ (the Lie algebra of a maximal torus $T \subset G$) that gives volumes of symplectic quotients.

Let $\Psi:\calM \rightarrow L\lieg^\ast$ be a proper Hamiltonian $LG$-space.  We prove a `norm-square localization' formula expressing a twisted DH distribution $\frakm$ on $\liet$ as a sum of contributions:
\begin{equation}
\label{Goal}
\frakm = \sum_{\beta \in W \cdot \calB} \frakm_{\beta},
\end{equation}
where $\calB$ indexes components of the critical set of $\|\Psi\|^2$, and $W = N_G(T)/T$ is the Weyl group.  The sum is infinite but locally finite, in the sense that the supports of only finitely many terms intersect each bounded set.  The terms of \eqref{Goal} consist of a central contribution (from the critical value $0$), and correction terms supported in half-spaces not containing the origin.  Using the terminology of Harada-Karshon \cite{KarshonHarada}, the term $\frakm_{\beta}$ can itself be described as a twisted DH distribution for a \emph{polarized completion} of a small finite dimensional submanifold near the part of the critical set indexed by $\beta$.  Figure \ref{fig:WoodwardSums} shows an example of the decomposition \eqref{Goal} for the DH distribution of a multiplicity-free Hamiltonian $LSU(3)$-space (example due to Chris Woodward).  In this example the contribution from $0$ vanishes, and the remaining terms in \eqref{Goal} are constant multiples ($\pm 1$ relative to suitably normalized Lebesgue measure) of indicator functions for half-spaces (see Section 5 for further discussion).

\begin{figure}
\begin{minipage}[c]{0.32\textwidth}
\begin{tikzpicture}[scale=0.6]
\path[use as bounding box] (0,0) circle (4.1);   %this cuts off WHITESPACE at circle
\coordinate (0) at (0,0);

\path[name path=circle] (0,0) circle (4.1);
\path[name path=circle2] (0,0) circle (5);

\begin{pgfinterruptboundingbox}
\path[name path global=line1] ($(30:2)!-100cm!(90:2)$) -- ($(30:2)!100cm!(90:2)$);
\path[name path global=line2] ($(30:-2)!-100cm!(90:-2)$) -- ($(30:-2)!100cm!(90:-2)$);
\path[name path global=line3] ($(30:2)!-100cm!(-30:2)$) -- ($(30:2)!100cm!(-30:2)$);
\path[name path global=line4] ($(30:-2)!-100cm!(-30:-2)$) -- ($(30:-2)!100cm!(-30:-2)$);
\path[name path global=line5] ($(150:2)!-100cm!(90:2)$) -- ($(150:2)!100cm!(90:2)$);
\path[name path global=line6] ($(-30:2)!-100cm!(90:-2)$) -- ($(-30:2)!100cm!(90:-2)$);
\path[name path global=line7] ($(30:4)!-100cm!(90:4)$) -- ($(30:4)!100cm!(90:4)$);
\path[name path global=line8] ($(30:-4)!-100cm!(90:-4)$) -- ($(30:-4)!100cm!(90:-4)$);
\path[name path global=line9] ($(30:4)!-100cm!(-30:4)$) -- ($(30:4)!100cm!(-30:4)$);
\path[name path global=line10] ($(30:-4)!-100cm!(-30:-4)$) -- ($(30:-4)!100cm!(-30:-4)$);
\path[name path global=line11] ($(150:4)!-100cm!(90:4)$) -- ($(150:4)!100cm!(90:4)$);
\path[name path global=line12] ($(-30:4)!-100cm!(90:-4)$) -- ($(-30:4)!100cm!(90:-4)$);
\path[name path global=line14] ($(30:1)!-100cm!(90:1)$) -- ($(30:1)!100cm!(90:1)$);
\end{pgfinterruptboundingbox}

\draw[name intersections={of=circle2 and line14}] \foreach \i in {1,2} {(intersection-\i) coordinate (q\i)};

\clip (0,0) circle (4.1);  %clips everything outside the circle
\draw[fill,color=lightgray] (q1) arc (150:-30:5)--cycle;
\draw[name intersections={of=circle and line1}] (intersection-1)--(intersection-2);
\draw[name intersections={of=circle and line2}] (intersection-1)--(intersection-2);
\draw[name intersections={of=circle and line3}] (intersection-1)--(intersection-2);
\draw[name intersections={of=circle and line4}] (intersection-1)--(intersection-2);
\draw[name intersections={of=circle and line5}] (intersection-1)--(intersection-2);
\draw[name intersections={of=circle and line6}] (intersection-1)--(intersection-2);
\draw[name intersections={of=circle and line7}] (intersection-1)--(intersection-2);
\draw[name intersections={of=circle and line8}] (intersection-1)--(intersection-2);
\draw[name intersections={of=circle and line9}] (intersection-1)--(intersection-2);
\draw[name intersections={of=circle and line10}] (intersection-1)--(intersection-2);
\draw[name intersections={of=circle and line11}] (intersection-1)--(intersection-2);
\draw[name intersections={of=circle and line12}] (intersection-1)--(intersection-2);
\draw (30:-4.1)--(30:4.1);
\draw (90:-4.1)--(90:4.1);
\draw (-30:4.1)--(-30:-4.1);
\coordinate (A) at ($(30:1)!(0,0)!(90:1)$);
\draw[fill] (A) circle (1.5pt);
\end{tikzpicture}
\end{minipage}
\begin{minipage}[c]{0.32\textwidth}
\begin{tikzpicture}[scale=0.6]
\path[use as bounding box] (0,0) circle (4.1);   %this cuts off WHITESPACE at circle
\coordinate (0) at (0,0);

\path[name path=circle] (0,0) circle (4.1);
\path[name path=circle2] (0,0) circle (5);

\begin{pgfinterruptboundingbox}
\path[name path global=line1] ($(30:2)!-100cm!(90:2)$) -- ($(30:2)!100cm!(90:2)$);
\path[name path global=line2] ($(30:-2)!-100cm!(90:-2)$) -- ($(30:-2)!100cm!(90:-2)$);
\path[name path global=line3] ($(30:2)!-100cm!(-30:2)$) -- ($(30:2)!100cm!(-30:2)$);
\path[name path global=line4] ($(30:-2)!-100cm!(-30:-2)$) -- ($(30:-2)!100cm!(-30:-2)$);
\path[name path global=line5] ($(150:2)!-100cm!(90:2)$) -- ($(150:2)!100cm!(90:2)$);
\path[name path global=line6] ($(-30:2)!-100cm!(90:-2)$) -- ($(-30:2)!100cm!(90:-2)$);
\path[name path global=line7] ($(30:4)!-100cm!(90:4)$) -- ($(30:4)!100cm!(90:4)$);
\path[name path global=line8] ($(30:-4)!-100cm!(90:-4)$) -- ($(30:-4)!100cm!(90:-4)$);
\path[name path global=line9] ($(30:4)!-100cm!(-30:4)$) -- ($(30:4)!100cm!(-30:4)$);
\path[name path global=line10] ($(30:-4)!-100cm!(-30:-4)$) -- ($(30:-4)!100cm!(-30:-4)$);
\path[name path global=line11] ($(150:4)!-100cm!(90:4)$) -- ($(150:4)!100cm!(90:4)$);
\path[name path global=line12] ($(-30:4)!-100cm!(90:-4)$) -- ($(-30:4)!100cm!(90:-4)$);
\path[name path global=line13] ($(-30:1)!-100cm!(30:1)$) -- ($(-30:1)!100cm!(30:1)$);
\path[name path global=line14] ($(30:1)!-100cm!(90:1)$) -- ($(30:1)!100cm!(90:1)$);
\path[name path global=line15] ($(90:1)!-100cm!(150:1)$) -- ($(90:1)!100cm!(150:1)$);
\path[name path global=line16] ($(150:1)!-100cm!(210:1)$) -- ($(150:1)!100cm!(210:1)$);
\path[name path global=line17] ($(210:1)!-100cm!(270:1)$) -- ($(210:1)!100cm!(270:1)$);
\path[name path global=line18] ($(270:1)!-100cm!(330:1)$) -- ($(270:1)!100cm!(330:1)$);
\end{pgfinterruptboundingbox}
\draw[name intersections={of=circle2 and line13}] \foreach \i in {1,2} {(intersection-\i) coordinate (p\i)};
\draw[name intersections={of=circle2 and line14}] \foreach \i in {1,2} {(intersection-\i) coordinate (q\i)};
\draw[name intersections={of=circle2 and line15}] \foreach \i in {1,2} {(intersection-\i) coordinate (r\i)};
\draw[name intersections={of=circle2 and line16}] \foreach \i in {1,2} {(intersection-\i) coordinate (s\i)};
\draw[name intersections={of=circle2 and line17}] \foreach \i in {1,2} {(intersection-\i) coordinate (t\i)};
\draw[name intersections={of=circle2 and line18}] \foreach \i in {1,2} {(intersection-\i) coordinate (u\i)};

\foreach \n in {30,150,270} {\path[fill,color=lightgray] (\n:1)--++(\n+60:1)--++(\n+180:1)--cycle;}
\foreach \n in {90,210,330} {\path[fill] (\n:1)--++(\n+60:1)--++(\n+180:1)--cycle;}
\foreach \n in {1,2,...,6} \coordinate (A\n) at ($(60*\n-30:2)!(0,0)!(60*\n+30:2)$);
%\foreach \n in {1,3,5} \path[pattern=dots] (A\n)--++(30+60*\n-60:1)--++(150+60*\n-60:1)--cycle;
%\foreach \n in {2,4,6} \path[pattern=dots] (A\n)--++(30+60*\n-60:1)--++(150+60*\n-60:1)--cycle;
%\foreach \n in {1,2,...,6} \path[pattern=dots] (30+60*\n:3)--++(270+240*\n:1)--++(30+240*\n:1)--cycle;
%\foreach \n in {0,1,...,5} \path[pattern=dots] (30+60*\n:3)--++(90+240*\n:1)--++(210+240*\n:1)--cycle;

\foreach \n in {1,2,6} \coordinate (B\n) at ($(60*\n-30:1)!(0,0)!(60*\n+30:1)$);
\coordinate (c) at ($(30:2)!(0,0)!(90:2)$);
\clip (0,0) circle (4.1);  %clips everything outside the circle
\draw[fill,color=lightgray] (30:1)--(A1)--(90:1)--cycle;
%\draw[fill] (30:1)--(p2)--(q2)--cycle;
%\draw[fill] (90:1)--(q1)--(r2)--cycle;
\draw[fill] (A1)--(p1)--(r1)--cycle;
\draw[fill,color=lightgray] (A2)--(q1)--(s1)--cycle;
\draw[fill] (A3)--(t1)--(r2)--cycle;
\draw[fill,color=lightgray] (A4)--(u2)--(s2)--cycle;
\draw[fill] (A5)--(p2)--(t2)--cycle;
\draw[fill,color=lightgray] (A6)--(u1)--(q2)--cycle;
\draw[name intersections={of=circle and line1}] (intersection-1)--(intersection-2);
\draw[name intersections={of=circle and line2}] (intersection-1)--(intersection-2);
\draw[name intersections={of=circle and line3}] (intersection-1)--(intersection-2);
\draw[name intersections={of=circle and line4}] (intersection-1)--(intersection-2);
\draw[name intersections={of=circle and line5}] (intersection-1)--(intersection-2);
\draw[name intersections={of=circle and line6}] (intersection-1)--(intersection-2);
\draw[name intersections={of=circle and line7}] (intersection-1)--(intersection-2);
\draw[name intersections={of=circle and line8}] (intersection-1)--(intersection-2);
\draw[name intersections={of=circle and line9}] (intersection-1)--(intersection-2);
\draw[name intersections={of=circle and line10}] (intersection-1)--(intersection-2);
\draw[name intersections={of=circle and line11}] (intersection-1)--(intersection-2);
\draw[name intersections={of=circle and line12}] (intersection-1)--(intersection-2);
\draw (30:-4.1)--(30:4.1);
\draw (90:-4.1)--(90:4.1);
\draw (-30:4.1)--(-30:-4.1);
\end{tikzpicture}
\end{minipage}
\begin{minipage}[c]{0.32\textwidth}
\begin{tikzpicture}[scale=0.6]
\path[use as bounding box] (0,0) circle (4.1); %cuts off WHITESPACE at the circle (so it doesn't take too much room in the document) you still need to make sure that your picture falls in this region though, otherwise it will probably overlap with text (see clip command below)
\coordinate (0) at (0,0);
\draw (30:-4.1)--(30:4.1);
\draw (90:-4.1)--(90:4.1);
\draw (-30:4.1)--(-30:-4.1);

\path[name path=circle] (0,0) circle (4.1);

\begin{pgfinterruptboundingbox}
\path[name path global=line1] ($(30:2)!-100cm!(90:2)$) -- ($(30:2)!100cm!(90:2)$);
\path[name path global=line2] ($(30:-2)!-100cm!(90:-2)$) -- ($(30:-2)!100cm!(90:-2)$);
\path[name path global=line3] ($(30:2)!-100cm!(-30:2)$) -- ($(30:2)!100cm!(-30:2)$);
\path[name path global=line4] ($(30:-2)!-100cm!(-30:-2)$) -- ($(30:-2)!100cm!(-30:-2)$);
\path[name path global=line5] ($(150:2)!-100cm!(90:2)$) -- ($(150:2)!100cm!(90:2)$);
\path[name path global=line6] ($(-30:2)!-100cm!(90:-2)$) -- ($(-30:2)!100cm!(90:-2)$);
\path[name path global=line7] ($(30:4)!-100cm!(90:4)$) -- ($(30:4)!100cm!(90:4)$);
\path[name path global=line8] ($(30:-4)!-100cm!(90:-4)$) -- ($(30:-4)!100cm!(90:-4)$);
\path[name path global=line9] ($(30:4)!-100cm!(-30:4)$) -- ($(30:4)!100cm!(-30:4)$);
\path[name path global=line10] ($(30:-4)!-100cm!(-30:-4)$) -- ($(30:-4)!100cm!(-30:-4)$);
\path[name path global=line11] ($(150:4)!-100cm!(90:4)$) -- ($(150:4)!100cm!(90:4)$);
\path[name path global=line12] ($(-30:4)!-100cm!(90:-4)$) -- ($(-30:4)!100cm!(90:-4)$);
\end{pgfinterruptboundingbox}

\draw[name intersections={of=circle and line1}] (intersection-1)--(intersection-2);
\draw[name intersections={of=circle and line2}] (intersection-1)--(intersection-2);
\draw[name intersections={of=circle and line3}] (intersection-1)--(intersection-2);
\draw[name intersections={of=circle and line4}] (intersection-1)--(intersection-2);
\draw[name intersections={of=circle and line5}] (intersection-1)--(intersection-2);
\draw[name intersections={of=circle and line6}] (intersection-1)--(intersection-2);
\draw[name intersections={of=circle and line7}] (intersection-1)--(intersection-2);
\draw[name intersections={of=circle and line8}] (intersection-1)--(intersection-2);
\draw[name intersections={of=circle and line9}] (intersection-1)--(intersection-2);
\draw[name intersections={of=circle and line10}] (intersection-1)--(intersection-2);
\draw[name intersections={of=circle and line11}] (intersection-1)--(intersection-2);
\draw[name intersections={of=circle and line12}] (intersection-1)--(intersection-2);

\foreach \n in {30,150,270} {\path[fill,color=lightgray] (\n:1)--++(\n+60:1)--++(\n+180:1)--cycle;}
\foreach \n in {90,210,330} {\path[fill] (\n:1)--++(\n+60:1)--++(\n+180:1)--cycle;}
\foreach \n in {1,2,...,6} \coordinate (A\n) at ($(60*\n-30:2)!(0,0)!(60*\n+30:2)$);
\foreach \n in {1,3,5} \path[fill] (A\n)--++(30+60*\n-60:1)--++(150+60*\n-60:1)--cycle;
\foreach \n in {2,4,6} \path[fill,color=lightgray] (A\n)--++(30+60*\n-60:1)--++(150+60*\n-60:1)--cycle;
\foreach \n in {1,2,...,6} \path[fill] (30+60*\n:3)--++(270+240*\n:1)--++(30+240*\n:1)--cycle;
\foreach \n in {0,1,...,5} \path[fill,color=lightgray] (30+60*\n:3)--++(90+240*\n:1)--++(210+240*\n:1)--cycle;
%\foreach \n in {0,1,2} \coordinate (c\n) at ($(30+120*\n:3)+(90:1)$);
\clip (0,0) circle (4.1);  %clips off everything outside the circle
\foreach \n in {0,1,2} \path[fill] ($(30+120*\n:3)+(90+120*\n:1)$)--++(90+120*\n:1)--++(-30+120*\n:1)--cycle;
\foreach \n in {0,1,2} \path[fill] ($(90+120*\n:3)+(30+120*\n:1)$)--++(30+120*\n:1)--++(150+120*\n:1)--cycle;
\foreach \n in {0,1,2} \path[fill,color=lightgray] ($(-30+120*\n:3)+(30+120*\n:1)$)--++(30+120*\n:1)--++(-90+120*\n:1)--cycle;
\foreach \n in {0,1,2} \path[fill,color=lightgray] ($(30+120*\n:3)+(-30+120*\n:1)$)--++(-30+120*\n:1)--++(90+120*\n:1)--cycle;
\end{tikzpicture}
\end{minipage}
\label{fig:WoodwardSums} \caption{The left-most image shows a single contribution to the norm-square formula for a certain multiplicity-free Hamiltonian $LSU(3)$-space.  The next two images show the sum of the first 6 (resp. 12) contributions.  Light gray indicates regions where the sum is $+1$, while black indicates regions where it is $-1$.}
\end{figure}
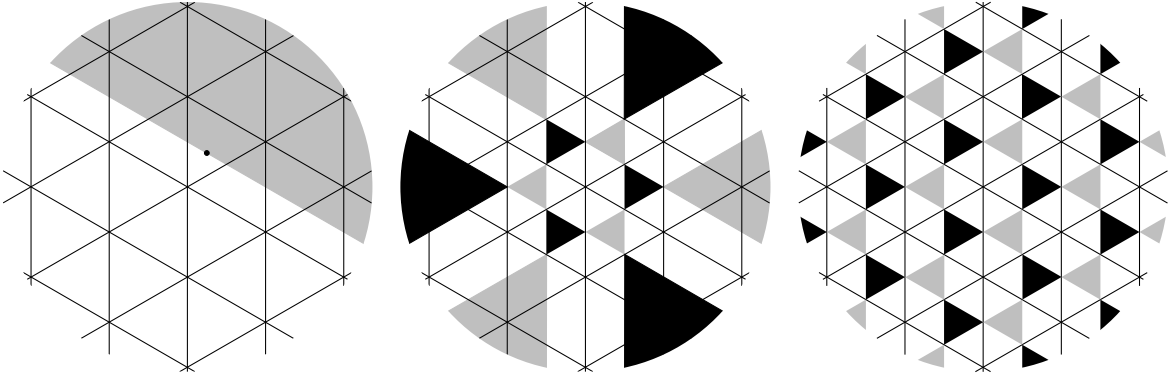

The formula \eqref{Goal} is related to the decomposition formula for Bernoulli series in \cite{BoysalVergne}.  Indeed, a collection of examples of Hamiltonian loop group spaces are moduli spaces of flat connections on a compact Riemann surface having at least 1 boundary component, with moment map given by pullback of the connection to the boundary (cf. \cite{MWVerlindeFactorization}).  In these examples, the twisted Duistermaat-Heckman distributions are essentially Bernoulli series, and the formula \eqref{Goal} coincides with the decomposition formula in \cite{BoysalVergne}.  (These examples are closely related to Witten's formulas for intersection pairings on moduli spaces of flat connections on Riemann surfaces.)  It is possible to derive \eqref{Goal} from the decomposition formula of \cite{BoysalVergne} and the Duistermaat-Heckman formula for group-valued moment maps \cite{AMWEquivariantLocalization}; further details will appear in the author's PhD thesis.  We adopt the cobordism approach here, as this approach leads quickly to the basic result (Theorem \ref{HKNormSquareFormula}).

We briefly outline the contents of the sections.  In Section 2 we summarize results of Harada-Karshon, with small modifications needed for our purposes.  As we rely on these results heavily, for completeness we have included brief outlines of the main proofs in Appendix B.  In Section 3 we recall the 1-1 correspondence between proper Hamiltonian $LG$-spaces and compact \emph{quasi-Hamiltonian} $G$-spaces \cite{AlekseevMalkinMeinrenken}.  A key tool for us is the \emph{abelianization} of a quasi-Hamiltonian $G$-space $M$: this is a (degenerate) quasi-Hamiltonian $T$-space associated to $M$.  Passing to a covering space, we obtain a (degenerate) \emph{Hamiltonian} $T$-space.  In Section 4, we apply the Harada-Karshon Theorem to this Hamiltonian $T$-space.  This yields a (rather inexplicit) norm-square localization formula.  To obtain more explicit formulas as in \cite{Paradan97,Paradan98}, we explain how to compute the contributions in cross-sections.  This takes some work, which is carried out in Section 4.  Section 5 contains examples.  Appendix A contains a proof of a version of abelian localization used in the main part of the paper.  Appendix C briefly discusses the case of a singular localizing set.

\vspace{0.3cm}

\noindent {\textbf{Acknowledgements.}}  I thank my Ph.D. supervisor, Eckhard Meinrenken, for introducing me to this problem, as well as for his patient explanations and insightful suggestions over the past few years.  I thank Michele Vergne for helpful discussions and for providing detailed feedback on an earlier version.  I also thank Yael Karshon and Paul-Emile Paradan for helpful conversations.\\

\noindent {\textbf{Notation.}}  Throughout, $G$ will denote a compact, connected Lie group with Lie algebra $\lieg$.  Fix an Ad-invariant inner product (denoted $\cdot$) on $\lieg$.  Let $T$ be a maximal torus with Lie algebra $\liet \subset \lieg$, and $W=N_G(T)/T$ the Weyl group.  We say that a function $f$ on $\liet$ is $W$-anti-symmetric if for $w \in W$, $w \cdot f = (-1)^{l(w)}f$, where $l(w)$ is the length of the element $w \in W$.  If $H \supset T$ is a maximal rank subgroup with Lie algebra $\lieh$, there is a unique $H$-invariant complement $\lieh^\perp$, the orthogonal complement to $\lieh$ for any invariant inner product on $\lieg$.  Let $\Lambda \subset \liet$ denote the kernel of the exponential map $\exp:\liet \rightarrow T$.  The dual lattice $\Lambda^\ast=\Hom(\Lambda,\mathbb{Z}) \subset \liet^\ast$ is the real weight lattice.  The real roots form a subset $\calR \subset \Lambda^\ast$.  Fix a positive Weyl chamber $\liet_+$, and let $\calR_+$ denote the positive roots.  This choice determines a complex structure on $\liet^\perp$ (hence also an orientation).

Let $M$ be a $G$-manifold.  Let $\calD^\prime(M)$ denote the space of distributions on $M$, that is, the dual of the space of compactly supported smooth functions $C^\infty_{\comp}(M)$.  A superscript $G$ will denote the $G$-invariant elements, for example, $C^\infty_{\comp}(M)^G$, $\calD^\prime(M)^G$, etc.  Given $X \in \lieg$, let $X_M$ denote the corresponding vector field on $M$.  The map $X \in \lieg \mapsto X_M$ is a Lie algebra homomorphism.  We will make extensive use of the Cartan model for equivariant cohomology, which we now briefly recall.  Let $\Omega_G(M)=(\Omega(M) \otimes \Pol(\lieg))^G$, where $\Pol(\lieg)=S\lieg^\ast$ denotes the polynomial algebra on $\lieg$ (with the coadjoint action of $G$).  Let $H_G(M)=\ker(d_G)/\im(d_G)$ where the equivariant differential $d_G$ is defined by
\[ (d_G\alpha)(X)=(d-\iota(X_M))\alpha(X). \]
For $\xi \in \lieg^\ast$, we write $\xi(\partial)$ for the directional derivative on $\lieg^\ast$ in the direction $\xi$.  This extends to a map $p \in \Pol(\lieg)=S\lieg^\ast \mapsto p(\partial)$ from polynomials on $\lieg$ to constant-coefficient differential operators on $\lieg^\ast$.  An equivariant differential form $\alpha \in \Omega_G(M)$ can be decomposed as a sum $\alpha=\sum_k \alpha_k p_k$ where $\alpha_k \in \Omega(M)$, $p_k \in \Pol(\lieg)$.  Let $\phi:M \rightarrow \lieg^\ast$ be a smooth map.  We define a linear map $C^\infty(\lieg^\ast) \rightarrow \Omega(M)$ by
\[ f \mapsto \alpha(-\partial) f \circ \phi:=\sum_k \alpha_k \phi^\ast(p_k(-\partial)f).\]
For a couple of arguments it will be convenient to use equivariant currents.  Let $\calC(M)$ denote the space of currents on $M$, that is, the dual to the space $\Omega_{\comp}(M)$ of compactly supported smooth differential forms.  The space of equivariant currents is $\calC_G(M)=(\calC(M) \otimes \Pol(\lieg))^G$.  Assuming $M$ is oriented, $\Omega(M)$ can be identified with a subspace of $\calC(M)$, and the above definitions extend naturally to $\calC_G(M)$.

\section{Harada-Karshon approach to norm-square localization}
In this section we describe a version of the Harada-Karshon (H-K) Theorem \cite{KarshonHarada}.  We restrict ourselves to the case that the localizing set $Z$ is a smooth submanifold.  This avoids some complications and has the appealing feature that the contribution from a component $Z_i \subset Z$ is especially simple to describe: it is the twisted DH distribution of any \emph{polarized completion} of a tubular neighbourhood of $Z_i$.  Only the smooth case is used in Section 3, as we eventually use a perturbation to ensure a smooth localizing set.  Some discussion of the situation when $Z$ is not smooth is included in Appendix C.  A small difference---needed later on in Section 4.1---between our setting and that in \cite{KarshonHarada}, is that we work with a moment map which might only be proper when restricted to the support of the equivariant cocycle $\alpha$ used to twist the Duistermaat-Heckman distribution.  This explains why conditions on the support of $\alpha$ appear in the statements.  For completeness, some proofs are outlined in Appendix B; the interested reader should also consult the original paper \cite{KarshonHarada} for a detailed treatment.

\subsection{Twisted Duistermaat-Heckman measures.}
Throughout this section, $(N,\omega,\phi)$ will be an \emph{oriented}, possibly degenerate, Hamiltonian $G$-space, that is, an oriented $G$-manifold $N$, equipped with a closed, equivariant 2-form
\[ \omega_G(X) = \omega - \pair{\phi}{X}, \hspace{1cm} X \in \lieg.\]
Let $\alpha \in \Omega_G(N)$ be a closed equivariant differential form, or more generally, a closed equivariant current.

We are interested in an invariant of the 4-tuple $(N,\omega,\phi,\alpha)$, the $\alpha$-\emph{twisted Duistermaat-Heckman distribution}, $\DH(N,\omega,\phi,\alpha) \in \mathcal{D}^\prime(\lieg^\ast)$.  These distributions were introduced by Jeffrey-Kirwan \cite{JeffreyKirwan}, and contain information about cohomology pairings on symplectic quotients.  In the case that $\alpha$ is compactly supported, $\DH(N,\omega,\phi,\alpha)$ can be defined in terms of its Fourier coefficients.  For $X \in \lieg$,
\begin{equation}
\label{FourierCoefficients}
\langle \DH(N,\omega,\phi,\alpha),e^{-2\pi \i \pair{\cdot}{X}} \rangle =
\int_N e^{\omega-2\pi \i \pair{\phi}{X}} \alpha(2\pi \i X).
\end{equation}
Note that the integrand is closed for the differential $d_{2\pi \i X}:=d-2\pi \i \iota(X_M)$.  More generally we define
\begin{equation}
\label{DefDH}
\langle \DH(N,\omega,\phi,\alpha), f \rangle = \int_N e^{\omega}\alpha(-\partial) f \circ \phi,
\end{equation}
for $f \in C^\infty_{\textnormal{comp}}(\lieg^\ast)$.  To ensure that the integral exists, we require that $\phi$ be \emph{proper on the support of} $\alpha$.  In the case that $\alpha=\sum_k \alpha_k p_k$ is an equivariant current, the above integral \eqref{DefDH} should be interpreted as a sum of pairings between currents $\alpha_k$ and forms $e^{\omega} \phi^\ast(p_k(-\partial) f)$.  The DH distribution $\DH(N,\omega,\phi,\alpha)$ can also be expressed in terms of push-forwards of currents:
\[ \DH(N,\omega,\phi,\alpha) = \sum_k p_k(\partial) \phi_\ast(e^\omega \alpha_k)^{[\textnormal{top}]}.\]

\begin{theorem}
\label{CohomologyClass}
Let $(N,\omega,\phi)$ be an oriented Hamiltonian $G$-space.  Let $\alpha_1,\alpha_2$ be closed equivariant currents, and suppose $\alpha_1-\alpha_2=d_G\beta$ for some equivariant current $\beta$.  Assume that $\phi$ is proper on the support of $\alpha_1,\alpha_2,\beta$.  Then
\[ \DH(N,\omega,\phi,\alpha_1)=\DH(N,\omega,\phi,\alpha_2).\]
\end{theorem}
\noindent H-K give a proof by cobordism (\cite{KarshonHarada}, Lemma 5.15). We include a proof based on methods in \cite{Meinrenken2005}.
\begin{proof}
We must show $\DH(N,\omega,\phi,d_G \beta)=0$.  If $\beta$ is not compactly supported, let $\{\rho_k \}$ be a smooth $G$-invariant (locally finite) partition of unity on $N$ with $\rho_k$ compactly supported.  For $f \in C^\infty_{\comp}(\lieg^\ast)$, $\supp(\beta)\cap \supp(\phi^\ast f)$ is compact, so intersects the supports of only finitely many of the functions $\rho_k$.  Therefore for all but finitely many $k$
\[ (d_G(\rho_k\beta))(-\partial)f\circ \phi = 0,\]
so that the sum can be brought outside the integral:
\[ \pair{\DH(N,\omega,\phi,d_G\beta)}{f}=
\sum_k \pair{\DH(N,\omega,\phi,d_G(\rho_k\beta))}{f}.\]
Since this holds for any $f \in C^\infty_{\comp}(\lieg^\ast)$ we conclude
\[ \DH(N,\omega,\phi,d_G\beta)=\sum_k \DH(N,\omega,\phi,d_G(\rho_k \beta)),\]
and each $\rho_k\beta$ is compactly supported.

So assume $\beta$ is compactly supported, in which case $\DH(N,\omega,\phi,d_G\beta)$ is a compactly supported distribution.  Let $X \in \lieg$.  Using the formula for the Fourier coefficients \eqref{FourierCoefficients}
\begin{align*}
\pair{\DH(N,\omega,\phi,d_G\beta)}{e^{-2\pi \i\pair{\cdot}{X}}}&=\int_N e^{\omega-2\pi \i\pair{\phi}{X}}d_{2\pi \i X}\beta(2\pi \i X)\\
&= \int_N d_{2\pi \i X}(\beta(2\pi \i X) e^{\omega-2\pi \i\pair{\phi}{X}})\\
&=0,
\end{align*}
where the last line follows by Stokes' theorem.
\end{proof}
\begin{remark}
It is not true in general that if $[\omega_1-\phi_1]=[\omega_2-\phi_2]$ in $H_G(N)$ then the corresponding Duistermaat-Heckman measures are equal (even for $N=\mathbb{R}^2$).  Theorem \ref{CohomologousTwoForms} gives a sufficient condition.
\end{remark}

\subsection{Taming maps and polarized completions.}
\begin{definition}
A \emph{taming map} $v:N\rightarrow \lieg$ is a smooth $G$-equivariant map.  A \emph{bounded} taming map is one for which the map $v$ is bounded.  A taming map $v$ defines a vector field $v_N$ on $N$ by
\[ v_N(p) := (v(p))_N(p), \hspace{1cm} p \in N. \]
The \emph{localizing set} or \emph{critical set} of $v$ is
\[ Z := \{p\in N | v_N(p)=0 \}. \]
We make the convention that the elements of $\phi(Z) \subset \lieg^\ast$ will be referred to as \emph{critical values}.
\end{definition}

An example of a taming map is obtained from a $G$-invariant inner product $\cdot$ on $\lieg$.  Using the inner product to identify $\lieg$ and $\lieg^\ast$, we can take $v$ to be the moment map itself.  The localizing set for $v=\phi$ is
\[ Z = G \cdot \bigcup_{\beta \in \liet^\ast_+} N^\beta \cap \phi^{-1}(\beta). \]
If $\omega$ is non-degenerate, the localizing set coincides with the set of critical points for the norm-square of the moment map \cite{Kirwan}.

An important reason for introducing taming maps is that they are a convenient way to prove that a map is \emph{proper}, as a result of the following simple Lemma.
\begin{lemma}[\cite{KarshonHarada}, Lemma 2.9]
\label{TameProper}
Let $v$ be a bounded taming map on a $G$-manifold $N$, and let $\phi:N \rightarrow \lieg^\ast$ be a continuous function.  Suppose $\pair{\phi}{v}$ is proper on a closed subset $S$.  Then $\phi$ is proper on $S$.
\end{lemma}

\begin{definition}
Let $\alpha$ be a closed equivariant differential form on a Hamiltonian $G$-space $(U,\omega,\phi)$.  Let $v$ be a taming map with smooth localizing set $\iota_Z:Z \hookrightarrow U$.  Suppose that $\pair{\phi}{v}$ is proper and bounded below on $Z \cap \supp(\alpha)$.  A $v$-\emph{polarized completion of} $(U,\omega,\phi,\alpha)$ is a Hamiltonian $G$-space $(U,\tomega,\tphi)$ with $\iota_Z^\ast (\omega-\phi)=\iota_Z^\ast (\tomega-\tphi)$, and such that $\pair{\tphi}{v}$ is proper and bounded below on the support of $\alpha$.
\end{definition}
\noindent In other words, $\tphi$ is an \emph{extension} of $\phi|_Z$ such that $\pair{\tphi}{v}$ is proper and bounded below on the support of $\alpha$.
\begin{definition}
\label{DefinitionSmoothCollapse}
Let $N$ be a $G$-manifold, $Z$ a $G$-invariant submanifold, and $U$ a $G$-invariant open set in $N$.  We say that $U$ can be \emph{smoothly collapsed to part of} $Z$ if there is a smooth $G$-invariant submanifold $Z^\prime \subset Z \cap U$ and smooth $G$-equivariant map $p:[0,1]\times U \rightarrow U$ such that
\begin{enumerate}
\item $p_0$ is the identity map on $U$,
\item $p_1(U) \subset Z^\prime$,
\item $p_t(Z^\prime) \subset Z^\prime$ for all $t \in [0,1]$.
\end{enumerate}
\end{definition}
\noindent Mainly we will use $G$-invariant tubular neighbourhoods $U \supset Z$, which are clearly examples of the definition above.  But for an argument in Section 4.4 it is convenient to have this slightly more flexible definition.
\begin{remark}
\label{SmoothCollapse}
The definition says that the identity map $\Id_U$ and the inclusion $\iota_{Z^\prime}$ are connected by a smooth, $G$-equivariant homotopy $p$.  In particular, if $\alpha_0$, $\alpha_1$ are closed equivariant differential forms on $U$ whose pullbacks to $Z^\prime$ are equal, then $\alpha_0$, $\alpha_1$ are cohomologous.  H-K \cite{KarshonHarada} work with a similar condition.
\end{remark}

\noindent The next result is the main result on \emph{existence} of polarized completions, and the \emph{uniqueness} (under conditions) of the resulting DH distribution.  A proof is included in Appendix B.

\begin{lemma}[\cite{KarshonHarada} Proposition 3.4, Lemmas 4.12, 4.17]
\label{Hypotheses}
Let $\alpha$ be a closed equivariant differential form on an oriented Hamiltonian $G$-space $(N,\omega,\phi)$.  Let $v:N \rightarrow \lieg$ be a bounded taming map with smooth localizing set $Z$.  Let $U \supset Z \cap \supp(\alpha)$ be an open subset that can be smoothly collapsed to part of $Z$.  Suppose that $\pair{\phi}{v}$ is proper and bounded below on $Z \cap \supp(\alpha)$.  Then
\begin{enumerate}
\item There exists a $v$-polarized completion $(U,\tomega,\tphi)$ of the restriction $(U,\omega|_U,\phi|_U,\alpha|_U)$.
\item Suppose $v_i, Z_i, (U_i,\tomega_i,\tphi_i)$, $i=0,1$ are two sets of data satisfying the above conditions, and such that $v_0$, $v_1$ agree on the support of $\alpha$.  Then
\[ \DH(U_1,\tomega_1,\tphi_1,\alpha)=\DH(U_0,\tomega_0,\tphi_0,\alpha).\]
\end{enumerate}
\end{lemma}
\noindent Note that Lemma \ref{TameProper} implies that $\tphi$ is proper on the support of $\alpha|_{U}$, hence $\DH(U,\tomega,\tphi,\alpha)$ is defined.

\subsection{The Harada-Karshon Theorem.}
\begin{definition}
\label{DHGerms}
Following \cite{KarshonHarada}, in the setting of Lemma \ref{Hypotheses}, we define
\begin{equation}
\label{DefOfGermDH}
\DH^v_Z(N,\omega,\phi,\alpha) = \DH(U,\tomega,\tphi,\alpha).
\end{equation}
If $Z=\sqcup Z_i$ and $U=\sqcup U_i$, where $U_i$ is an open $G$-invariant subset that can be smoothly collapsed to part of $Z_i$, define
\[\DH^{v}_{Z_i}(N,\omega,\phi,\alpha)=\DH(U_i,\tomega,\tphi,\alpha).\]
We have
\[ \DH^{v}_Z(N,\omega,\phi,\alpha) = \sum_i \DH^{v}_{Z_i}(N,\omega,\phi,\alpha).\]
\end{definition}
The right-hand-side of \eqref{DefOfGermDH} is stable under various changes (which justifies the notation): according to Lemma \ref{Hypotheses}, it is independent of the choice of open $U$ and the choice of $v$-polarized completion.  There is flexibility to modify the taming map $v$ away from the support of $\alpha$.  Also, if the properness condition in Theorem \ref{CohomologyClass} holds, then $\alpha$ can be replaced by a cohomologous form $\alpha^\prime$; one simple case in which this properness condition is immediate is if $\supp(\alpha^\prime) \subset \supp(\alpha)$.

We can now state the main theorem of \cite{KarshonHarada}.  A proof is included in Appendix B.
\begin{theorem}[\cite{KarshonHarada}, Theorem 5.20]
\label{HKTheorem}
Consider the setting of Lemma \ref{Hypotheses}.  Suppose further that $\pair{\phi}{v}$ is proper and bounded below on the support of $\alpha$.  Then
\begin{equation}
\label{HKFormula}
\DH(N,\omega,\phi,\alpha) = \DH^{v}_Z(N,\omega,\phi,\alpha). 
\end{equation}
\end{theorem}
\begin{remark}
This result is related to earlier work on the subject, in particular the work of Paradan \cite{Paradan97,Paradan98} and Woodward \cite{Woodward}.  As examples, one might look to Proposition 3.3 in \cite{Paradan97} (a general localization result) and Proposition 2.9 in \cite{Paradan98} (a formula for the contribution of a smooth component of the localizing set, in terms of a twisted DH distribution for a tubular neighbourhood $U \supset Z$).  Theorem 5.1 in \cite{Woodward} is a norm-square localization formula, expressed in terms of twisted DH distributions of small submanifolds around the critical set.
\end{remark}

The condition that $v$ be bounded can be relaxed.  If $v$ is unbounded, but is bounded on a neighbourhood $\iota_i:N_i \hookrightarrow N$ of each component $Z_i$ of $Z$, then we can extend the definition:
\[ \DH^{v}_{Z_i}(N,\omega,\phi,\alpha):=\DH^{v}_{Z_i}(N_i,\iota_i^\ast \omega, \iota_i^\ast \phi, \iota_i^\ast \alpha).\]
The flexibility noted in Definition \ref{DHGerms} holds in this setting.  It is also often possible to apply Theorem \ref{HKTheorem} (notably for the case $v=\phi$, which can be unbounded):
\begin{corollary}
\label{UnboundedTaming}
Consider the setting of Theorem \ref{HKTheorem}, except that $v$ may be unbounded.  Suppose however that the subset $\|v\|^2(Z) \subset \mathbb{R}$ is discrete, and set $Z_r=Z \cap (\|v\|^2)^{-1}(r)$.  Then
\begin{equation} 
\label{CorResult}
\DH(N,\omega,\phi,\alpha)=\sum_r \DH^{v}_{Z_r}(N,\omega,\phi,\alpha).
\end{equation}
\end{corollary}
\begin{proof}
Since $\|v\|^2(Z) \subset \mathbb{R}$ is discrete, it is possible to find a smooth, positive, non-increasing function $f \in C^\infty(\mathbb{R}_{\ge 0})$, constant near each $r \in \mathbb{R}_{\ge 0}$ such that $Z_r \ne \emptyset$, and vanishing sufficiently rapidly at infinity, such that the modified taming map $v^\prime=f(\|v\|^2)v$ is bounded.  The localizing sets for $v$, $v^\prime$ are the same, and Theorem \ref{HKTheorem} can be applied, which gives the right-hand-side of \eqref{CorResult} except for $v^\prime$.  Since $f(\|v\|^2)$ is constant near each $Z_r \ne \emptyset$, the condition that $(U_r,\tomega_r,\tphi_r)$ is a $v^\prime$-polarized completion is equivalent to it being a $v$-polarized completion, since we can take $U_r$ sufficiently small that $f(\|v\|^2)$ is constant on $U_r$.  Thus $\DH^{v^\prime}_{Z_r}(N,\omega,\phi,\alpha)=\DH^v_{Z_r}(N,\omega,\phi,\alpha)$.
\end{proof}

\noindent The following Lemma will be useful in Section 4.  It allows a DH distribution for a Hamiltonian $G$-space $N$, twisted by a Poincare dual form $\tau_S$ to a submanifold $S \subset N$, to be reduced to a DH distribution for the submanifold $S$.
\begin{lemma}
\label{DeltaCurrent}
Let $(N,\omega,\phi)$ be an oriented Hamiltonian $G$-space, and $\alpha$ a closed equivariant differential form.  Let $\iota:S \hookrightarrow N$ be an oriented $G$-invariant submanifold (without boundary), which is closed as a subset of $N$.  Let $\tau_S$ be a $G$-equivariant Thom form for the normal bundle $\nu(S,N) \rightarrow S$, which we view as a closed equivariant differential form on $N$ via a tubular neighbourhood embedding.  Let $v:N \rightarrow \lieg$ be a bounded taming map with smooth localizing set $Z$, such that $Z_S:=S \cap Z$ is also smooth.  Suppose $\pair{\phi}{v}$ is proper and bounded below on $Z \cap \supp(\tau_S \alpha)$.  Then
\[ \DH^v_Z(N,\omega,\phi,\tau_S \alpha)=\DH^{v}_{Z_S}(S,\iota^\ast \omega,\iota^\ast \phi,\iota^\ast \alpha).\]
\end{lemma}
\begin{proof}
Let $U_S$ be a small $G$-invariant tubular neighbourhood of $Z_S$ in $S$, and let $U$ be a small $G$-invariant tubular neighbourhood of $U_S$ in $N$.  Since we can replace $\tau_S$ with a cohomologous form having support contained in an arbitrarily small neighbourhood of $S$, we can assume $U$ contains $Z \cap \supp(\tau_S \alpha)$.  Moreover, $U$ can be smoothly collapsed to part of $Z$ (first collapse $U$ to $U_S$, then collapse the latter to $Z_S$).  Let $(U,\tomega,\tphi)$ be a $v$-polarized completion of $(U,\omega,\phi,\tau_S \alpha)$.  By definition,
\[ \DH^v_Z(N,\omega,\phi,\tau_S \alpha)=\DH(U,\tomega,\tphi,\tau_S \alpha).\]
Let $\delta_S$ be the closed equivariant current defined by $S$,
\[ \pair{\delta_S}{\eta}=\int_S \iota^\ast \eta, \hspace{1cm} \eta \in \Omega_\comp(N).\]
Then $\delta_S$, $\tau_S$ are cohomologous in $\calC_G(N)$.  Since $v$ is bounded, $\tphi$ is proper on $\supp(\tau_S \alpha)$ (Lemma \ref{TameProper}), hence also on $\supp(\delta_S \alpha)$.  Applying Theorem \ref{CohomologyClass},
\[ \DH^v_Z(N,\omega,\phi,\tau_S \alpha) = \DH(U,\tomega,\tphi,\delta_S \alpha).\]
Because of the delta-like current, the right-hand-side can be viewed as a twisted-DH distribution for $S \cap U=U_S$.  Thus
\[ \DH^v_Z(N,\omega,\phi,\tau_S \alpha)=\DH(U_S,\iota^\ast \tomega, \iota^\ast \tphi,\iota^\ast \alpha).\]
The right-hand-side of this equation is $\DH^{v}_{Z_S}(S,\iota^\ast \omega,\iota^\ast \phi,\iota^\ast \alpha)$.
\end{proof}

\section{Hamiltonian $LG$-spaces and q-Hamiltonian $G$-spaces}
Let $LG$ denote the Banach Lie group consisting of maps $S^1 \rightarrow G$ of a fixed Sobolev level $s \ge 1$, with group operation given by pointwise multiplication, and let $L \lieg=\Omega^0_s(S^1,\lieg)$ denote its Lie algebra.  We define $L \lieg^\ast$ to be the space $\Omega^1_{s-1}(S^1, \lieg)$ of $\lieg$-valued 1-forms of Sobolev level $s-1$, where the pairing with $L \lieg$ is defined using the inner product on $\lieg$, and integration over the circle.  Identifying $L\lieg^\ast$ with the space of connections on the trivial $G$-bundle over the circle, we have the affine-linear action of $LG$ by gauge transformations
\[ g \cdot \xi = \Ad_g \xi - dg g^{-1}.\]
\begin{definition}
A \emph{Hamiltonian} $LG$-\emph{space} $(\calM,\omega,\Psi)$ is a Banach manifold $\calM$ with a smooth $LG$-action, equipped with a weakly non-degenerate, $LG$-invariant 2-form $\omega$, and an $LG$-equivariant \emph{moment map} $\Psi:\calM \rightarrow L\lieg^\ast$ satisfying
\[ \iota(\xi_\calM)\omega = -d\pair{\Psi}{\xi}, \hspace{1cm} \xi \in L\lieg.\]
\end{definition}
\noindent Let $L_0G \subset LG$ denote the subgroup consisting of loops based at the identity $e \in G$.  This subgroup acts freely on $L\lieg^\ast$.  Hence, by equivariance of $\Psi$, $L_0G$ also acts freely on $\calM$.

\subsection{The norm-square of the moment map.}
For $\xi \in L\lieg^\ast$, let
\[ \|\xi\|^2 = \frac{1}{2\pi }\int_0^{2\pi} |\xi(\theta)|^2 d\theta,\]
where $|\xi(\theta)|$ denotes the norm on $\lieg$.  Given a Hamiltonian $LG$-space $(\calM,\omega,\Psi)$, the \emph{norm-square of the moment map} is the $G$-invariant function
\[ \|\Psi\|^2:\calM \rightarrow \mathbb{R}.\]
The following result gives a description of the critical set of $\|\Psi\|^2$.
\begin{proposition}[cf. \cite{BottTolmanWeitsman}]
\label{NormSquareSet}
Let $\Psi: \calM \rightarrow L\lieg^\ast$ be a proper Hamiltonian $LG$-space.  We have a decomposition
\begin{equation}
\label{DecompositionOfCrit}
\textnormal{Crit}(\|\Psi\|^2) = G \cdot \bigcup_{\beta \in \mathcal{B}} \calM^\beta \cap \Psi^{-1}(\beta)
\end{equation}
where $\calB = \{\beta \in \liet^\ast_+|\calM^\beta \cap \Psi^{-1}(\beta) \ne \emptyset \} \subset \liet^\ast_+$ is discrete.
\end{proposition}
We make some brief remarks to motivate this result.  If $m \in \calM$ is a critical point of $\|\Psi\|^2$, then, by equivariance of $\Psi$, $\xi:=\Psi(m)$ must be a critical point of the restriction of $\|\cdot\|^2$ to the $L_0G$ orbit $L_0G \cdot \xi \subset L\lieg^\ast$.  Let $g \in G$ be the holonomy of the connection $\xi$, and let $P_{e,g}G$ denote the space of paths $\gamma:[0,1] \rightarrow G$ of Sobolev class $s$ having fixed endpoints $\gamma(0)=e$, $\gamma(1)=g$.  $L_0G$ acts on $P_{e,g}G$ by left multiplication.  For $\gamma \in P_{e,g}G$, let $E(\gamma)$ denote the \emph{energy} of the path:
\[ E(\gamma)=\int_0^{1} |\gamma(t)^{-1}\gamma^\prime(t)|^2 dt.\]
There is an $L_0G$-equivariant map
\[ \Hol:L_0G \cdot \xi \rightarrow P_{e,g}G,\]
the holonomy of the connection (using the trivialization $S^1 \times G$ of the bundle), and under this map
\[ \|\zeta\|^2 = E(\Hol(\zeta)), \hspace{1cm} \zeta \in L_0G \cdot \xi \subset L\lieg^\ast.\]
The critical points of the energy functional on $P_{e,g}G$ are smooth geodesics (for the bi-invariant metric) from $e$ to $g$.  It follows that $\Hol(\xi)=\exp(t \xi)$ with $\xi \in \lieg \subset L\lieg^\ast$ a constant connection.

We have
\[ 0=d_m\|\Psi\|^2=2\pair{d_m\Psi}{\xi}=-2\iota(\xi_\calM)\omega_m.\]
By weak non-degeneracy of $\omega$, $\xi_{\calM}(m)=0$, which shows that $m \in \calM^\xi \cap \Psi^{-1}(\xi)$.  Since $\|\Psi\|^2$ is $G$-invariant, the decomposition \eqref{DecompositionOfCrit} in terms of a subset $\calB \subset \liet^\ast_+$ follows.  Using properness of $\Psi$, it can be shown that $\calB$ is discrete.

\subsection{Q-Hamiltonian $G$-spaces and the Equivalence Theorem.}
The left (resp. right) invariant Maurer-Cartan forms on $G$ will be denoted $\theta^L$ (resp. $\theta^R$).  The Cartan 3-form $\eta$ is a bi-invariant closed 3-form on $G$ given by
\[ \eta=\tfrac{1}{12}\theta^L \cdot [\theta^L,\theta^L].\]
It has an equivariant extension (for $G$ acting on itself by conjugation)
\[ \eta_G(X)=\eta-\tfrac{1}{2}(\theta^L+\theta^R)\cdot X, \hspace{1cm} X \in \lieg.\]

\begin{definition}[\cite{AlekseevMalkinMeinrenken}]
A \emph{quasi-Hamiltonian} (q-Hamiltonian) $G$ \emph{space} is a $G$-manifold $M$, together with a $G$-invariant 2-form $\omega$ and a smooth $G$-equivariant map $\Phi:M \rightarrow G$ satisfying
\begin{equation} 
\label{DifferentialOfOmega}
d_G \omega = -\Phi^\ast \eta_G,
\end{equation}
and such that $\ker(\omega_m) \cap \ker(d_m\Phi) = \{0\}$ for all $m \in M$.  If this last \emph{minimal degeneracy} condition is not satisfied, then we refer to $M$ as a \emph{degenerate} q-Hamiltonian $G$-space.
\end{definition}
\begin{remark}
If $\rho=\tfrac{1}{2} \sum_{\alpha \in \calR_+} \alpha$, the half-sum of the positive roots, is a weight of the group $G$, then $M$ is automatically even-dimensional and orientable \cite{AMWDuistermaatHeckman} (for example, this happens if $G$ is the product of a simply connected group and a torus).  In any case, from now on we take $M$ to be a connected, oriented, even-dimensional q-Hamiltonian $G$-space.
\end{remark}
According to the Equivalence Theorem 8.3 in \cite{AlekseevMalkinMeinrenken}, there is a 1-1 correspondence between compact q-Hamiltonian $G$-spaces and proper Hamiltonian $LG$-spaces.  Let $\Psi:\calM \rightarrow L\lieg^\ast$ be a proper Hamiltonian $LG$-space with 2-form $\omega_\calM$.  The corresponding q-Hamiltonian space $\Phi:M \rightarrow G$ fits into a pullback diagram
\begin{equation}
\label{PullbackDiagram}
\begin{CD}
\calM @> \Psi >> L\lieg^\ast \\
@VV \Hol V		@VV \Hol V\\
M @> \Phi >> G
\end{CD}
\end{equation}
The vertical maps are quotients by the free action of the group of based loops $L_0G$.  There is an $LG$-invariant 2-form $\varpi \in \Omega^2(L\lieg^\ast)$ such that
\begin{equation}
\label{OmegaDescends}
\omega :=\omega_\calM + \Psi^\ast \varpi,
\end{equation}
is $L_0G$-basic, hence descends to the quotient $M$.  It satisfies
\[ d_G \omega = -\Phi^\ast \eta_G, \]
thus $(M,\omega,\Phi)$ is a q-Hamiltonian $G$-space.

For later use, define the (possibly singular) $T$-invariant closed subsets
\[ X=\Phi^{-1}(T) \subset M, \hspace{1cm} \hX =\Psi^{-1}(\liet) \subset \calM.\]
Restricting the pullback diagram \eqref{PullbackDiagram} shows that $\hX$ is the fibre product of $X$ with $\liet \subset L\lieg^\ast$.

\subsection{Cross-section Theorem.}
Let $g \in T$, and let $G_g$ be the centralizer of $g$, with Lie algebra $\lieg_g$.  Since $G_g \supset T$, $\lieg_g$ has a unique $G_g$-invariant complement $\lieg_g^\perp$.  There exists a \emph{slice} for the conjugation action near $g$, that is, an $\Ad(G_g)$ invariant open neighbourhood $U_g$ of $g$ in $G_g$ such that there is a $G$-equivariant diffeomorphism onto an open subset of $G$:
\[ G \times_{G_g} U_g \xrightarrow{\sim} \Ad(G) \cdot U_g, \hspace{1cm} [h,x] \mapsto hxh^{-1}.\]

Let $\Phi:M\rightarrow G$ be a q-Hamiltonian $G$-space, and let
\[ Y_g = \Phi^{-1}(U_g).\]
This is a smooth $G_g$-invariant submanifold.  Moreover, $G \cdot Y_g$ is an open subset of $M$ and
\begin{equation}
\label{Flowout}
G \cdot Y_g \simeq G\times_{G_g} Y_g,
\end{equation}
as $G$-spaces.  By Proposition 7.1 of \cite{AlekseevMalkinMeinrenken}, $Y_g$ is a q-Hamiltonian $G_g$-space, with the restriction of the 2-form and moment map.

In fact, the cross-section $Y_g$ can be made into a \emph{Hamiltonian} $G_g$-space.  If the chosen open neighbourhood $U_g$ is sufficiently small, then the image of the map $g^{-1}\Phi|_{Y_g}$ is contained in a neighbourhood of $e \in G$ on which the exponential map has a smooth inverse $\log$ sending $e$ to $0 \in \lieg$.  Let
\[ \phi_g := \log(g^{-1}\Phi|_{Y_g}).\]
The pullback of the Cartan 3-form $\eta$ to $U_g$ is exact.  Using a homotopy operator, one constructs a $G_g$-invariant 2-form $\varpi$ on $\lieg_g$ such that on $Y_g$ the 2-form
\[ \omega_g = \omega-\phi_g^\ast \varpi.\]
satisfies $d\omega_g=0$.  See \cite{AlekseevMalkinMeinrenken} for further details.  It follows from the construction and the fact that $\iota_T^\ast \eta=0$, that
\begin{equation}
\label{PullbackOfVarpi}
\iota_\liet^\ast \varpi =0.
\end{equation}
\begin{theorem}[\cite{AlekseevMalkinMeinrenken}]
$(Y_g,\omega_g,\phi_g)$ is a non-degenerate Hamiltonian $G_g$-space.
\end{theorem}

As a result of equation \eqref{Flowout}, the normal bundle
\[ \nu(Y_g,M) \simeq Y_g \times \lieg_g^\perp.\]
We orient the even-dimensional subspace $\lieg_g^\perp$ compatibly with the orientation provided by the Liouville form on the symplectic submanifold $Y_g$, and the orientation on $M$.  Note in particular that when $\lieg_g^\perp = \{0 \}$, an orientation is just a sign $\pm 1$.  Let $\lieg_g/\liet$ denote the unique $T$-invariant complement of $\liet$ in $\lieg_g$.\footnote{We avoid using $\liet^\perp$ since we use this to denote the $T$-invariant complement in $\lieg$.}  Hence
\[ \liet^\perp = (\lieg_g/\liet) \oplus \lieg_g^\perp.\]
Orient $\lieg_g/\liet$ compatibly with the orientation on $\lieg_g^\perp$ and the orientation on $\liet^\perp$ determined by the positive roots.  Let $\Eul(\lieg_g/\liet,X) \in \Pol(\liet)$ denote the $T$-equivariant Euler class for the oriented $T$-equivariant vector bundle $(\lieg_g/\liet) \times \liet \rightarrow \liet$.  It is given by
\[ \Eul(\lieg_g/\liet,X) = \sgn(g) (-1)^{|\calR_{g,+}|}\prod_{\alpha \in \calR_{g,+}} \pair{\alpha}{X},\]
where $\calR_{g,+} \subset \calR_+$ is a set of positive roots of $\lieg_g$, and $\sgn(g)=\pm 1$ is a sign depending on $g$ (determined from the orientations).

\subsection{Abelianization.}
This subsection describes the `abelianization' of a q-Hamiltonian $G$-space $\Phi:M \rightarrow G$, a construction introduced in \cite{Meinrenken2005} Section 3.2.  This is a rough analog, for q-Hamiltonian spaces, of converting a Hamiltonian $G$-space into a Hamiltonian $T$-space by \emph{projecting} the moment map to $\liet^\ast$.  Consider the $\Ad(T)$-equivariant smooth map $r:T \times \liet^\perp \rightarrow G$
\[ r(t,X)=t \exp(X).\]
For a sufficiently small ball around the origin $B_\epsilon(\liet^\perp)$, the map $r$ restricts to a diffeomorphism from $T \times B_{\epsilon}(\liet^\perp)$ onto a tubular neighbourhood $\calT$ of $T$ in $G$.  Let $\pi_T:\calT \simeq T \times B_\epsilon(\liet^\perp) \rightarrow T$ denote the projection onto the first factor.  Let $\calX=\Phi^{-1}(\calT) \subset M$, and define a smooth map $\Phi_{\a}:\calX \rightarrow T$ by
\[ \Phi_{\a}=\pi_T \circ \Phi. \]
The pullback of the Cartan 3-form $\eta$ to $\calT$ is exact (its pullback to $T$ vanishes).  Using a homotopy operator, one constructs a $T$-invariant 2-form $\gamma_{\a}$ such that on $\calX:=\Phi^{-1}(\calT) \subset M$, the 2-form
\[\omega_{\a}=\omega-\Phi^\ast \gamma_{\a}, \]
satisfies $d\omega_{\a}=0$.  See \cite{Meinrenken2005} for futher details.  It follow from the construction and the fact that $\iota_T^\ast \eta=0$ that
\begin{equation}
\label{PullbackOfGamma}
\iota_T^\ast \gamma_{\a} =0.
\end{equation}
\begin{theorem}[\cite{Meinrenken2005} Proposition 3.4]
$(\calX,\omega_{\a},\Phi_{\a})$ is a (degenerate) q-Hamiltonian $T$-space, called the \emph{abelianization} of $M$.
\end{theorem}

Let $\hcalX$ be the fibre product $\calX \times_T \liet$, and $\hPhi_{\a}:\hcalX \rightarrow \liet$ the corresponding map:
\begin{equation}
\label{SecondPullbackDiagram}
\begin{CD}
\hcalX @> \hPhi_{\a} >> \liet \\
@VV \exp V		@VV \exp V\\
\calX @> \Phi_{\a} >> T
\end{CD}
\end{equation}
The covering map $\exp:\hcalX \rightarrow \calX$ is the quotient by the action of the integral lattice $\Lambda$.  Pulling $\omega_{\a}$ back to $\hcalX$, we obtain a (degenerate) Hamiltonian $T$-space $(\hcalX,\exp^\ast \omega_{\a},\hPhi_{\a})$.  Since $\calX$ is an open subset of $M$, $\calX$ and $\hcalX$ both inherit an orientation.  Comparing to pullback diagram \eqref{PullbackDiagram} shows that the subset $\hX=\Psi^{-1}(\liet) \subset \calM$ of the corresponding Hamiltonian $LG$-space, can be identified with the subset $\exp^{-1}(X) \subset \hcalX$.

\subsection{Twisted DH distributions for a q-Hamiltonian space.}
The twisted DH distributions that we study in this paper are not the same as those defined in \cite{AMWDuistermaatHeckman} (which are conjugation-invariant distributions on the group $G$), but are closely related to them by an induction-type map.  We briefly indicate the relation in remark \ref{RelationToGEquivariant} below.  The type of DH distribution considered here was first introduced for q-Hamiltonian spaces in \cite{Meinrenken2005}, where further results about DH distributions of this kind can be found.  The idea of using a smooth cut-off form, playing the role of a Poincare dual to a possibly singular subset, appeared in a paper of Jeffrey-Kirwan (see section 5 in \cite{JeffreyKirwanModuli}).

To a $G$-equivariant cocycle $\alpha \in \Omega_G(M)$, we will associate a $\Lambda$-periodic, $W$-anti-symmetric distribution $\frakm^\alpha$ on $\liet \simeq \liet^\ast$.  The distribution $\frakm^\alpha$ is intended to represent the $\alpha$-twisted DH distribution of the $T$-space $\hX=\Psi^{-1}(\liet) \subset \calM$, the only problem being that the latter might not be smooth.  To deal with this, we work instead with the slightly larger space $\hcalX \supset \hX$ (see definition in previous subsection), together with a cut-off form $\tau_{\hcalX}$ supported near $\hX$ that plays the role of a Poincare dual to $\hX$.

Let $\tau$ denote a compactly supported $T$-equivariant Thom form for the vector bundle $\pi_T:\calT \simeq T \times \liet^\perp \rightarrow T$, and let $\tau_{\hcalX}=\exp^\ast \Phi^\ast \tau$.  The twisted DH distribution that we will study for the remainder of the article is
\[ \frakm^\alpha :=\DH(\hcalX,\exp^\ast \omega_{\a},\hPhi_{\a},\tau_{\hcalX} \cdot \exp^\ast \alpha).\]
(In the future, pullbacks by $\exp$ will be omitted from the notation.)  This is well-defined since $\hPhi_{\a}$ is proper on the support of $\tau_{\hcalX}$.  It is independent of the choice of tubular neighbourhood $\calT \supset T$ and of the choice of $\tau$ by Theorem \ref{CohomologyClass}.  $\frakm^\alpha$ can equivalently be thought of as a $W$-anti-symmetric distribution on $T$.

If $\Psi$ is transverse to $\liet$ so that $\hX$ is smooth, then $\tau_{\hcalX}$ is Poincare dual to $\hX$, and it follows (similar to Lemma \ref{DeltaCurrent})  that $\frakm^{\alpha}=\DH(\hX,\exp^\ast \omega_{\a},\hPhi_{\a},\exp^\ast \alpha)$, where the orientation of $\hX$ is determined by the orientations of $\hcalX$ and $\liet^\perp$.

As already mentioned, the main reason for introducing the twisted DH distributions $\frakm^\alpha$ is that they encode certain cohomology pairings on reduced spaces---see \cite{Meinrenken2005} for the exact relation in the case of $\frakm^\alpha$.  As an example, consider the distribution $\frakm$ ($\alpha=1$), and assume $\Phi$ has a non-trivial regular value.  In this case $\frakm = f dm$ for a $\Lambda$-periodic, $W$-anti-symmetric function $f$, which is polynomial on the chambers of an affine hyperplane arrangement in $\liet$ (and $dm$=Lebesgue measure).  If $\xi \in \liet$ is a regular value of $\hPhi_{\a}$ and $g=\exp(\xi)$ is a regular element of $T$ (i.e. $G_g=T$), then $f(\xi)$ is the volume of the reduced space $\Phi^{-1}(\exp(\xi))/T$.

\vspace{0.3cm}

\begin{remark}
\label{RelationToGEquivariant}
We briefly indicate the relation of $\frakm^\alpha$ to the more usual $G$-invariant twisted Duistermaat-Heckman distributions; see \cite{Meinrenken2005} sections 4, 5 for a more detailed discussion.  Consider first the case of a $G$-equivariant cocycle $\alpha$ on a proper \emph{Hamiltonian} $G$-space $(N,\omega,\phi)$.  Let $\phi_{\liet}$ and $\phi_{\liet^\perp}$ be the components of $\phi$ with respect to the orthogonal direct sum $\lieg=\liet \oplus \liet^\perp$.  Let $U \subset N$ be the inverse image under $\phi_{\liet^\perp}$ of a small ball $B$ around the origin in $\liet^\perp$, and let $\tau_U$ be the pullback to $U$ of a $T$-equivariant Thom form with support contained in $B$.  We have the following relation between twisted Duistermaat-Heckman distributions
\begin{equation}
\label{RliegEquation}
\DH(U,\omega,\phi_{\liet},\tau_U \cdot \alpha)=R_{\lieg}(\DH(N,\omega,\phi,\alpha)),
\end{equation}
where $R_{\lieg}$ is an \emph{isomorphism}
\[ R_{\lieg}:\calD^\prime(\lieg^\ast)^G \rightarrow \calD^\prime(\liet^\ast)^{W-\anti}.\]
Up to a constant normalization, this map is the inverse of the dual map to the isomorphism
\[ f \in C_\comp^\infty(\lieg^\ast)^G \mapsto f|_{\liet^\ast} \cdot \prod_{\alpha>0} \pair{\alpha}{-} \in C_\comp^\infty(\liet^\ast)^{W-\anti}.\]
(Any $W$-anti-symmetric smooth function can be divided by $\prod_{\alpha>0} \pair{\alpha}{-}$, and the result extends to a smooth $G$-invariant function.)  The distribution $\DH(U,\omega,\phi,\tau_U \cdot \alpha)$ therefore carries the same information.  In the case of the Hamiltonian $T$-space $\hcalX$ above, we are in a situation in which the left side of \eqref{RliegEquation} is defined, but the right side is not, as the $G$-space $N$ does not exist.

The relation between $\frakm^\alpha$, viewed as a $W$-anti-symmetric distribution on $T$, and the conjugation-invariant distributions studied in \cite{AMWDuistermaatHeckman} is similar.  Assuming (for simplicity) that the half-sum of the positive roots $\rho$ is a weight of $G$, the relation is given by an isomorphism:
\[ R_G:\calD^\prime(G)^{G} \rightarrow \calD^\prime(T)^{W-\anti}.\]
Up to a constant normalization, this map is the inverse of the dual map to the isomorphism
\[ f \in C^\infty(G)^G \mapsto \big(f\big|_T \big) \cdot \sum_{w \in W} (-1)^{l(w)}e_{w\rho} \in C^\infty(T)^{W-\anti},\]
given by restriction to $T$, followed by multiplication by the Weyl denominator.
\end{remark}

\section{Norm-square localization formula}
In this section we apply the Harada-Karshon Theorem to $(\hcalX,\omega_{\a},\hPhi_{\a})$ to obtain a norm-square localization result for Hamiltonian $LG$-spaces (or equivalently, for q-Hamiltonian $G$-spaces).  In the later subsections, local models for the polarized completions are described, and these are used to derive more explicit formulas for the contributions.  Through most of this section, we identify $\liet$ and $\liet^\ast$ using an invariant inner product on $\lieg$.  We use notation introduced in Section 3 for the q-Hamiltonian space $M$, its abelianization $\calX$, the Hamiltonian covering space $(\hcalX,\omega_{\a},\hPhi_{\a})$, and so on.

\subsection{Application of the Harada-Karshon theorem.}
The map $\hPhi_{\a}: \hcalX \rightarrow \liet^\ast$ introduced in the previous section is $T$-equivariant, hence defines a taming map.  However, the localizing set for this taming map need not be smooth.  To get around this problem, we perturb the taming map by a small `generic' vector $\gamma \in \liet$, letting
\[ v=\hPhi_{\a} - \gamma.\]
Then $\pair{\hPhi_{\a}}{v}=\|\hPhi_{\a}\|^2-\pair{\hPhi_{\a}}{\gamma}$ is proper and bounded below on the support of $\tau_{\hcalX}$ (its behavior at infinity is dominated by the $\|\hPhi_{\a}\|^2$ term).  We will explain what we mean by `generic' in the next section, but it will imply that the localizing set
\[ \calZ = \{v_{\hcalX}=0 \} \subset \hcalX\] 
has some desirable properties, and in particular that it is smooth.

We have shown that the 4-tuple $(\hcalX,\omega_{\a},\hPhi_{\a},\tau_{\hcalX}\alpha)$ equipped with the taming map $v$ satisfy all the conditions required in the Harada-Karshon Theorem \ref{HKTheorem} (see also Corollary \ref{UnboundedTaming}), which gives the following `norm-square localization formula'.
\begin{theorem}
\label{HKNormSquareFormula}
Let $\Phi:M\rightarrow G$ be a q-Hamiltonian $G$-space, and let $\hPhi_{\a}:\hcalX \rightarrow \liet$ be the covering space of its abelianization, as above.  Let $v=\hPhi_{\a}-\gamma$ for a generic vector $\gamma$.  Let $\tau_{\hcalX}$ be the pullback of the Thom form.  Then
\begin{equation} 
\label{BasicNormSquareFormula}
\DH(\hcalX,\omega_{\a},\hPhi_{\a},\tau_{\hcalX}\alpha)=\DH^v_{\calZ}(\hcalX,\omega_{\a},\hPhi_{\a},\tau_{\hcalX}\alpha).
\end{equation}
\end{theorem}

It will take some work to extract a more explicit description of the right-hand-side of this equation.  As a first step, note that
\[ \calZ= \bigcup_{\beta \in \liet} \hcalX^{\bbeta} \cap \hPhi_{\a}^{-1}(\beta),\]
where
\[ \bbeta := \beta-\gamma.\]
We will refer to the set
\[ \calZ_\beta=\hcalX^{\bbeta} \cap \hPhi_{\a}^{-1}(\beta)\]
as the $\beta$-\emph{component} of the localizing set (it might not be connected).

\subsection{Relation to $\|\Psi\|^2$.}
By taking the support of the Thom form $\tau$ to be sufficiently small in the `vertical' $(\liet^\perp)$ direction, we can assume that $\tau_{\hcalX}$ vanishes identically on each component of $\calZ_\beta$ which does not intersect $\hX=\exp^{-1}(\Phi^{-1}(T))$, so that these components give trivial contribution to \eqref{BasicNormSquareFormula}.  Recall that $\hX$ can be identified with the set $\Psi^{-1}(\liet)$, inside the corresponding Hamiltonian $LG$-space $\calM$, and $\hPhi_{\a}|_{\hX}=\Psi|_{\hX}$.  Under this identification
\begin{equation}
\label{IntersectWithX}
\calZ_\beta \cap \hX = \calM^{\bbeta} \cap \Psi^{-1}(\beta).
\end{equation}

If $\gamma=0$, this simplifies to
\[ \calZ_\beta \cap \hX = \calM^\beta \cap \Psi^{-1}(\beta).\]
By Proposition \ref{NormSquareSet}, the latter is non-empty iff $\beta \in W\cdot \calB$, where $\calB$ indexes components of the critical set of $\|\Psi\|^2$.  Thus if $\gamma=0$ then the non-trivial contributions to the norm-square localization formula \eqref{BasicNormSquareFormula} are indexed by the set $W \cdot \calB$.  More generally if $\gamma \ne 0$ is small and generic, then the localizing set $\calZ$ can be viewed as a kind of desingularization (cf. \cite{Paradan97,Paradan98}) of the intersection of $\hX$ with the critical set of $\|\Psi\|^2$.  We next explain this in more detail.

Let $\liet_s$, $s\in \mathcal{S}$ index the sub-algebras of $\liet$ which arise as infinitesimal stabilizers of points in $\calX$.\footnote{In our case, $\mathcal{S}$ is finite, since $M$ is compact.  The discussion here goes through more generally however, see \cite{Paradan97}, \cite{Paradan98}.}  Each connected component of $\hcalX^{\liet_s}$ is mapped by $\hPhi_{\a}$ into an affine subspace $\Delta \subset \liet$ which is a translate of the orthogonal complement to $\liet_s$.

In this way, we obtain a periodic collection $\calW$ of affine subspaces $\Delta$.  For example, for $\liet_s=0$ we obtain $\Delta=\liet$.  For $\Delta \in \calW$, let $\liet_\Delta$ be the subspace orthogonal to $\Delta$ (this is one of the sub-algebras $\liet_s$, $s \in \calS$); this is a rational subalgebra, in the sense that $T_\Delta:= \exp(\liet_\Delta)$ is closed.  Let $\liet_\Delta^\perp$ denote the orthogonal complement to $\liet_\Delta$; the affine subspace $\Delta$ is a translate of $\liet_\Delta^\perp$.

For an affine subspace $\Delta$ and point $x \in \liet$, let $\pr_\Delta(x) \in \Delta$ denote the orthogonal projection of $x$ onto $\Delta$.  Since $\liet_\Delta$, $\Delta$ are orthogonal, $\pr_\Delta(\gamma)=\beta^\prime \Rightarrow \overline{\beta^\prime}=\beta^\prime - \gamma \in \liet_\Delta$.  Thus there is a decomposition
\begin{equation}
\label{Decomposition} 
\calZ_{\beta^\prime} = \bigcup_{ \pr_\Delta(\gamma)=\beta^\prime} \hcalX^{\liet_\Delta} \cap \hPhi_{\a}^{-1}(\beta^\prime),
\end{equation}
the union being indexed by a subset of $\calW$.  Taking the intersection with $\hX$ we obtain
\begin{equation}
\label{DecompositionTwo} 
\calZ_{\beta^\prime} \cap \hX = \bigcup_{\pr_\Delta(\gamma)=\beta^\prime} \calM^{\liet_\Delta} \cap \Psi^{-1}(\beta^\prime).
\end{equation}
In particular, the critical values of $v=\hPhi_{\a}-\gamma$ are contained in the set
\[ \{\pr_\Delta(\gamma)|\Delta \in \calW \}.\]
Likewise, $W \cdot \calB$ is contained in the set
\[ \{\pr_\Delta(0)| \Delta \in \calW \}.\]
For each $\beta \in W \cdot \calB$, introduce the set
\[ S(\beta):=\{ \pr_{\Delta}(\gamma)|\Delta \in \calW, \pr_\Delta(0)=\beta \}. \]
Since $W \cdot \calB$ is discrete, if $\gamma$ is sufficiently small, the finite sets $S(\beta)$ are disjoint.

\begin{proposition}
If $\gamma$ is sufficiently small and $\beta^\prime \in \liet$ is such that $\calZ_{\beta^\prime} \cap \hX \ne \emptyset$, then $\beta^\prime \in S(\beta)$ for some $\beta \in W \cdot \calB$.
\end{proposition}
\begin{proof}
By \eqref{DecompositionTwo}, $\calZ_{\beta^\prime} \cap X$ is non-empty iff there is a $\Delta \in \calW$ such that $\pr_\Delta(\gamma)=\beta^\prime$ and 
\[\calM^{\liet_\Delta} \cap \Psi^{-1}(\beta^\prime) \ne \emptyset.\]
Let $\beta=\pr_\Delta(0)$.  Each component $C \subset \calM^{\liet_\Delta} \cap \Psi^{-1}(\Delta)$ is mapped by $\Psi$ onto a closed set in $\Delta$.  By adjusting $\gamma$, we can make $\beta^\prime=\pr_\Delta(\gamma)$ as close as we like to $\beta=\pr_\Delta(0)$, and so ensure that $\beta^\prime \in \Psi(C)$ $\Rightarrow$ $\beta \in \Psi(C)$.  It follows that for $\gamma$ sufficiently small,
\[\calM^{\liet_\Delta} \cap \Psi^{-1}(\beta^\prime) \ne \emptyset \Rightarrow \calM^{\liet_\Delta} \cap \Psi^{-1}(\beta) \ne \emptyset.\]
And thus $\beta \in W \cdot \calB$.
\end{proof}

The non-trivial contributions in \eqref{BasicNormSquareFormula} are indexed by the set of $\beta^\prime \in \liet$ for which $\calZ_{\beta^\prime} \cap \hX \ne \emptyset$.  The Proposition shows that, for $\gamma$ sufficiently small, this set is contained in the disjoint union of the finite sets $S(\beta)$, $\beta \in W \cdot \calB$.  If we define
\[ \frakm_{\beta}=\sum_{\beta^\prime \in S(\beta)} \DH^{v}_{\calZ_{\beta^\prime}}(\hcalX,\omega_{\a},\hPhi_{\a},\tau_{\hcalX}\alpha),\]
then \eqref{BasicNormSquareFormula} yields the norm-square decomposition described in the introduction:
\begin{equation}
\label{FormulaNotPerturbed}
\DH(\hcalX,\omega_{\a},\hPhi_{\a},\tau_{\hcalX}\alpha)=\sum_{\beta \in W \cdot \calB} \frakm_{\beta}.
\end{equation}

\begin{remark}
The H-K results can also be applied to $v_1=\hPhi_{\a}$ directly (without perturbing by $\gamma$), with $Z_1=\{(v_1)_{\hcalX}=0 \}$ possibly singular.  See the discussion in Appendix C, where we explain that the norm-square contributions for $v_1$ are the terms $\frakm_\beta$ in the formula above.  In particular this implies that the terms are $W$-anti-symmetric: $w_\ast \frakm_{\beta}=(-1)^{l(w)}\frakm_{w\beta}$.
\end{remark}

\subsection{Genericity conditions.}
Let $\beta \in \liet$, and $g=\exp(\beta)$.  The root space decomposition for the infinitesimal stabilizer $\lieg_g$ is
\begin{equation}
\label{RootSpaceDecomposition}
\lieg_g^{\mathbb{C}} = \liet^{\mathbb{C}} \oplus \bigoplus_{\calR_g} \lieg_\alpha,
\end{equation}
where $\calR_g$ consists of those roots $\alpha \in \calR$ such that $g^\alpha=e^{2\pi \i \pair{\alpha}{\beta}}=1$, i.e. such that $\pair{\alpha}{\beta} \in \mathbb{Z}$.  The \emph{Stieffel diagram} is the affine hyperplane arrangement in $\liet$ consisting of all $H_{\alpha,n}=\{\xi \in \liet|\pair{\alpha}{\xi} = n \}$ where $\alpha \in \calR$, $n \in \mathbb{Z}$.  Let $\sigma$ be an open face of the Stieffel diagram.  It follows from \eqref{RootSpaceDecomposition} that for $\beta \in \sigma$, the infinitesimal stabilizer $\lieg_{\exp(\beta)} \supset \liet$ is independent of $\beta$; we denote it by $\lieg_\sigma$.

In order to derive a more explicit formula for the contribution, it will be convenient to choose $\gamma$ (the center of the decomposition) sufficiently generic, similar to \cite{Paradan97}, \cite{Paradan98}.
\begin{proposition}
There is an open dense set (the complement of a periodic hyperplane arrangement in $\liet$) of points $\gamma \in \liet$ such that for all $\Delta, \Delta^\prime \in \calW$ and all open faces $\sigma$ of the Stieffel diagram we have
\begin{enumerate}
\item $\Delta^\prime \subsetneq \Delta$ $\Rightarrow$ $\pr_\Delta(\gamma) \ne \pr_{\Delta^\prime}(\gamma)$.
\item $\pr_\Delta(\gamma) \in \sigma$ $\Rightarrow$ $\Delta \subset \langle \sigma \rangle$.
\end{enumerate}
Here $\langle \sigma \rangle$ denotes the affine extension of $\sigma$.
\end{proposition}
\begin{remark}
The first condition was introduced by Paradan (\cite{Paradan97}, around Proposition 6.9; see there for more discussion).  The second condition is an additional one, see Proposition \ref{AbelCrossSection}.
\end{remark}
\begin{proof}
Since $\Delta^\prime$ has lower dimension than $\Delta$, one can perturb $\gamma$ to ensure $\pr_\Delta(\gamma) \ne \pr_{\Delta^\prime}(\gamma)$.  Likewise in the second condition, if $\Delta$ is not contained in $\langle \sigma \rangle$, then $\Delta \cap \langle \sigma \rangle$ has strictly smaller dimension than $\Delta$, and so a small perturbation of $\gamma$ will ensure that $\pr_\Delta(\gamma) \notin \sigma$.
\end{proof}

\begin{definition}
We say that $\gamma \in \liet$ is \emph{generic} if it is in the open dense set described in the previous proposition.
\end{definition}

\begin{proposition}
\label{DisjointUnion}
Let $\gamma \in \liet$ be generic.  Then the union \eqref{Decomposition} is disjoint.  Equivalently, for all $p \in \hcalX^{\liet_\Delta} \cap \hPhi_{\a}^{-1}(\beta)$, the stabilizer of $p$ in $\liet$ is exactly $\liet_\Delta$.  Thus $\liet_\Delta^\perp$ acts locally freely on $\hcalX^{\liet_\Delta} \cap \hPhi_{\a}^{-1}(\beta)$.
\end{proposition}
\begin{proof}
If $p \in \hcalX^{\liet_\Delta} \cap \hPhi_{\a}^{-1}(\beta)$ had larger stabilizer, then $\beta=\hPhi_{\a}(p)$ would be contained in some $\Delta^\prime \in \calW$ with $\liet_{\Delta^\prime} \supset \liet_\Delta \Rightarrow \Delta^\prime \subset \Delta$.  In this case, $\pr_\Delta(\gamma)=\pr_{\Delta^\prime}(\gamma)=\beta$, which contradicts the genericity assumption.
\end{proof}

\begin{corollary}
Let $\gamma \in \liet$ be generic, $\beta=\pr_\Delta(\gamma)$.  Then $\bbeta$ acts with non-zero weights on the normal bundle $\nu(\hcalX^{\liet_\Delta},\hcalX)|_{\calZ_\beta}$.
\end{corollary}
\begin{proof}
Without loss of generality assume $\calZ_\beta$ is connected and let $C$ be the component of $\hcalX^{\liet_\Delta}$ containing $\calZ_\beta$.  Suppose $\bbeta$ acts trivially on a non-trivial subbundle $\nu^\prime$ of $\nu(C,\hcalX)$.  Then a slightly enlarged submanifold $C^\prime \subset \hcalX^{\bbeta}$ containing $C$ can be found, with $TC^\prime|_{C}=TC \oplus \nu^\prime$.   Since $\liet_\Delta$ acts non-trivially on $\nu^\prime$, the generic infinitesimal stabilizer of $C^\prime$ is strictly smaller than $\liet_\Delta$, and thus there is a $\Delta^\prime \in \calW$ properly containing $\Delta$.  Since $\bbeta \in \liet_{\Delta^\prime}$, $\pr_{\Delta^\prime}(\gamma)=\beta$, and $\hcalX^{\liet_{\Delta^\prime}} \cap \hPhi_{\a}^{-1}(\beta) \supset \hcalX^{\liet_\Delta} \cap \hPhi_{\a}^{-1}(\beta)$.  This contradicts Proposition \ref{DisjointUnion}.
\end{proof}

\begin{proposition}
\label{AbelCrossSection}
Let $\gamma \in \liet$ be generic, $\Delta \in \calW$, $\beta=\pr_\Delta(\gamma)$, and $g=\exp(\beta)$.  Let $\sigma \subset \liet$ be the open face of the Stieffel diagram containing $\beta$, and let $\lieg_\sigma$ be the corresponding infinitesimal stabilizer.  Then:
\begin{enumerate}
\item The centralizer of $\liet_\Delta$ in $\lieg_\sigma$ is $\liet$.
\item There is a slice $U_g$ around $g$ for the conjugation action of $G$ on itself, such that $U_g^{\liet_\Delta}=U_g \cap T$.
\item Let $Y_g=\Phi^{-1}(U_g)$ be the corresponding cross-section.  Then $Y_g^{\liet_\Delta} \subset X=\Phi^{-1}(T)$.
\end{enumerate}
\end{proposition}
\begin{proof}
The root space decomposition for $\lieg_\sigma$ is:
\[ \lieg_\sigma^{\mathbb{C}} = \liet^{\mathbb{C}} \oplus \bigoplus_{\alpha|_{\sigma} \in \mathbb{Z}} \lieg_\alpha.\]
Suppose the root space $\lieg_\alpha \subset \lieg_\sigma^{\mathbb{C}}$ is fixed by $\liet_\Delta$.  Then:
\[ \alpha|_{\sigma} \in \mathbb{Z} \hspace{0.7cm} \text{and} \hspace{0.7cm} \alpha|_{\liet_\Delta} = 0.\]
Let $\sigma_0$ be the subspace parallel to $\sigma$ which passes through $0$.  Then $\alpha|_{\sigma_0}$ is a fixed integer, hence must be $0$.  Thus
\[ \alpha(\sigma_0 + \liet_\Delta)=0.\]
But $\Delta$ is contained in the affine extension of $\sigma$.  This implies that the subspace $\liet_\Delta$ orthogonal to $\Delta$ contains $\sigma_0^\perp$.  Therefore
\[\sigma_0 + \liet_\Delta = \liet,\]
and $\alpha = 0$.  We have thus shown that $\lieg_\sigma^{\liet_\Delta}=\liet$ as desired.

Recall $\lieg_g=\lieg_\sigma$ is the Lie algebra of $G_g$.  The first statement shows that the centralizer $C_{G_g}(T_\Delta)=G_g^{\liet_\Delta}$ has Lie algebra $\liet$, and thus its identity component is $T$.  Since $g \in T$, we can choose a slice $U_g$ around $g$ sufficiently small that it only intersects the identity component $T$ of $G_g^{\liet_\Delta}$, hence $U_g^{\liet_\Delta}=U_g \cap T$.  That $Y_g^{\liet_\Delta} \subset X=\Phi^{-1}(T)$ follows from equivariance of the moment map $\Phi$.
\end{proof}
\begin{remark}
\label{IdentificationInCovering}
We will further assume that $U_g$ is chosen sufficiently small that $\exp$ restricts to a diffeomorphism on each component of $\exp^{-1}(U_g \cap T) \subset \liet$.  Considering the pullback diagram \eqref{SecondPullbackDiagram}, we can identify $U_g \cap T$ with the unique component of $\exp^{-1}(U_g \cap T)$ containing $\beta$.  Similarly, the cross-section $Y_g^{\liet_\Delta}$ is identified with the corresponding subset of $\hX \subset \hcalX$.  By taking $U_g$ smaller if necessary, we can assume that the only translate of $\liet_\Delta^\perp$ in $\calW$ which meets $U_g \cap T$ is $\Delta=\beta + \liet_\Delta^\perp$, and thus $\hPhi_{\a}(Y_g^{\liet_\Delta}) \subset \Delta$.
\end{remark}
\begin{corollary}
\label{HamTSpace}
In the setting of Proposition \ref{AbelCrossSection} and making the identification in Remark \ref{IdentificationInCovering}, $(Y_g^{\liet_\Delta},\omega_{\a},\hPhi_{\a})$ is a (non-degenerate) Hamiltonian $T$-space.
\end{corollary}
\begin{proof}
By Proposition \ref{AbelCrossSection} and Remark \ref{IdentificationInCovering}, $Y_g^{\liet_\Delta} \subset \hX$.  It follows that the restrictions of $\hPhi_{\a}$ and $\phi_g + \beta$ agree, where $\phi_g=\log(\exp(-\beta)\Phi)$ is a Hamiltonian moment map for the cross-section.  Moreover, by equation \eqref{PullbackOfGamma}, the pullback of $\omega_{\a}$ to $Y_g^{\liet_\Delta}$ agrees with the pullback of $\omega$.  By equation \eqref{PullbackOfVarpi} this also agrees with the pullback of $\omega_g$ (the 2-form for the cross-section).  As $\omega_g$ is \emph{symplectic}, $Y_g^{\liet_\Delta}$ is a non-degenerate Hamiltonian $T$-space.
\end{proof}

\begin{proposition}
\label{RegValue}
In the setting of Corollary \ref{HamTSpace} and Remark \ref{IdentificationInCovering}, $\beta$ is a regular value for the restriction of $\hPhi_{\a}$, viewed as a map into $\Delta$:
\[ \hPhi_{\a}:Y_g^{\liet_\Delta} \rightarrow \Delta.\]
\end{proposition}
\begin{proof}
By Proposition \ref{DisjointUnion}, the stabilizer of each point in $Y_g^{\liet_\Delta} \cap \hPhi_{\a}^{-1}(\beta) \subset \hcalX^{\liet_\Delta} \cap \hPhi_{\a}^{-1}(\beta)$ is \emph{exactly} $\liet_\Delta$.  Since $Y_g^{\liet_\Delta}$ is a non-degenerate Hamiltonian $T$-space, it follows that $\beta$ is a regular value of the restriction of $\hPhi_{\a}$, when the latter is viewed as a map into $\Delta$.
\end{proof}
\begin{remark}
Shrinking $\calT$ if necessary, $\beta$ is also a regular value for the restriction
\[ \hPhi_{\a}:\hcalX^{\liet_\Delta} \cap \hPhi_{\a}^{-1}(\Delta) \rightarrow \Delta.\]
(This is an open condition.)  This implies that the level set $\hcalX^{\liet_\Delta} \cap \hPhi_{\a}^{-1}(\beta)$ is smooth.  Hence for generic $\gamma$, the critical set $\calZ$ is smooth.
\end{remark}

\subsection{Reduction to cross-sections.}
Recall that the contribution of a critical value $\beta \in \liet$ to the norm-square localization formula \eqref{BasicNormSquareFormula} is
\begin{equation} 
\label{TheContribution}
\DH^v_{\calZ_\beta}(\hcalX,\omega_{\a},\hPhi_{\a},\tau_{\hcalX}\alpha).
\end{equation}
By definition, this is a twisted DH distribution for a small tubular neighbourhood of $\calZ_\beta$ in $\hcalX$.  In this section we relate this to a twisted DH distribution for a small tubular neighbourhood of $Z_\beta:=\calZ_\beta \cap Y_g$ inside a cross-section $Y_g$, $g=\exp(\beta)$ (when $\gamma$ is generic, the level set $Z_\beta$ is smooth by Proposition \ref{RegValue}).  This in turn will lead to a more explicit formula for \eqref{TheContribution}.

Let us briefly motivate the method used to reduce \eqref{TheContribution} to a computation inside a cross-section $Y_g$.  Recall that $\tau_{\hcalX}$ is the pullback of an equivariant Thom form for the vector bundle $T \times \liet^\perp$.  We can split $\liet^\perp$ into a direct sum $\liet^\perp=\lieg_g^\perp \oplus (\lieg_g/\liet)$.  The normal bundle to $Y_g$ is isomorphic to $Y_g \times \lieg_g^\perp$ by taking tangents to the $\lieg_g^\perp$-orbit directions.  Using this observation, we replace $\tau_{\hcalX}$ with a product of forms, one of which is a Thom form for $\nu(Y_g,\hcalX)$, and then apply Lemma \ref{DeltaCurrent}.

Let $g \in T$, let $U_g \subset G_g$ be a slice around $g$ for the conjugation action, chosen sufficiently small as in Proposition \ref{AbelCrossSection}.  Recall that $\lieg_g/\liet$ denotes the unique $T$-invariant complement to $\liet$ in $\lieg_g$, oriented as described in Section 3.3.  The $T$-equivariant smooth map $r:(U_g \cap T) \times (\lieg_g/\liet) \rightarrow G$ given by
\[ r(h,X)=h\exp(X),\]
restricts to a $T$-equivariant diffeomorphism on a small neighbourhood of $(U_g \cap T) \times \{0 \}$, and in particular induces an isomorphism
\[ (U_g \cap T) \times (\lieg_g/\liet) \simeq \nu(U_g \cap T,U_g).\]
Let $\tau(\lieg_g/\liet) \in \Omega_T(U_g)$ be a $T$-equivariant Thom form for $\nu(U_g \cap T,U_g)$ (viewed as a form on $U_g$ using the map $r$), chosen such that its support is contained in the support of $\tau$ (the Thom form that we had chosen for $\pi_T:\calT \rightarrow T$).  The pullback of $\tau(\lieg_g/\liet,X)$ to $T$ is $\Eul(\lieg_g/\liet,X) \in \Pol(\liet)$.

Consider the map $c:U_g \times \lieg_g^\perp \rightarrow G$ given by
\[ c(u,X)=\exp(X)u\exp(-X).\]
This map is $T$-equivariant for the adjoint action of $T$ on $U_g$, $\lieg_g^\perp$, and $G$.  It is a smooth surjection onto the open subset $\Ad_{G}(U_g)$.  Restricting $c$ to a sufficiently small neighbourhood of $U_g \times \{0 \}$, we obtain a diffeomorphism onto an open tubular neighbourhood $\calU_g$ of $U_g$ in $G$.  This gives an identification of the normal bundle 
\[\nu(U_g,G)\simeq U_g \times \lieg_g^\perp.\]
We orient $\lieg_g^\perp$ as described in Section 3.3.  Let 
\[\pi_g:\calU_g \rightarrow U_g\] 
be the map obtained by inverting $c$ on $\calU_g$ and projecting to the first factor.  Let $\tau(\lieg_g^\perp) \in \Omega_T(\calU_g)$ be a $T$-equivariant Thom form (viewed as a form on $\calU_g$ using the map $c$), chosen such that $\supp \big(\tau(\lieg_g^\perp)\cdot \pi_g^\ast \tau(\lieg_g/\liet) \big) \subset \supp(\tau)$.  The pullback of $\tau(\lieg_g^\perp,X)$ to $T$ is $\Eul(\lieg_g^\perp,X)$.  Setting $\calY_g=\Phi^{-1}(\calU_g)$ we have a pullback diagram:
\begin{equation}
\label{FlowoutProjection}
\begin{CD}
\calY_g @> \Phi >> \calU_g \\
@VV \pi_g V		@VV \pi_g V\\
Y_g @> \Phi >> U_g
\end{CD}
\end{equation}

By choosing a smaller ball $U \subset U_g \cap T$ around $g$ and shrinking $\calT$ if necessary, we can ensure that
\[ \pi_T^{-1}(U)=: \calU \subset \calU_g.\]
See Figure \ref{fig:NearV}.  It then makes sense to restrict the equivariant forms $\tau$, $\pi_g^\ast \tau(\lieg_g/\liet)$, $\tau(\lieg_g^\perp)$ to the open set $\calU$.
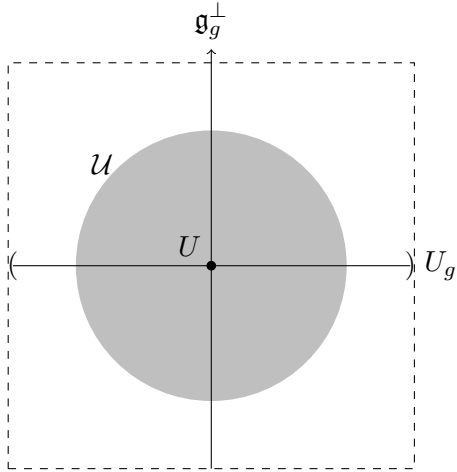
\begin{figure}
\begin{minipage}[c]{8cm}
\centering
\begin{tikzpicture}[scale=0.9]
\coordinate (0) at (0,0);
\path[fill,color=lightgray] (0,0) circle [radius=2];
\draw (-2.93,0)--(2.94,0);
\node at (2.94,0) {)};
\node at (-2.93,0) {(};
\node[right] at (3,0) {$U_g$};
\draw[->] (0,-3)--(0,3.2);
\node[above] at (0,3.2) {$\mathfrak{g}_g^\perp$};
\node[above left] at (-1.3,1.2) {$\calU$};
\draw[dashed] (-3,-3)--(-3,3)--(3,3)--(3,-3)--cycle;
\draw[fill,color=black] (0,0) circle [radius=1.8pt];
\node[above left] at (0,0) {$U$};
\end{tikzpicture}
\end{minipage}
\begin{minipage}[c]{\textwidth-8cm}
\caption{Neighbourhood of $U$ in $G$.}
View the torus $T$ as pointing perpendicular to the page.  $U$ is a small open ball in $T$ around $g$, $\calU=\pi_T^{-1}(U)$ is the gray region.  The region inside the dashed-line border is $\calU_g \simeq U_g \times \mathfrak{g}_g^\perp$.
\end{minipage}
\label{fig:NearV}
\end{figure}
\begin{lemma}
\label{CohomologousThomForms}
The restrictions to $\calU$ of $\tau$ and $\tau(\lieg_g^\perp)\cdot \pi_g^\ast \tau(\lieg_g/\liet)$ are $T$-equivariantly cohomologous.  Moreover, a $T$-equivariant primitive $\beta \in \Omega_T(\calU)$ for the difference can be taken with support contained in $D \cap \calU$.
\end{lemma}
\begin{proof}
The restriction $\tau(\lieg_g^\perp)\cdot \pi_g^\ast \tau(\lieg_g/\liet)|_{\calU}$ is in the same cohomology class with compact vertical supports as $\tau|_{\calU}$ (both are `Poincare dual' to $\calU \cap T$ in $\calU$).  To see this, note that the $T$-equivariant cohomology of $\calU \cap D$ with compact vertical supports is generated by $\tau$ (as a $\Pol(\liet)$-module), and pull-back to $\calU \cap T$ is injective (the bundle $\pi_T:\calU \cap D \rightarrow \calU \cap T$ is $T$-equivariantly diffeomorphic to $(\calU \cap T) \times \liet^\perp$, the base $\calU \cap T$ is contractible).  We have
\[ \iota_T^\ast \big(\tau(\lieg_g^\perp,X)\cdot \pi_g^\ast \tau(\lieg_g/\liet,X)\big)=\Eul(\lieg_g^\perp,X)\Eul(\lieg_g/\liet,X)=\Eul(\liet^\perp,X),\]
where the last equality follows because we chose the orientations on $\lieg_g^\perp$, $\lieg_g/\liet$ such that the product orientation agrees with the orientation on $\liet^\perp$ determined by the positive roots.  This is the same as the pull-back of $\tau(X)$, hence the two forms correspond to the same class in $T$-equivariant cohomology of $\calU\cap D$ with compact vertical supports.
\end{proof}

Let $\Phi:M\rightarrow G$ be a q-Hamiltonian $G$-space, $Y_g=\Phi^{-1}(U_g)$ a cross-section, with $U_g$ as in Proposition \ref{AbelCrossSection}.  Consider a component $\calZ_\beta \subset \calZ$.  Via the covering map $\exp:\hcalX \rightarrow \calX$ (see \eqref{SecondPullbackDiagram}), we make various identifications: $\calZ_\beta$ is identified with a subset of $\Phi^{-1}(\calU) \subset \calX$, $\hPhi_{\a}$ can be viewed as a map $\Phi^{-1}(\calU) \rightarrow \liet$, and $\tau_{\hcalX}$ is identified with $\Phi^\ast \tau$.  By Lemma \ref{CohomologousThomForms} and the flexibility in Definition \ref{DHGerms}, $\tau_{\hcalX}=\Phi^\ast \tau$ can be replaced with the cohomologous (on $\Phi^{-1}(\calU)$) form $\Phi^\ast \tau(\lieg_g^\perp) \cdot \pi_g^\ast \tau(\lieg_g/\liet)$ (the properness condition is immediate as $\supp(\tau(\lieg_g^\perp) \cdot \pi_g^\ast \tau(\lieg_g/\liet)) \subset \supp(\tau)$).  Therefore the contribution of $\beta$ to the norm-square formula is given by:
\begin{equation}
\label{ContributionOnN}
\DH^v_{\calZ_\beta}(\hcalX,\omega_{\a},\hPhi_{\a},\alpha \cdot \tau_{\hcalX})=\DH^v_{\calZ_\beta}(\hcalX,\omega_{\a},\hPhi_{\a},\alpha \cdot \Phi^\ast \tau(\lieg_g^\perp)\cdot \tau(\lieg_g/\liet)).
\end{equation}
Let
\[ Z_\beta:=\calZ_\beta \cap Y_g= Y_g^{\bbeta} \cap \hPhi_{\a}^{-1}(\beta).\]
In the previous subsection we showed that for generic $\gamma$, $Z_\beta$ is smooth.
\begin{theorem}
\label{ReductionToCrossSection}
The contribution from $\beta \in \liet$ to the norm-square formula is equal to
\[ \DH^{v}_{Z_\beta}(Y_g \cap \hcalX, \omega_{\a}, \hPhi_{\a}, \alpha \cdot \Phi^\ast \tau(\lieg_g/\liet)).\]
(We have omitted pullbacks from the notation.)
\end{theorem}
\begin{proof}
By equations \eqref{FlowoutProjection}, \eqref{ContributionOnN}, the contribution of $\beta$ to the norm-square formula is
\[\DH^v_{\calZ_\beta}(\hcalX,\omega_{\a},\hPhi_{\a},\alpha \cdot \Phi^\ast \tau(\lieg_g^\perp)\cdot \tau(\lieg_g/\liet)).\]
By equivariance of $\Phi$, the normal bundle to $Y_g$ in $M$ is isomorphic to $Y_g \times \lieg_g^\perp$, and the pullback of the Thom form $\Phi^\ast \tau(\lieg_g^\perp)$ is a $T$-equivariant Thom form for this vector bundle.  Let $N$ be a small neighbourhood of $\calZ_\beta$ in $\Phi^{-1}(\calU)$, and note that $v|_N$ takes values in a small neighbourhood of $\bbeta$, hence is bounded on $N$.  Applying Lemma \ref{DeltaCurrent} with $S=N \cap Y_g$, $\tau_S=\Phi^\ast \tau(\lieg_g^\perp)$ (and with $\alpha \cdot \Phi^\ast \tau(\lieg_g/\liet)$ in place of $\alpha$) gives the result.
\end{proof}

\subsection{Polarized completions.}
Theorem \ref{ReductionToCrossSection} expressed the contribution of a critical value $\beta$ to the norm-square localization formula as
\begin{equation} 
\label{ContributionInCrossSection}
\DH^{v}_{Z_\beta}(Y_g \cap \hcalX, \omega_{\a}, \hPhi_{\a}, \alpha \cdot \Phi^\ast \tau(\lieg_g/\liet)), \hspace{0.5cm} Z_\beta = Y_g^{\bbeta} \cap \hPhi_{\a}^{-1}(\beta)
\end{equation}
where $Y_g$ is a cross-section for the q-Hamiltonian space $M$ near $g=\exp(\beta)$.  Let $U_\beta$ be a small tubular neighbourhood of $Z_\beta$ \emph{inside} $Y_g \cap \hcalX$.  To compute \eqref{ContributionInCrossSection}, we must construct a $v$-polarized completion $(U_\beta,\tomega,\tphi)$ of $(U_\beta,\omega_{\a},\hPhi_{\a},\Phi^\ast \tau(\lieg_g/\liet) \alpha)$.  In detail: it suffices to find a 2-form $\tomega$ and moment map $\tphi$ on $U_\beta$, which agree with $\omega_{\a}$ and $\hPhi_{\a}$ (respectively) on the localizing set $Z_\beta$, and such that
\[ \pair{\tphi}{v} = \pair{\tphi}{\hPhi_{\a} - \gamma},\]
is proper and bounded below.

To keep the notation simple, assume that a single sub-algebra $\liet_\Delta$ contributes to the disjoint union
\[ Z_\beta = \bigcup_{\pr_\Delta(\gamma)=\beta} Y_g^{\liet_\Delta} \cap \hPhi_{\a}^{-1}(\beta).\]
(In the general case, the arguments here and in the next subsections are applied separately to each sub-algebra $\liet_\Delta$ which contributes.)  The normal bundle splits into a direct sum
\begin{equation} 
\label{SplittingOfNormalBundle}
\nu(Z_\beta,Y_g) = 
\nu(Z_\beta,Y_g^{\liet_\Delta}) \oplus \nu(Y_g^{\liet_\Delta},Y_g)|_{Z_\beta}.
\end{equation}

Recall that by Corollary \ref{HamTSpace}, $(Y_g^{\liet_\Delta},\omega_{\a},\hPhi_{\a})$ is a \emph{non-degenerate} Hamiltonian $T$-space.  By Proposition \ref{RegValue}, $\hPhi_{\a}(Y_g^{\liet_\Delta}) \subset \Delta$ and $\beta$ is a regular value for the restriction of $\hPhi_{\a}$ viewed as a map into $\Delta=\beta + (\liet_\Delta^\perp)^\ast$.  In this situation, the Coisotropic Embedding Theorem provides a local model for a small neighbourhood of the corresponding level set $Z_\beta=Y_g^{\liet_\Delta} \cap \hPhi_{\a}^{-1}(\beta)$ inside $Y_g^{\liet_\Delta}$.  Hence, shrinking $U_\beta$ if necessary, we obtain a local model for $U_\beta^{\liet_\Delta}$.  This, together with the splitting of the normal bundle \eqref{SplittingOfNormalBundle}, give a local model for $U_\beta$.  To describe this local model, we choose a connection $\theta \in \Omega^1(Z_\beta) \otimes \liet_\Delta^\perp$ for the $\liet_\Delta^\perp$-action on $Z_\beta$.  Let
\[ q:Z_\beta \rightarrow Z_\beta/T \]
be the quotient map, and let $\omega_\red$ be the symplectic form on the reduced space $Y_g^{\liet_\Delta}//T=Z_\beta/T$.
\begin{proposition}
\label{NormalForm}
There is a $T$-equivariant diffeomorphism
\[ \psi_0: B_\epsilon \times \calV \xrightarrow{\sim} U_\beta,\]
where $B_\epsilon$ denotes an $\epsilon$-ball around $0 \in (\liet_\Delta^\perp)^\ast$, and
\[ \calV = \nu(Y_g^{\liet_\Delta},Y_g)|_{Z_\beta}.\]
$\psi_0$ is such that $\psi_0(B_\epsilon \times Z_\beta) = U_\beta^{\liet_\Delta}$, and the pullback of $\omega_{\a}- \hPhi_{\a}$ to $B_\epsilon \times Z_\beta \subset B_\epsilon \times \calV$ is
\[ (q^\ast \omega_\red + d \pair{\pr_1}{\theta})-(\pr_1 + \beta).\]
\end{proposition}

\noindent Referring to Proposition \ref{NormalForm}, the construction of the polarized completion below is organized as follows:
\begin{enumerate}[leftmargin=*]
\item Construct $\tomega$, $\tphi$ on $(\liet_\Delta^\perp)^\ast \times \calV$ using minimal coupling.  The resulting $\tphi$ will involve the `polarized' weights of the $\liet_\Delta$-action on $\calV$.
\item Choose a $T$-equivariant diffeomorphism
\[ \psi: (\liet_\Delta^\perp)^\ast \times \calV \rightarrow U_\beta.\]
This diffeomorphism will agree with $\psi_0$ on $\{0 \} \times \calV$, but be such that $\psi((\liet_\Delta^\perp)^\ast \times Z_\beta)=U_\beta^{\liet_\Delta}$ (instead of $\psi_0(B_\epsilon \times Z_\beta)=U_\beta^{\liet_\Delta}$).  This is the `completion' step.
\item Show that $\pair{\tphi}{\psi^\ast v}$ is proper and bounded below on $(\liet_\Delta^\perp)^\ast \times \calV$.  Using $\psi$ to identify $(\liet_\Delta^\perp)^\ast \times \calV$ with $U_\beta$, this will complete the argument.
\end{enumerate}

\subsubsection{Construction of $\tomega-\tphi$ on $(\liet_\Delta^\perp)^\ast \times \calV$.}
Let $\pi: \calV \rightarrow Z_\beta$ denote the projection map.  Recall that, as a consequence of the genericity assumptions, $\bbeta \in \liet_\Delta$ acts with nonzero weights on $\calV$.  Let $\alpha_k^- \in \liet_\Delta^\ast$, $k=1,...,n$ be the (distinct) weights with signs fixed by the condition
\[ \pair{\alpha_k^-}{\bbeta} < 0,\]
and let
\[ \alpha_k^+=-\alpha_k^-, \hspace{1cm} k=1,...,n.\]
$\calV$ can be equipped with a complex structure such that the (distinct) weights of the $\liet_\Delta$-action are $\alpha^-_1,...,\alpha^-_n$.
In fact, this condition determines the complex structure up to homotopy.  Let $V\simeq \mathbb{C}^N$ be a fibre of $\calV$.  Choosing a hermitian inner product on $\calV$, we obtain a reduction of the structure group of $\calV$ to $U(V)=U(N)$.  Let $P$ denote the corresponding $T$-equivariant principal $U(V)$-bundle.  Thus
\[\calV=P \times_{U(V)} V.\]

Let $\omega_V$ be the standard symplectic form on $V$.  The action of $U(V)$ is Hamiltonian; let $\phi_V:V \rightarrow \fraku(V)^\ast$ be the moment map.  The $T_\Delta$-action on $V$ induces a linear map $a:\liet_\Delta \rightarrow \fraku(V)$, and the composition $a^\ast \circ \phi_V$ is given by
\begin{equation}
\label{DeltaMoment}
a^\ast \circ \phi_V(z)=-\pi \sum_k |z_k|^2 \alpha_k^-,
\end{equation}
where $z=\sum z_k$ is the weight-space decomposition.

At this stage, it is convenient to choose a complementary subtorus $T_\Delta^\prime$ to $T_\Delta$ ($\liet_\Delta^\perp$ itself might not be integral).  The exact choice is not important.  We can, for example, choose $T_\Delta^\prime$ such that $T=T_\Delta \times T_\Delta^\prime$ ($\Rightarrow$ $T/T_\Delta \simeq T_\Delta^\prime$).  We thus have two splittings $\liet=\liet_\Delta \oplus \liet_\Delta^\perp=\liet_\Delta \oplus \liet_\Delta^\prime$, and this canonically determines an isomorphism $(\liet_\Delta^\perp)^\ast \simeq (\liet_\Delta^\prime)^\ast$ (both can be identified with the annihilator $\ann(\liet_\Delta) \subset \liet^\ast$).

Next choose a $T$-invariant connection $\sigma$ on $P$.  Since $\liet_\Delta^\perp$ acts locally freely on $\calZ_\beta$, while $\liet_\Delta$ acts trivially, $T_\Delta^\prime$ also acts locally freely on $\calZ_\beta$.  The connection can thus be chosen such that the induced vector fields $X_P$, $X \in \liet_\Delta^\prime$ are horizontal; the connection 1-form $\sigma \in \Omega^1(P) \otimes \fraku(V)$ is then $\liet_\Delta^\prime$-basic.  The $T$-\emph{equivariant curvature} of $\sigma$ is by definition:
\[ R_T=d_T \sigma +\tfrac{1}{2}[\sigma,\sigma] = R - \mu \in \Omega^2_T(P) \otimes \fraku(V)\]
where $R=d\sigma+\tfrac{1}{2}[\sigma,\sigma]$ is the ordinary curvature, and $\mu(X):=\sigma(X_P)$ is the
\emph{moment map} of the connection $\sigma$ (cf. \cite{MeinrenkenEncyclopedia}).  This vanishes for $X \in \liet_\Delta^\prime$, while for $X \in \liet_\Delta$, $\mu(X)$ is the constant function $-a(X) \in \Omega^0(P) \otimes \fraku(V)$.

Consider the composition
\[ \Omega_{U(V)}(V) \rightarrow \Omega_{T \times U(V)}(P \times V)
\rightarrow \Omega_{T}(P \times_{U(V)} V)=\Omega_T(\calV),\]
where the first map is simply pullback.  The second map is the \emph{Cartan map}, defined by substituting the $T$-equivariant curvature for the $\fraku(V)$ variables, followed by using the connection $\sigma$ to project the resulting form to its horizontal part (the result is $U(V)$-basic and descends to the quotient).  The Cartan map is a homotopy inverse to the pullback map $\Omega_{T}(P \times_{U(V)} V) \rightarrow \Omega_{T \times U(V)}(P \times V)$ (cf. \cite{MeinrenkenEncyclopedia}).  Applying this composition to $\omega_V-\phi_V$ we obtain a (`minimal coupling') $T$-equivariant 2-form $\omega_{\min}-\phi_{\min}$
on the total space of $\calV$:
\[ \omega_\min = \Pi_\hor(\omega_V) - \pair{\phi_V}{R}, \hspace{1cm}
\phi_{\min}=-\pair{\phi_V}{\mu}.\]
Here $\Pi_\hor$ denotes the horizontal projection operator for the connection $\sigma$.  Since $\mu(X)=-a(X)$ for $X \in \liet_\Delta$ (and vanishes for $X \in \liet_\Delta^\prime$), $\phi_{\min}$ is given by the same formula \eqref{DeltaMoment}, where now $|\cdot |$ denotes the hermitian inner product on $\calV$.  Since $\sigma$ is $\liet_\Delta^\prime$-basic, the 2-form $\omega_{\min}$ is $\liet_\Delta^\prime$-basic.

We define the `total' $T$-equivariant 2-form $\omega_\calV - \phi_\calV$ on $\calV$ by
\[\omega_\calV=\omega_\red + \omega_\min, \hspace{1cm} \phi_\calV=\beta + \phi_{\min}.\]
We summarize with a Theorem.
\begin{theorem}
\label{FullPolarizedCompletion}
On the local model
\[ \Umod:= (\liet_\Delta^\perp)^\ast \times \calV, \]
there is a closed $T$-equivariant 2-form $\tomega-\tphi$,
\begin{align*}
\tomega &:= \omega_{\calV} + d \pair{\pr_1}{\theta}=\omega_{\red} +\omega_\min+ d \pair{\pr_1}{\theta},\\
\tphi &:= \phi_{\calV}+\pr_1=\beta - \pi \sum_k |z_k|^2 \alpha_k^-+\pr_1,
\end{align*}
where $\omega_{\calV}-\phi_{\calV}$ is a closed $T$-equivariant 2-form on $\calV$.  Its basic properties are:
\begin{itemize}
\item The pullback to the base $(\liet_\Delta^\perp)^\ast \times Z_\beta \subset \Umod$ is $\omega_\red + d\pair{\pr_1}{\theta}-(\pr_1+\beta)$.
\item The pullback to $\calV$ is $\omega_\calV - \phi_\calV$.  Pulling back further to $Z_\beta$ gives $\omega_\red - \beta$.
\item $\omega_\calV$ is $T_\Delta^\prime$-basic, $d\pair{\pr_1}{\theta}$ is $T_\Delta$-basic.
\item $\phi_{\calV}$ takes values in $\liet_\Delta^\ast$, $\pr_1$ takes values in $(\liet_\Delta^\perp)^\ast$.
\end{itemize}
\end{theorem}

\subsubsection{The diffeomorphism $\psi$.}
Recall that the diffeomorphism $\psi_0$ identified $U_\beta$ with a neighbourhood 
\[B_\epsilon \times \calV \subset (\liet_\Delta^\perp)^\ast \times \calV = \Umod.\]
Under pullback
\[ (\psi_0^\ast \hPhi_{\a})|_{B_\epsilon \times Z_\beta} = \pr_1 + \beta\]
and consequently
\[ (\psi_0^\ast v)|_{B_\epsilon \times Z_\beta} = \pr_1+\beta - \gamma=\pr_1+\bbeta.\]
Let $f:(\liet_\Delta^\perp)^\ast \xrightarrow{\sim} B_\epsilon$ be the diffeomorphism
\[ f(\xi)=\frac{\epsilon \xi}{\sqrt{1+|\xi|^2}},\]
and define
\[\psi:=\psi_0 \circ (f \times \Id):\Umod=(\liet_\Delta^\perp)^\ast \times \calV \xrightarrow{\sim} U_\beta.\]
Then
\begin{equation} 
\label{RestrictionToBase}
(\psi^\ast v)|_{(\liet_\Delta^\perp)^\ast \times Z_\beta} = f^\ast \pr_1 + \bbeta.
\end{equation}
Note that $|f^\ast \pr_1| < \epsilon$.

\begin{proposition}
With notation as above, and assuming the neighourhood $U_\beta$ is chosen sufficiently small, $(\Umod, \tomega,\tphi)$ is a $\psi^\ast v$-polarized completion of $(\Umod,\psi^\ast \omega_{\a},\psi^\ast \hPhi_{\a})$.
\end{proposition}
\begin{proof}
By Proposition \ref{NormalForm}, Theorem \ref{FullPolarizedCompletion}, and using the fact that $(f \times \Id)$ restricts to the identity on $Z_\beta$, it follows that the $T$-equivariant 2-form $\tomega-\tphi$ agrees with $\omega_{\a} -\hPhi_{\a}$ on $Z_\beta$.  It remains to check that $\pair{\tphi}{\psi^\ast v}$ is proper and bounded below.  Write
\[ \psi^\ast v = f^\ast \pr_1 + \bbeta + v^\prime,\]
where $v^\prime|_{(\liet_\Delta^\perp)^\ast \times Z_\beta}=0$ by \eqref{RestrictionToBase} ($v^\prime$ is the error term).

For any $\epsilon^\prime>0$, we can ensure that $|v^\prime|<\epsilon^\prime$ by taking the tubular neighbourhood $U_\beta$ of $Z_\beta$ to be sufficiently small in the direction normal to $U_\beta^{\liet_\Delta} \simeq B_\epsilon \times Z_\beta$.  Decompose $v^\prime$ into components,
\[v^\prime= v_0^\prime + v_1^\prime \in \liet_\Delta \oplus \liet_\Delta^\perp.\]
Let $z \in \calV$ and $\xi \in (\liet_\Delta^\perp)^\ast$.  We have
\begin{align*} 
\psi^\ast v(z,\xi)&=f(\xi) + \bbeta + v_0^\prime(z,\xi) + v_1^\prime(z,\xi),\\
\tphi(z,\xi)&= \xi + \beta -\pi \sum |z_k|^2\alpha_k^-.
\end{align*}
Thus
\[ \pair{\tphi}{\psi^\ast v}(z,\xi)=\bigg[ \tfrac{\epsilon |\xi|^2}{\sqrt{1+|\xi|^2}} +\pair{\xi}{v_1^\prime} \bigg] + \pair{\beta}{\psi^\ast v}
- \pi \sum |z_k|^2 \pair{\alpha_k^-}{\bbeta + v_0^\prime}.\]
(Note that some terms vanish because they involve pairings of elements of $\liet_\Delta$ with elements of $(\liet_\Delta^\perp)^\ast$ or the reverse.)  We examine each of the three terms in turn:
\begin{enumerate}[leftmargin=*]
\item If we take $\epsilon^\prime < \tfrac{1}{2} \epsilon$, then
\[ |\pair{\xi}{v_1^\prime}| \le \tfrac{1}{2}\epsilon |\xi|.\]
Thus the first term is dominated by $\tfrac{\epsilon |\xi|^2}{\sqrt{1+|\xi|^2}}$ as $|\xi|$ goes to infinity.  It is therefore bounded below, and goes to infinity as $|\xi|$ goes to infinity.
\item The second term is bounded by a constant:
\[ |\pair{\beta}{\psi^\ast v}| \le |\beta|(\epsilon + |\bbeta| + \epsilon^\prime).\]
\item Recall that $\pair{\alpha_k^-}{\bbeta} < 0$.  For the third term, we take $\epsilon^\prime$ sufficiently small that for each $k=1,...,n$,
\[ -\pi \pair{\alpha_k^-}{\bbeta+v_0^\prime} > \epsilon^{\prime \prime} > 0,\]
for some constant $\epsilon^{\prime \prime}> 0$.  Then the third term is non-negative for all $z,\xi$ and
\[ -\pi \sum |z_k|^2 \pair{\alpha_k^-}{\bbeta+v_0^\prime} > \epsilon^{\prime \prime}|z|^2,\]
which goes to infinity as $|z|$ goes to infinity.
\end{enumerate}
Since $\pair{\tphi}{\psi^\ast v}$ is bounded below and goes to infinity as $|(z,\xi)| \rightarrow \infty$, this completes the proof.
\end{proof}
\noindent As explained at the beginning of this section, using $\psi$ to identify $U_\beta$ with $\Umod=(\liet_\Delta^\perp)^\ast \times \calV$, we obtain a $v$-polarized completion of $U_\beta$.

\subsection{Explicit formulas.}
Recall that the contribution of $\beta$ to the H-K formula for $\DH(\hcalX,\omega_{\a},\hPhi_{\a},\tau_{\hcalX}\alpha)$ is the twisted Duistermaat-Heckman measure
\[ \DH(\Umod, \tomega,\tphi,\Phi^\ast \tau(\lieg_g/\liet) \alpha), \]
where $\Umod, \tomega, \tphi$ are described in Theorem \ref{FullPolarizedCompletion}.  In this section, we derive an explicit formula for this contribution.  Paradan \cite{Paradan97}, \cite{Paradan98} obtained the same type of formula using different methods.  Woodward \cite{Woodward} also obtained similar formulas.  The calculations in this section show how these expressions can be recovered from the Harada-Karshon Theorem.\vspace{0.3cm} \\
{\textbf{Additional notation.}}
\begin{itemize}[leftmargin=*]
\item $X,\zeta, Y, \xi$ will denote elements of $\liet_\Delta$, $\liet_\Delta^\ast$, $\liet_\Delta^\perp$, $(\liet_\Delta^\perp)^\ast$ respectively.  (We distinguish between elements of $\liet$ and $\liet^\ast$ in this section, as this will make the presentation clearer.)
\item The top degree part of the form $e^{\pair{d\xi}{\theta}} \in \Omega(\Umod)$ has a decomposition (determined up to scalar multiples) into a product $d\vol(\xi) \cdot \nu$ where $d\vol(\xi)$ is a volume form on $(\liet_\Delta^\perp)^\ast$.  In coordinates $d\vol(\xi)=d\xi_1 \cdots d\xi_d$, $\nu = (-1)^{\tfrac{1}{2}d(d-1)} \theta^1 \cdots \theta^d$.  It is convenient to fix such a decomposition, and let $\vol(T_\Delta^\prime)$ denote the induced volume of $T_\Delta^\prime$ (using the canonical isomorphism $(\liet_\Delta^\perp)^\ast \simeq (\liet_\Delta^\prime)^\ast$).
\item Let $\balpha \in \Omega_{T_\Delta}(Z_\beta/T_\Delta^\prime)$ denote the image of a cocycle $\alpha \in \Omega_T(\Umod)$ under the composition of restriction to $Z_\beta$ followed by the \emph{Cartan map} for $Z_\beta \rightarrow Z_\beta/T_\Delta^\prime$ (cf. \cite{MeinrenkenEncyclopedia}, or the discussion in Section 4.5.1).  This composition implements the Kirwan map at the level of cocycles.
\end{itemize}
{\textbf{Orientations.}}  The reduced space $Z_\beta/T_\Delta^\prime$ is a symplectic orbifold, which we orient using its Liouville form.  We orient $(\liet_\Delta^\perp)^\ast$ using $d\vol(\xi)$, and the fibres of $Z_\beta \rightarrow Z_\beta/T_\Delta^\prime$ using $\nu$.  The product orientation on $(\liet_\Delta^\perp)^\ast \times Z_\beta$ then agrees with the orientation induced by the 2-form on the base $(\liet_\Delta^\perp)^\ast \times Z_\beta$ (which is symplectic near $Z_\beta \times \{ 0 \}$ by the Coisotropic Embedding Theorem).  The orientation of the vector bundle $\calV \rightarrow Z_\beta$ is fixed by requiring the total orientation to agree with that of $\Umod$ (an open subset of the symplectic manifold $Y_g$).  The resulting orientation on $\calV$ might not agree with the orientation induced by the complex structure on $\calV$ for which the weights are $\alpha_k^-$.  The difference is a sign $\epsilon=\pm 1$.

\begin{lemma}[cf. \cite{Vergne94}]
\label{ReplaceWithKappa}
We can replace $\Phi^\ast \tau(\lieg_g/\liet) \cdot \alpha$ with the pull-back to $\Umod$ of $\Phi^\ast \Eul(\lieg_g/\liet) \cdot \balpha$ without changing the result, that is,
\[ \DH(\Umod, \tomega,\tphi,\Phi^\ast \tau(\lieg_g/\liet) \alpha)=\Eul(\lieg_g/\liet,\partial)\DH(\Umod,\tomega,\tphi,\balpha).\]
\end{lemma}
\noindent (Recall $g=\exp(\beta)$ and $\Eul(\lieg_g/\liet,\xi)$ is a polynomial on $\liet$; see Section 3.3.)
\begin{proof}
As shown in the previous section, $\tphi$ is $v$-polarized.  Since $v|_{U_\beta}$ is bounded, $\tphi$ is proper.  Therefore by Theorem \ref{CohomologyClass}, we can replace $\Phi^\ast \tau(\lieg_g/\liet) \alpha$ with any $T$-equivariantly cohomologous form.  Since $\pr_{Z_\beta}:\Umod \rightarrow Z_\beta$ is a vector bundle, the pullback map $\iota_{Z_\beta}^\ast$ is homotopy inverse to $\pr_{Z_\beta}^\ast$.  Moreover, the Cartan map for the locally free $T_\Delta^\prime$ action on $Z_\beta$ is homotopy inverse to the pullback map $\Omega_{T_\Delta}(Z_\beta/T_\Delta^\prime) \rightarrow \Omega_T(Z_\beta)$.  This shows that we can replace $\Phi^\ast \tau(\lieg_g/\liet) \cdot \alpha$ with the pullback of its image under the map $\Omega(\Umod) \rightarrow \Omega(Z_\beta/T_\Delta^\prime)$ given by pullback to $Z_\beta$ followed by the Cartan map for $Z_\beta \rightarrow Z_\beta/T_\Delta^\prime$.  Recall that $\tau(\lieg_g/\liet)$ is a $T$-equivariant Thom form for the trivial vector bundle with fibre $\lieg_g/\liet$ over an open subset $U$ of $T$, its pullback to $T$ being $\Eul(\lieg_g/\liet)$.  Since $Z_\beta \subset \Phi^{-1}(T)$, the pullback of $\Phi^\ast \tau(\lieg_g/\liet)$ to $Z_\beta$ equals the pullback of $\Eul(\lieg_g/\liet)$.  Finally it is immediate from the definition of DH distributions that for any polynomial $p$,
\[ \DH(\Umod,\tomega,\tphi,p\balpha)=p(\partial)\DH(\Umod,\tomega,\tphi,\balpha).\]
\end{proof}
\noindent Let
\[ \frakm :=\DH(\Umod,\tomega,\tphi,\balpha).\]
For the remainder of this section, we will focus on computing $\frakm$, and replace the factor $\Eul(\lieg_g/\liet,\partial)$ only at the very end.
\begin{lemma}
$\frakm$ is a tempered distribution.  Its pairing with a Schwartz function $f$ is given by
\[ \pair{\frakm}{f} = \int_{(\liet_\Delta^\perp)^\ast} d\vol(\xi) 
\int_{\calV} \nu e^{\omega_{\calV}+\pair{\xi}{F}} \balpha(-\partial_\zeta) f \circ \tphi.\]
\end{lemma}
\begin{proof}
By definition,
\begin{equation} 
\label{DefiningExpressionForM}
\pair{\frakm}{f}=\int_{(\liet_\Delta^\perp)^\ast \times \calV} e^{\tomega} \balpha(-\partial_\zeta) f \circ \tphi. 
\end{equation}
Recall that the $(\liet_\Delta^\perp)^\ast$ component of $\tphi$ is $\pr_1:(\liet_\Delta^\perp)^\ast \times \calV \rightarrow (\liet_\Delta^\perp)^\ast$, so that $|\tphi(z,\xi)| \ge |\xi|$.  Also, $\pair{\tphi(\xi,z)}{\bbeta}\ge C|z|^2+D$ for some constants $C>0,D$.  Combining these facts shows that $|\tphi(z,\xi)|\ge C^\prime(|\xi|+|z|^2) + D^\prime$ for some constants $C^\prime>0, D^\prime$.  It follows that the integrand is rapidly decreasing on $\Umod$ for any Schwartz function $f$, and $\frakm$ is tempered.

We have
\[ e^{\tomega}=e^{\omega_\calV + \pair{\xi}{F}} e^{\pair{d\xi}{\theta}}.\]
Only the top degree part $d\vol(\xi) \, \nu$ of $e^{\pair{d\xi}{\theta}}$ contributes to the integral over $(\liet_\Delta^\perp)^\ast$.  Then use Fubini's theorem to re-write \eqref{DefiningExpressionForM} as an iterated integral.
\end{proof}

\begin{lemma}
\label{InverseTransform}
The $\liet_\Delta$-Fourier transform of $\frakm$ is a generalized function on $\liet_\Delta \times (\liet_\Delta^\perp)^\ast$ given by
\[ \pair{\calF_{\liet_\Delta}(\frakm)}{f dX} =\int_{(\liet_\Delta^\perp)^\ast} d\vol(\xi) \int_{\calV} \int_{\liet_\Delta} 
\balpha(2\pi \i X) \nu e^{\omega_\calV(2\pi \i X) + \pair{\xi}{F}}f(X,\xi) dX.\]
\end{lemma}
\noindent In this expression, $f(X,\xi)$ is a Schwartz function on $\liet_\Delta \times (\liet_\Delta^\perp)^\ast$, and $dX$ is a smooth translation-invariant measure on $\liet_\Delta$.
\begin{proof}
Using $\tphi=(\phi_\calV, \pr_1):\Umod \rightarrow \liet_\Delta^\ast \times (\liet_\Delta^\perp)^\ast$,
\[ \calF_{\liet_\Delta}(f\, dX)\circ \tphi=\int_{\liet_\Delta} f(X,\pr_1)e^{-2\pi \i \pair{\phi_\calV}{X}}dX.\]
By the previous lemma
\begin{align*}
\pair{\calF_{\liet_\Delta}(&\frakm)}{f\,dX}\\
&= \int_{(\liet_\Delta^\perp)^\ast} d\vol(\xi) \int_{\calV} \nu
e^{\omega_\calV + \pair{\xi}{F}}  \balpha(-\partial_\zeta) \calF_{\liet_\Delta}(f \, dX) \circ \tphi \\
&=\int_{(\liet_\Delta^\perp)^\ast} d\vol(\xi) \int_{\calV} \nu e^{\omega_\calV + \pair{\xi}{F}} \int_{\liet_\Delta}
\balpha(2\pi \i X) f(X,\xi)e^{-2\pi \i\pair{\phi_\calV}{X}} dX.
\end{align*}
\end{proof}

\noindent For the next result, let $s_\Delta$ denote the locally constant function on $Z_\beta$, whose value on a component is the number of elements in the stabilizer for the action of $T_\Delta^\prime$ on that component.  We will make use of a $T$-equivariant form with tempered generalized coefficients $\Eul_{\bbeta}^{-1}(\calV)$ introduced by Paradan \cite{Paradan97}, \cite{Paradan98}---a definition is provided in appendix A.
\begin{theorem}[compare \cite{Paradan98} equation (4.57)]
\label{ParadanFormula}
We have
\[ \frakm =
\int_{Z_\beta/T}\frac{\vol(T_\Delta^\prime)}{s_\Delta} \delta_\beta \star \balpha(\partial_\zeta) \calF^{-1}_{\liet_\Delta}\bigg(\Eul_{\bbeta}^{-1}(\calV,2\pi \i X)\bigg) e^{\omega_\red+\pair{\xi}{F}} d\vol(\xi).\]
\end{theorem}
\begin{proof}
Lemma \ref{InverseTransform} can be written more concisely as the equality:
\[ \calF_{\liet_\Delta}(\frakm)(X,\xi) = \int_{\calV} \balpha(2\pi \i X) \nu e^{\omega_\calV(2\pi \i X) + \pair{\xi}{F}},\]
with the understanding that the integrand should be smeared with a Schwartz function $f(X,\xi)$ on $\liet_\Delta \times (\liet_\Delta^\perp)^\ast$ times a smooth translation invariant measure $dX$ on $\liet_\Delta$, before the integral over $\calV$ is performed.  With the exception of $\nu$, the forms in the integrand are $d_{T_\Delta}$-cocycles: indeed, $\balpha$ and $e^{(\omega_\calV)_{T_\Delta}}$ are $d_{T_\Delta}$-cocycles; as $F$ is the pullback of the closed form $d\theta$ on $Z_\beta \subset \calV^{T_\Delta}$, $\pair{\xi}{F}$ is also $d_{T_\Delta}$-closed.  Thus, to apply Corollary \ref{OrbifoldVectorBundleLocalization}, we just need to check a certain polarization condition.  This condition (with exponent $a=2$) is satisfied because of the \emph{polarization} of the weights $\pair{\alpha_k^-}{\bbeta}<0$, $k=1,...,n$ used in the construction of $\tphi$.  Therefore, the integral over the fibres $\calV \rightarrow Z_\beta$ can be computed using Corollary \ref{OrbifoldVectorBundleLocalization}:
\[
\calF_{\liet_\Delta}(\frakm)(X)
=\int_{Z_\beta} d\vol(\xi)\, \nu \,\balpha(2\pi \i X) e^{\omega_\red - 2\pi \i\pair{\beta}{X} +\pair{\xi}{F}} \Eul_{\bbeta}^{-1}(\calV,2\pi \i X).
\]
All the forms in the integrand are $T_\Delta^\prime$-basic except $\nu$ (the Euler form is $T_\Delta^\prime$-basic because the connection $\sigma$ that we chose on $\calV$ was $T_\Delta^\prime$-basic).  Integrating over the fibres of $q:Z_\beta \rightarrow Z_\beta/T_\Delta^\prime=Z_\beta/T$ we have
\[ q_\ast \nu = \frac{\vol(T_\Delta^\prime)}{s_\Delta}.\]
(There is no sign, because the fibres are oriented using $\nu$.)  Hence,
\[ \calF_{\liet_\Delta}(\frakm)(X)=\int_{Z_\beta/T} \frac{\vol(T_\Delta^\prime)}{s_\Delta} \balpha(2\pi \i X) e^{\omega_\red - 2\pi \i\pair{\beta}{X} + \pair{\xi}{F}} \Eul_{\bbeta}^{-1}(\calV,2\pi \i X) d\vol(\xi).\]
Now take the $\liet_\Delta$-inverse Fourier transform.
\end{proof}

It follows from Theorem \ref{ParadanFormula} that the contribution is \emph{polynomial} in $\xi$ (the exponential $e^{\pair{\xi}{F}}$ truncates at finite degree).  For $\alpha \in \liet_\Delta^\ast$, the \emph{Heaviside distribution} $H_\alpha$ has support $\mathbb{R}_{\ge 0} \alpha$ and is defined by
\[ \pair{H_\alpha}{f} = \int_0^\infty f(t\alpha) dt.\]
In the appendix we briefly describe how to compute the inverse Fourier transform of the form $\Eul^{-1}_{\bbeta }(\calV)$ (cf. \cite{GuilleminLermanSternberg,DaSilvaGuillemin,Paradan97,Paradan98} for further discussion).  Recall that $\alpha_k^+=-\alpha_k^-$, $k=1,...,n$ denote the distinct weights of the $\liet_\Delta$-action on $\calV$, with signs such that $\pair{\alpha_k^+}{\bbeta}>0$.  Let $\calV_k$ denote the subbundle on which $T_\Delta$ acts with weight $\alpha_k$.  The end result is
\begin{equation}
\label{FinalForm}
\frakm =
\epsilon \int_{Z_\beta/T}\frac{\vol(T_\Delta^\prime)}{s_\Delta} \balpha(\partial_\zeta) \delta_\beta \star \prod_{k=1}^n H_{\alpha_k^+}^{r_k} \star \sum_{\ell \ge 0} \bigg(-\sum_{m=1}^{r_k} c_m(\calV_k) H_{\alpha_k^+}^m \bigg)^\ell e^{\omega_\red+\pair{\xi}{F}} d\vol(\xi).
\end{equation}
where $c_m(\calV_k)$ is the $m^{th}$ Chern class of $\calV_k$, $r_k$ is the complex rank of $\calV_k$.  The product over $k$ and exponents $\ell$, $m$, $r_k$ in this expression denote convolution.  The sum over $\ell$ truncates after finitely many terms.  The constant $\epsilon = \pm 1$ is $+1$ iff the orientation induced by the complex structure on $\calV$ for which the $T_\Delta$ weights are $\alpha_1^-,...,\alpha_n^-$ agrees with the original orientation on $\calV$.  

Equation \eqref{FinalForm} is a finite sum of terms of the form
\[ c p(\partial_\zeta) \delta_\beta \star \prod_{k=1}^m H_{\alpha_k^+}^{N_k} d\vol(\xi),\]
where $c$ is a constant (obtained by integrating a differential form over $Z_\beta/T$), $p$ is a polynomial, and $N_k \ge r_k$.  We can see explicitly that $\frakm$ is supported in the closed cone $\beta + (\liet_\Delta^\perp)^\ast + \mathbb{R}_{\ge 0}\alpha_1^+ + \cdots + \mathbb{R}_{\ge 0}\alpha_n^+$.

Finally, we put back the factor of $\Eul(\lieg_g/\liet,\partial)$ to obtain an explicit formula for the contribution from $\beta$ to the norm-square localization formula:
\begin{equation}
\label{ExplicitForm}
\DH^v_{Z_\beta}(Y_g \cap \calN,\omega_a,\phi,\tau_{\calN}\alpha)= 
\Eul(\lieg_g/\liet,\partial) \frakm.
\end{equation}
with $\frakm$ given by equation \eqref{FinalForm} above.

\section{Examples}
Here we give two examples: the 4-sphere (a q-Hamiltonian $SU(2)$-space), and a certain multiplicity-free q-Hamiltonian $SU(3)$-space due to Chris Woodward.
\subsection{The 4-sphere.}
The q-Hamiltonian structure on $S^4$ can be constructed by gluing together two copies of a ball in $\mathbb{C}^2$ (a \emph{Hamiltonian} $SU(2)$-space) along their boundaries, analogous to the way the 2-sphere (a Hamiltonian $U(1)$-space) can be glued together from two copies of a ball in $\mathbb{R}^2$ (also Hamiltonian $U(1)$-spaces).  See \cite{AMWDuistermaatHeckman} for details.

Identify $\liet^\ast \simeq \mathbb{R}$ in such a way that $\mathbb{Z}$ is the weight lattice.  Using the basic inner product to identify $\liet \simeq \liet^\ast$, the fundamental alcove is the closed interval $[0,1]$.  We use notation as in the main part of the paper: $\Phi$ is the $SU(2)$-valued moment map, $\Phi_{\a}:\calX \rightarrow T=U(1)$ the abelianization, $\hcalX$ the covering space of $\calX$ with moment map $\hPhi_{\a}:\hcalX \rightarrow \liet^\ast \simeq \mathbb{R}$, etc.  The covering space $\hX$ of $X=\Phi^{-1}(T)$ (inside $\hcalX$) is (topologically) an infinite line of 2-spheres, with each sphere touching its two neighbours at the poles (``beads on a string''), the poles being sent by $\hPhi_{\a}$ to $\mathbb{Z} \subset \mathbb{R}$.  The poles are precisely the (isolated) $T$-fixed points (they correspond to points in $S^4$ which are fixed by all of $SU(2)$).  The space $\hcalX$ is a 4-dimensional non-singular ``thickening'' of $\hX$.

Let $\frakm=\DH(\hcalX,\omega_{\a},\hPhi_{\a},\tau) \in \calD^\prime(\liet^\ast)$; according to the discussion of Section 3.5, this can heuristically be thought of as the (untwisted) Duistermaat-Heckman measure of $\hX$.  Choose $\gamma \in (0,1)=:\sigma$.  The standard cross-section $Y_\sigma$ is the finite cylinder $\hX \cap \hPhi_{\a}^{-1}(0,1)$, which has Duistermaat-Heckman measure $\mathbf{1}_{[0,1]}dx$ where $\mathbf{1}_S$ denotes the indicator function of the subset $S$, and $dx$ denotes Lebesgue measure.  The work of Sections 4.3---4.5 implies that the central contribution to the norm-square localization formula is polynomial times Lebesgue measure and agrees with the DH measure of $Y_\sigma$ near $\gamma$.  Therefore, the central contribution is 
\[ \frakm_\gamma = 1 dx.\]
(A second way to see this is to use the correspondence between $\frakm$ and volumes of reduced spaces, the reduced space $\Phi^{-1}(\exp(\gamma))/T$ being a point.)  Note that this, together with the anti-symmetry of $\frakm$ under the affine Weyl group (generated by reflections in the lattice points $\mathbb{Z}$), determines $\frakm$ on $\mathbb{R} \setminus \mathbb{Z}$.

The correction terms $\frakm_\beta$ come from the $T$-fixed points (the maximal torus is 1-dimensional, so there are no other subtori).  Consider, for example, the critical value $\beta=0$.  The corresponding part of the critical set is a single point $p:=\calZ_\beta=\hPhi_{\a}^{-1}(0)$, and $\bbeta=0-\gamma<0$.  The contribution $\frakm_\beta$ is supported in the half-space $\bbeta x \ge 0$ $\Leftrightarrow$ $x \le 0$.  The local normal form constructed in Section 4.4 is $\calV=T_pM$.  By construction, the cross-section around $p \in M$ is isomorphic to an open ball in $\mathbb{C}^2$, as a Hamiltonian $SU(2)$-space.  Using the induced complex structure $T_pM \simeq \mathbb{C}^2$, the weights of the $U(1) \subset SU(2)$ action are $+1, -1$.  To get the correct contribution, we flip the complex structure on the second copy of $\mathbb{C}$ so that the weights of the $U(1)$ action are both $+1$ (they must have negative pairing with $\bbeta$), and so the moment map is:
\[\tphi(z_1,z_2) = -\pi (|z_1|^2 + |z_2|^2),\]
while the 2-form is the standard symplectic 2-form for $\mathbb{C}^2$ (since we reversed the complex structure on one of the copies of $\mathbb{C}$, this 2-form induces the opposite orientation from the 2-form of the cross-section).  The DH measure of this local normal form is
\[ -H_{-1} \star H_{-1} = x \cdot \mathbf{1}_{(-\infty,0]} dx.\]
Here $H_{-1}$ is the Heaviside distribution equal to $1$ on $(-\infty, 0)$, and the factor of $\epsilon=-1$ appears because the orientation on $\calV=T_pM$ induced by the complex structure for which the weights are both $+1$ is opposite the original orientation.

To obtain the norm-square contribution, the last step is to apply the differential operator $\Eul(\lieg/\liet,\partial)$ (this takes into account the effect of the Thom form), which in this case is $-2\tfrac{d}{dx}$ (since the unique positive root is $\alpha=2$ with our normalization).  And so the contribution is:
\[ \frakm_0 = -2 \cdot \mathbf{1}_{(-\infty,0]}dx.\]
Taking additional critical points into account produces new corrections in a similar way:
\[ \frac{\frakm}{dx} = 1 - 2 \cdot \mathbf{1}_{(-\infty,0]} - 2 \cdot \mathbf{1}_{[1,\infty)}+
2\cdot \mathbf{1}_{(-\infty,-1]} + 2 \cdot \mathbf{1}_{[2,\infty)} - \cdots.\]

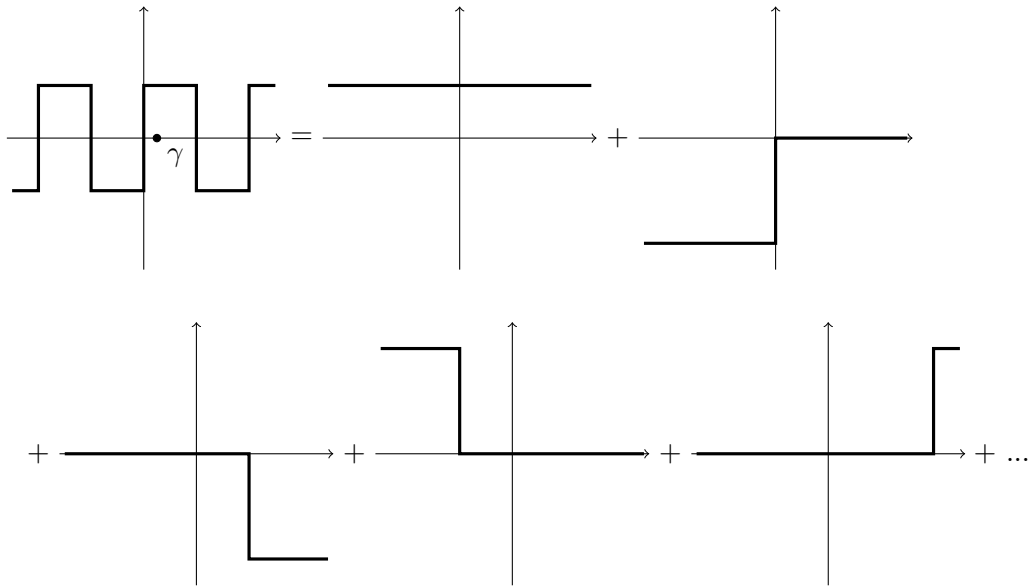
\begin{figure}
\centering
\begin{tikzpicture}[scale=0.7]
\foreach \x in {0,...,2}{
\coordinate (C\x) at ({6*\x},0);
\draw[->] ({6*\x-2.6},0)--({6*\x+2.6},0);
\draw[->] ({6*\x},-2.5)--({6*\x},2.5);
}
\node at (3,0) {=};
\node at (9,0) {+};
\node[below right] at (0.25,0) {$\gamma$};
\draw[fill,color=black] (0.25,0) circle [radius=2pt];
\foreach \x in {3,...,5}{
\coordinate (C\x) at ({6*(\x-3)+1},-6);
\draw[->] ({6*(\x-3)-1.6},-6)--({6*(\x-3)+3.6},-6);
\draw[->] ({6*(\x-3)+1},-8.5)--({6*(\x-3)+1},-3.5);
\node at ({6*(\x-3)-2},-6) {+};
}
\draw[very thick] (-2.5,-1)--(-2,-1)--(-2,1)--(-1,1)--(-1,-1)--(0,-1)--(0,1)--(1,1)--(1,-1)--(2,-1)--(2,1)--(2.5,1);
\draw[very thick] (3.5,1)--(8.5,1);
\draw[very thick] (9.5,-2)--(12,-2)--(12,0)--(14.5,0);
\draw[very thick] (3.5,-8)--(2,-8)--(2,-6)--(-1.5,-6);
\draw[very thick] (4.5,-4)--(6,-4)--(6,-6)--(9.5,-6);
\draw[very thick] (15.5,-4)--(15,-4)--(15,-6)--(10.5,-6);
\node at (16.3,-6) {+ ...};
\end{tikzpicture}
\caption{First 5 terms for the 4-sphere.} \label{fig:4sphere}
\end{figure}

\subsection{A multiplicity-free q-Hamiltonian $SU(3)$-space.}
Figure \ref{fig:Woodward} shows the moment map image $\hPhi_{\a}(\hcalX)$ corresponding to a certain multiplicity-free q-Hamiltonian $SU(3)$-space $M$ (example due to Chris Woodward; see also \cite{Woodward}, where the norm-square contributions for the analogous \emph{Hamiltonian} $SU(3)$-space are described).  The triangle indicated with a bold border is the fundamental alcove, and the smaller shaded triangle inside is the moment polytope of $M$.  The Duistermaat-Heckman measure is $0$ away from the moment map image, while inside the image, its value (relative to a suitably normalized Lebesgue measure) alternates between $\pm 1$ (it is anti-symmetric under the action of the affine Weyl group).  The three facets of the moment polytope correspond to three submanifolds fixed by 1-dimensional subtori of $T$.  Each of these fixed-point submanifolds is (topologically) a 2-torus.  Note in particular that the vertices of the moment polytope do not correspond to $T$-fixed points---this is an example of a q-Hamiltonian space with no $T$-fixed points whatsoever.

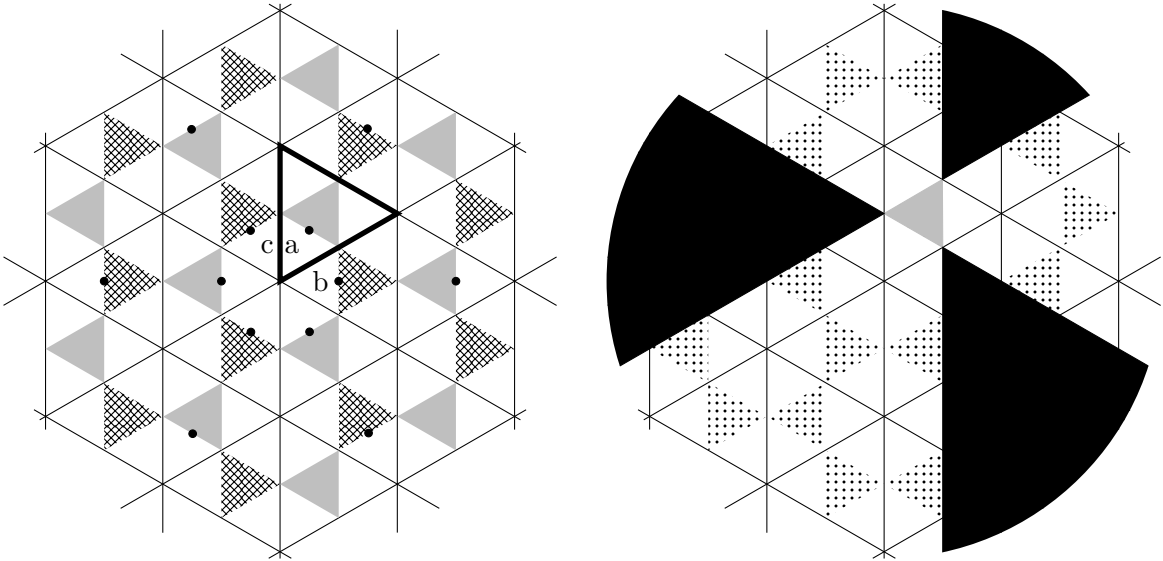
\begin{figure}
\begin{minipage}[c]{0.49\textwidth}
\begin{tikzpicture}[scale=0.9]
\path[use as bounding box] (0,0) circle (4.1);
\coordinate (0) at (0,0);
\draw (30:-4.1)--(30:4.1);
\draw (90:-4.1)--(90:4.1);
\draw (-30:4.1)--(-30:-4.1);

\path[name path=circle] (0,0) circle (4.1);

\begin{pgfinterruptboundingbox}
\path[name path global=line1] ($(30:2)!-100cm!(90:2)$) -- ($(30:2)!100cm!(90:2)$);
\path[name path global=line2] ($(30:-2)!-100cm!(90:-2)$) -- ($(30:-2)!100cm!(90:-2)$);
\path[name path global=line3] ($(30:2)!-100cm!(-30:2)$) -- ($(30:2)!100cm!(-30:2)$);
\path[name path global=line4] ($(30:-2)!-100cm!(-30:-2)$) -- ($(30:-2)!100cm!(-30:-2)$);
\path[name path global=line5] ($(150:2)!-100cm!(90:2)$) -- ($(150:2)!100cm!(90:2)$);
\path[name path global=line6] ($(-30:2)!-100cm!(90:-2)$) -- ($(-30:2)!100cm!(90:-2)$);
\path[name path global=line7] ($(30:4)!-100cm!(90:4)$) -- ($(30:4)!100cm!(90:4)$);
\path[name path global=line8] ($(30:-4)!-100cm!(90:-4)$) -- ($(30:-4)!100cm!(90:-4)$);
\path[name path global=line9] ($(30:4)!-100cm!(-30:4)$) -- ($(30:4)!100cm!(-30:4)$);
\path[name path global=line10] ($(30:-4)!-100cm!(-30:-4)$) -- ($(30:-4)!100cm!(-30:-4)$);
\path[name path global=line11] ($(150:4)!-100cm!(90:4)$) -- ($(150:4)!100cm!(90:4)$);
\path[name path global=line12] ($(-30:4)!-100cm!(90:-4)$) -- ($(-30:4)!100cm!(90:-4)$);
\end{pgfinterruptboundingbox}

\draw[name intersections={of=circle and line1}] (intersection-1)--(intersection-2);
\draw[name intersections={of=circle and line2}] (intersection-1)--(intersection-2);
\draw[name intersections={of=circle and line3}] (intersection-1)--(intersection-2);
\draw[name intersections={of=circle and line4}] (intersection-1)--(intersection-2);
\draw[name intersections={of=circle and line5}] (intersection-1)--(intersection-2);
\draw[name intersections={of=circle and line6}] (intersection-1)--(intersection-2);
\draw[name intersections={of=circle and line7}] (intersection-1)--(intersection-2);
\draw[name intersections={of=circle and line8}] (intersection-1)--(intersection-2);
\draw[name intersections={of=circle and line9}] (intersection-1)--(intersection-2);
\draw[name intersections={of=circle and line10}] (intersection-1)--(intersection-2);
\draw[name intersections={of=circle and line11}] (intersection-1)--(intersection-2);
\draw[name intersections={of=circle and line12}] (intersection-1)--(intersection-2);

\foreach \n in {30,150,270} {\path[fill,color=lightgray] (\n:1)--++(\n+60:1)--++(\n+180:1)--cycle;}
\foreach \n in {90,210,330} {\path[pattern=crosshatch] (\n:1)--++(\n+60:1)--++(\n+180:1)--cycle;}
\foreach \n in {1,2,...,6} \coordinate (A\n) at ($(60*\n-30:2)!(0,0)!(60*\n+30:2)$);
\foreach \n in {1,3,5} \path[pattern=crosshatch] (A\n)--++(30+60*\n-60:1)--++(150+60*\n-60:1)--cycle;
\foreach \n in {2,4,6} \path[fill,color=lightgray] (A\n)--++(30+60*\n-60:1)--++(150+60*\n-60:1)--cycle;
\foreach \n in {1,2,...,6} \path[pattern=crosshatch] (30+60*\n:3)--++(270+240*\n:1)--++(30+240*\n:1)--cycle;
\foreach \n in {0,1,...,5} \path[fill,color=lightgray] (30+60*\n:3)--++(90+240*\n:1)--++(210+240*\n:1)--cycle;

\foreach \n in {0,1,...,5} \coordinate (B\n) at ($(60*\n-30:1)!(0,0)!(60*\n+30:1)$);
\foreach \n in {0,1,...,5} \draw[fill,color=black] (B\n) circle [radius=1.6pt];
\node[left] at (B0) {b};
\node[below left] at (B1) {a};
\node[below right] at (B2) {c};
\foreach \n in {0,1,...,5} \coordinate (C\n) at ($(60*\n-30:3)!(0,0)!(60*\n+30:3)$);
\foreach \n in {0,1,...,5} \draw[fill,color=black] (C\n) circle [radius=1.6pt];
\draw[line width=2pt] (30:2)--(0:0)--(90:2)--cycle;
\end{tikzpicture}
\end{minipage}
\begin{minipage}[c]{0.49\textwidth}
\begin{tikzpicture}[scale=0.9]
\path[use as bounding box] (0,0) circle (4.1);
\coordinate (0) at (0,0);
\draw (30:-4.1)--(30:4.1);
\draw (90:-4.1)--(90:4.1);
\draw (-30:4.1)--(-30:-4.1);

\path[name path=circle] (0,0) circle (4.1);
\path[name path=circle2] (0,0) circle (5);

\begin{pgfinterruptboundingbox}
\path[name path global=line1] ($(30:2)!-100cm!(90:2)$) -- ($(30:2)!100cm!(90:2)$);
\path[name path global=line2] ($(30:-2)!-100cm!(90:-2)$) -- ($(30:-2)!100cm!(90:-2)$);
\path[name path global=line3] ($(30:2)!-100cm!(-30:2)$) -- ($(30:2)!100cm!(-30:2)$);
\path[name path global=line4] ($(30:-2)!-100cm!(-30:-2)$) -- ($(30:-2)!100cm!(-30:-2)$);
\path[name path global=line5] ($(150:2)!-100cm!(90:2)$) -- ($(150:2)!100cm!(90:2)$);
\path[name path global=line6] ($(-30:2)!-100cm!(90:-2)$) -- ($(-30:2)!100cm!(90:-2)$);
\path[name path global=line7] ($(30:4)!-100cm!(90:4)$) -- ($(30:4)!100cm!(90:4)$);
\path[name path global=line8] ($(30:-4)!-100cm!(90:-4)$) -- ($(30:-4)!100cm!(90:-4)$);
\path[name path global=line9] ($(30:4)!-100cm!(-30:4)$) -- ($(30:4)!100cm!(-30:4)$);
\path[name path global=line10] ($(30:-4)!-100cm!(-30:-4)$) -- ($(30:-4)!100cm!(-30:-4)$);
\path[name path global=line11] ($(150:4)!-100cm!(90:4)$) -- ($(150:4)!100cm!(90:4)$);
\path[name path global=line12] ($(-30:4)!-100cm!(90:-4)$) -- ($(-30:4)!100cm!(90:-4)$);
\path[name path global=line13] ($(-30:1)!-100cm!(30:1)$) -- ($(-30:1)!100cm!(30:1)$);
\path[name path global=line14] ($(30:1)!-100cm!(90:1)$) -- ($(30:1)!100cm!(90:1)$);
\path[name path global=line15] ($(90:1)!-100cm!(150:1)$) -- ($(90:1)!100cm!(150:1)$);
\end{pgfinterruptboundingbox}

\draw[name intersections={of=circle and line1}] (intersection-1)--(intersection-2);
\draw[name intersections={of=circle and line2}] (intersection-1)--(intersection-2);
\draw[name intersections={of=circle and line3}] (intersection-1)--(intersection-2);
\draw[name intersections={of=circle and line4}] (intersection-1)--(intersection-2);
\draw[name intersections={of=circle and line5}] (intersection-1)--(intersection-2);
\draw[name intersections={of=circle and line6}] (intersection-1)--(intersection-2);
\draw[name intersections={of=circle and line7}] (intersection-1)--(intersection-2);
\draw[name intersections={of=circle and line8}] (intersection-1)--(intersection-2);
\draw[name intersections={of=circle and line9}] (intersection-1)--(intersection-2);
\draw[name intersections={of=circle and line10}] (intersection-1)--(intersection-2);
\draw[name intersections={of=circle and line11}] (intersection-1)--(intersection-2);
\draw[name intersections={of=circle and line12}] (intersection-1)--(intersection-2);
\draw[name intersections={of=circle2 and line13}] \foreach \i in {1,2} {(intersection-\i) coordinate (p\i)};
\draw[name intersections={of=circle2 and line14}] \foreach \i in {1,2} {(intersection-\i) coordinate (q\i)};
\draw[name intersections={of=circle2 and line15}] \foreach \i in {1,2} {(intersection-\i) coordinate (r\i)};

\foreach \n in {30,150,270} {\path[pattern=dots] (\n:1)--++(\n+60:1)--++(\n+180:1)--cycle;}
\foreach \n in {90,210,330} {\path[pattern=dots] (\n:1)--++(\n+60:1)--++(\n+180:1)--cycle;}
\foreach \n in {1,2,...,6} \coordinate (A\n) at ($(60*\n-30:2)!(0,0)!(60*\n+30:2)$);
\foreach \n in {1,3,5} \path[pattern=dots] (A\n)--++(30+60*\n-60:1)--++(150+60*\n-60:1)--cycle;
\foreach \n in {2,4,6} \path[pattern=dots] (A\n)--++(30+60*\n-60:1)--++(150+60*\n-60:1)--cycle;
\foreach \n in {1,2,...,6} \path[pattern=dots] (30+60*\n:3)--++(270+240*\n:1)--++(30+240*\n:1)--cycle;
\foreach \n in {0,1,...,5} \path[pattern=dots] (30+60*\n:3)--++(90+240*\n:1)--++(210+240*\n:1)--cycle;

\foreach \n in {1,2,6} \coordinate (B\n) at ($(60*\n-30:1)!(0,0)!(60*\n+30:1)$);
\coordinate (c) at ($(30:2)!(0,0)!(90:2)$);
\clip (0,0) circle (4.1);
\draw[fill,color=lightgray] (30:1)--(c)--(90:1)--cycle;
\draw[fill] (30:1)--(p2)--(q2)--cycle;
\draw[fill] (90:1)--(q1)--(r2)--cycle;
\draw[fill] (c)--(p1)--(r1)--cycle;
\end{tikzpicture}
\end{minipage}
\caption{A multiplicity-free q-Hamiltonian $SU(3)$-space.  On the right, the sum of the contributions from a,b,c is shown.} \label{fig:Woodward}
\end{figure}

Also shown in Figure \ref{fig:Woodward} are the critical values which are closest to the origin.  In this case we can choose $\gamma=0$, and the central contribution is identically zero.  We next determine the contribution from the critical value $\beta=a$ lying in the fundamental alcove.  Let $\liet_\Delta \subset \liet$ be the subalgebra generated by $\beta$ (the direction orthogonal to the wall $\Delta$), and let $\liet_\Delta^\perp$ be its orthogonal complement.  We have a direct sum decomposition $\liet^\ast =\liet_\Delta^\ast \oplus (\liet_\Delta^\perp)^\ast$.  

As usual, let $\calX$ denote the abelianization.  The critical set $Z_\beta = \Phi_{\a}^{-1}(\exp(\beta)) \cap \calX^{\liet_\Delta}$ is (topologically) a circle inside the 2-torus $\calX^{\liet_\Delta}$.  According to the discussion in Sections 4.3---4.5, the contribution can be computed using the local normal form: $\calV \times (\liet_\Delta^\perp)^\ast$.  Since $M$ is multiplicity-free, the cross-section is $4$-dimensional, and so $\calV \simeq Z_\beta \times \mathbb{C}$.  Therefore the local normal form is
\[ \calV \times (\liet_\Delta^\perp)^\ast \simeq S^1 \times \mathbb{C} \times \mathbb{R}.\]
The weight of $T_\Delta$ on the rank $1$ complex line bundle $S^1 \times \mathbb{C}$ is $1$.

According to equation \eqref{ExplicitForm}, the contribution $\frakm_\beta$ is
\[ \frakm_\beta = \delta_\beta \star (H_1 \otimes d\vol(\xi)),\]
where $\xi \in (\liet_\Delta^\perp)^\ast$, $d\vol(\xi)$ is normalized Lebesgue measure on $(\liet_\Delta^\perp)^\ast$, and $H_1$ is a Heaviside distribution on $\liet_\Delta^\ast$.  The contributions from the critical points $b,c$ are Weyl reflections of the contribution from $a$, and also have the opposite sign ($-1$), coming from the factor $\Eul(\lieg_g/\liet,\partial)$.  The sum of the contributions from $a,b,c$ is shown in Figure \ref{fig:Woodward}.  Adding the contributions from more critical values produces similar additional corrections.

\appendix

\section{Abelian localization}
The abelian localization formula in equivariant cohomology goes back to the well-known papers of Berline-Vergne \cite{BerlineVergne2} and Atiyah-Bott \cite{AtiyahBottLocalization}.  In this appendix we give a proof of a version of abelian localization used in the main part of the paper.  This version appeared first in \cite{PratoWu}, and a generalization was proved in \cite{Paradan97}, \cite{Paradan98}.  We only treat the special case of the total space of a vector bundle, although it is not difficult to combine the cobordism methods with the vector bundle case here to obtain more general results.  Let $H$ be a torus with Lie algebra $\lieh$.  Let $\pi:\calV \rightarrow Z$ be an oriented, even-rank, $H$-equivariant vector bundle over a compact oriented base $Z$, such that $\calV^H=Z$.

We will want basic bounds on the growth of differential forms along the fibres of $\pi:\calV \rightarrow Z$.  Choose an inner product on the fibres, a metric on the base $Z$ and a connection on $\calV \rightarrow Z$.  The connection, inner product on the fibres, and metric on the base, induce a metric on $\calV$.  Let $|\cdot |$ denote the corresponding norm on the fibres of $T\calV$, and we use the same notation for the induced norm on the fibres of $\wedge T^\ast \calV$.  We will say that $f \in C^\infty(\calV)$ is rapidly decreasing along the fibres if, for each $p \in Z$ and $0\ne v \in \calV_p$, the function $t \mapsto f(tv)$ is rapidly decreasing.  Let $d\vol$ denote the Riemannian volume measure on $\calV$; it has the property that the volume of the disc bundle of radius $r$ is a polynomial in $r$ of degree equal to the real rank of the vector bundle $\calV$.  We will say that a form $\eta \in \Omega(\calV)$ is rapidly decreasing along the fibres if $|\eta|$ is rapidly decreasing along the fibres; in this case $\eta$ is integrable, and
\[ \left| \int_\calV \eta \right| \le \int_\calV |\eta| d\vol.\]

\begin{proposition}
\label{BasicAbelianLocalization}
Let $H, \calV, Z$ be as above.  Let $\eta$ be a differential form which is rapidly decreasing along the fibres of $\pi:\calV \rightarrow Z$.  Let $X=X_1+\i X_2 \in \lieh_\mathbb{C}$ and write $W=2\pi \i X$.  Suppose that $\calV^{X_2}=Z$ and
\[ d_{W}\eta :=(d-\iota(W))\eta=0.\]
Then we have the following abelian localization formula:
\[ \int_\calV \eta = \int_Z \frac{\eta}{\Eul(\calV,W)}.\]
\end{proposition}
\begin{proof}
We adapt the argument in \cite{EquivariantDeRhamTheory}.  Since $\calV^{X_2}=Z$, the weights of the $H$-action on the fibres of $\calV$ do not vanish on $X_2$, and so do not vanish on $W$.  It follows that $\Eul(\calV,W)$ is invertible.  We consider the integral
\begin{equation}
\label{LocalizedIntegral}
\int_Z \eta \Eul^{-1}(\calV,W)
\end{equation}
Let $\tau$ be a compactly supported $H$-equivariant Thom form for the vector bundle $\calV$.  Using the defining property of the Thom form, we can re-write equation \eqref{LocalizedIntegral} as an integral over $\calV$:
\begin{equation}
\label{LocalizedIntegral2}
\int_Z \eta \Eul^{-1}(\calV,W)=\int_\calV \eta \tau(W)\pi^\ast \Eul^{-1}(\calV,W).
\end{equation}
$\tau$ is equivariantly cohomologous to the pullback of the equivariant Euler form $\pi^\ast \Eul(\calV)$, hence
\[ \tau - \pi^\ast \Eul(\calV) = d_{H} \lambda,\]
where $\lambda \in \Omega_{H}(\calV)$.  Equation \eqref{LocalizedIntegral2} becomes
\[ \int_Z \eta \Eul^{-1}(\calV,W)=\int_\calV \eta (d_{W} \lambda(W)+\pi^\ast \Eul(\calV,W))\pi^\ast \Eul^{-1}(\calV,W).\]
One way to obtain a suitable form $\lambda$ is by using the standard de Rham homotopy operator.  This involves pulling $\tau$ back along the scalar multiplication map $[0,1]\times \calV \rightarrow \calV$, $(t,v) \mapsto tv$, and then integrating over $[0,1]$.  For $\lambda$ obtained in this way, the norms $|\lambda(W)|$ and $|d_{W}\lambda(W)|$ grow slowly along the fibres (in fact both can be bounded by a constant).  Since $\eta$ is rapidly decreasing along the fibres, its product with either $d_{W}\lambda(W)$ or $\lambda(W)$ is again rapidly decreasing along the fibres.  It follows that the integral
\begin{equation} 
\label{StokesStep}
\int_\calV \eta \pi^\ast \Eul^{-1}(\calV,W) d_{W} \lambda(W),
\end{equation}
exists, and vanishes by Stokes' theorem.

For the remaining term, $\pi^\ast \Eul(\calV,W)$ cancels, and we obtain:
\[\int_\calV \eta = \int_Z \eta \Eul^{-1}(\calV,W).\]
\end{proof}

We now specialize somewhat.  Let $H, \calV, Z$ be as above.  Let $\alpha$ be a bounded $H$-equivariantly closed form on $\calV$.  Let $\omega-\phi$ be a $H$-equivariantly closed 2-form on $\calV$.  Suppose there exists a vector $\bbeta \in \lieh$ such that (1) $\calV^{\bbeta}=Z$, (2) there is an open cone $\calC$ of directions around $\bbeta \in \lieh$ such that for $\xi \in \calC$, $\pair{\phi}{\xi}$ has homogeneous growth along the fibres of $\pi$, for some positive exponent ($\phi$ is \emph{polarized}).  More precisely, assume that there are constants $C>0,a>0,D$ such that for $\xi \in \calC$, $|\xi|=1$ we have
\[ |\pair{\phi}{\xi}|(v) \ge C|v|^a+D.\]
The polarization condition implies that $\phi$ is proper, and moreover, that if $f$ is a Schwartz function on $\lieh$ then $\phi^\ast f$ is rapidly decreasing along the fibres of $\pi$.

Since $\calV^{\bbeta}=Z$, the weights of the $H$-action on the fibres of $\calV$ do not vanish on $\bbeta$.  Consequently, the form $\Eul(\calV,X-\i s\bbeta)$ is invertible for all $X \in \lieh$, $s \in \mathbb{R} \setminus \{0 \}$.  Following Paradan \cite{Paradan97}, we introduce the following $H$-equivariantly closed differential form with tempered generalized coefficients:
\[ \Eul_{\bbeta}^{-1}(\calV,X):=\lim_{s\rightarrow 0^+} \frac{1}{\Eul(\calV,X-\i s\bbeta)}.\]
(The intended meaning of this expression is that the fraction should be integrated against a smooth rapidly decreasing measure on $\lieh$, before the limit is taken.)  That this form has tempered generalized coefficients follows from the fact that the generalized function
\[ \phi(x)=\lim_{s \rightarrow 0^+} \frac{1}{x-\i s}\]
is tempered.

\begin{theorem}
\label{VectorBundleLocalization}
Let $H, \calV, Z, \omega-\phi, \bbeta,\alpha$ be as above.  We have the following equality of tempered generalized functions on $\lieh$:
\begin{equation} 
\label{EquationOfThm}
\int_{\calV} \alpha(2\pi \i X) e^{\omega-2\pi \i \pair{\phi}{X}} = \int_{Z} 
\alpha(2\pi \i X)e^{\omega-2\pi \i \pair{\phi}{X}}\Eul^{-1}_{\bbeta}(\calV,2\pi \i X).
\end{equation}
\end{theorem}
\noindent (In this expression, the integrands should be paired with a rapidly-decreasing smooth measure on $\lieh$ before the integrals over $\calV$, $Z$ are performed.)
\begin{proof}
Let $f\, dX$ be a rapidly decreasing smooth measure on $\lieh$, with Fourier transform $\hat{f}$.  Write $\alpha=\sum_k \alpha_k p_k$.  Then
\[ \int_{\lieh} \alpha(2\pi \i X) e^{\omega-2\pi \i \pair{\phi}{X}}f(X)\, dX=\sum_k \alpha_k e^{\omega} \phi^\ast (p_k(-\partial)\hat{f}).\]
Since $p_k(-\partial) \hat{f}$ is a Schwartz function, its pullback by $\phi$ is rapidly decreasing along the fibres of $\pi$.  It follows that the left-hand-side of \eqref{EquationOfThm} defines a tempered generalized function.  Moreover, we have:
\begin{align*}
\int_{\calV} \int_{\lieh} \alpha(2\pi \i X) e^{\omega-2\pi \i \pair{\phi}{X}} f(X)\, dX &=
\int_{\calV} \int_{\lieh} \lim_{s\rightarrow 0^+} \alpha(2\pi \i(X-\i s\bbeta)) e^{\omega-2\pi \i \pair{\phi}{X-\i s\bbeta}} f(X)\, dX,\\
&=\lim_{s\rightarrow 0^+} \int_{\calV} \int_{\lieh} \alpha(2\pi \i(X-\i s\bbeta)) e^{\omega-2\pi \i \pair{\phi}{X-\i s\bbeta}} f(X)\, dX,\\
&=\lim_{s\rightarrow 0^+} \int_{\lieh} f(X)\, dX \int_{\calV} \alpha(2\pi \i(X-\i s\bbeta)) e^{\omega-2\pi \i \pair{\phi}{X-\i s\bbeta}}.
\end{align*}
Here, the second line follows from the dominated convergence theorem, since the integrand is rapidly decreasing along the fibres of $\calV$ and on $\lieh$.  The third line follows from Fubini's theorem, using the fact that $e^{-2\pi \i \pair{\phi}{X-\i s\bbeta}}$ is rapidly decreasing along the fibres (so that the decay of $f(X)$ is no longer needed to guarantee convergence of the integral over $\calV$).  We can now apply the previous Proposition \ref{BasicAbelianLocalization} to the integral over $\calV$, since the integrand is rapidly decreasing along the fibres, and is closed for the differential $d_W$, $W=2\pi \i (X-\i s\bbeta)$.  The result is:
\[ \int_{\calV} \int_{\lieh} \alpha(2\pi \i X) e^{\omega-2\pi \i \pair{\phi}{X}} f(X)\, dX
= \lim_{s\rightarrow 0^+} \int_{\lieh} f(X) \, dX \int_{Z} \frac{\alpha(2\pi \i(X-\i s\bbeta)) e^{\omega-2\pi \i \pair{\phi}{X-\i s\bbeta}}}{\Eul(\calV,2\pi \i (X-\i s\bbeta))}.\]
Applying Fubini's theorem and the dominated convergence theorem now in the opposite direction gives the result.
\end{proof}

\noindent In order to directly apply Theorem \ref{VectorBundleLocalization} in the main part of the paper, we would need to have a version of \ref{VectorBundleLocalization} for the orbifold $\calV/T_\Delta^\prime$.  Alternatively, we can slightly modify \ref{VectorBundleLocalization} to work with basic forms on $\calV$, which is what we do now.
\begin{corollary}
\label{OrbifoldVectorBundleLocalization}
Let $H, \calV, Z, \omega-\phi, \bbeta,\alpha$ be as in Theorem \ref{VectorBundleLocalization}.  Let $H^\prime$ be a \emph{second} torus acting on $\calV$ and commuting with $H$.  Suppose the action of $H^\prime$ on the base $Z$ is locally free, and that $\alpha$, $(\omega-\phi)$ are $H^\prime$-basic.  Let $q:\calV \rightarrow \calV/H^\prime$ be the quotient map, and let $\nu \in \Omega(Z)$ be a form such that $q_\ast \nu =c$, a constant.  Then we have the following equality of tempered generalized functions on $\lieh$:
\begin{equation} 
\label{ThmEquation}
\int_{\calV} \alpha(2\pi \i X) \nu e^{\omega-2\pi \i \pair{\phi}{X}} = \int_{Z} 
\alpha(2\pi \i X) \nu e^{\omega-2\pi \i \pair{\phi}{X}}\Eul^{-1}_{\bbeta}(\calV,2\pi \i X).
\end{equation}
\end{corollary}
\begin{proof}
Since $H^\prime$ acts locally freely on the base $Z$, we can choose an $H^\prime$-invariant connection $\sigma$ on $\pi:\calV \rightarrow Z$ such that the vector fields generated by the $H^\prime$ action are horizontal.  The connection 1-form is then $H^\prime$-basic, and thus the corresponding representative for the Euler form is $H^\prime$-basic.  We can also take the Thom form $\tau$ to be $H^\prime$-basic (for example, using the Mathai-Quillen representative built using $\sigma$).  The proof is now almost exactly the same: we repeat the proof of Proposition \ref{BasicAbelianLocalization} with $\eta \nu$ replacing $\eta$, and then repeat the proof of Theorem \ref{VectorBundleLocalization} with $\alpha \nu$ replacing $\alpha$.  The only step requiring modification is the final step \eqref{StokesStep} in the proof of Proposition \ref{BasicAbelianLocalization}: we are left with the integral
\[ \int_\calV \eta \nu \pi^\ast \Eul^{-1}(\calV,W)d_W \lambda(W)= \int_{\calV} d_W \bigg( \eta \pi^\ast \Eul^{-1}(\calV,W) \lambda(W)\bigg) \nu.\]
All forms in the integrand are $H^\prime$-basic, except $\nu$ ($\lambda$ is basic if we use the de Rham homotopy operator, for example).  Integrating over the fibres of $q$ and applying the orbifold version of Stokes' theorem on $\calV/H^\prime$ shows that this integral vanishes.
\end{proof}
In the main part of the paper, we apply \ref{OrbifoldVectorBundleLocalization} to the case $H=T_\Delta$, $H^\prime=T_\Delta^\prime$, $\omega-\phi=\omega_\calV-\phi_\calV$, and $\alpha$ is the $T_\Delta$-equivariantly closed, $T_\Delta^\prime$-basic (in fact, $T$-basic) form $\balpha(X)e^{\pair{\xi}{F}}$ with $\xi \in (\liet_\Delta^\perp)^\ast$ fixed.

Finally, we describe briefly how to compute the inverse Fourier transform of $\Eul_{\bbeta}^{-1}(\calV)(2\pi \i X)$.  Choose a complex structure on $\calV$ such that the complex weights $\alpha_1^-,...,\alpha_n^-$ of the $H$-action have negative pairing with $\bbeta$.  The bundle $\calV$ splits into a direct sum of weight sub-bundles $\calV_k$:
\[ \calV=\bigoplus_{k=1}^n \calV_k,\]
where $H$ acts on $\calV_k$ with weight $\alpha_k^- \in \lieh^\ast$.  As the Euler class is multiplicative, it suffices to give an expression for the inverse Fourier transform of $\Eul_{\bbeta}^{-1}(\calV_k)(2\pi \i X)$ (the full result will then be a convolution).  So assume $\calV$ has complex rank $r$ and that $H$ acts on $\calV$ by a single weight $\alpha^- \in \lieh^\ast$.

Choose an $H^\prime$-basic connection on $\calV$ with $\mathfrak{u}(\calV)$-valued curvature $R$.  One has
\[ \Eul(\calV,2\pi \i X)= \epsilon \, \, \textnormal{det}_{\bC}\Big( \tfrac{\i}{2\pi}R-2\pi \i\pair{\alpha^-}{X}\Big),\]
where $\epsilon=\pm 1$ is a sign, equal to $-1$ iff the orientation on $\calV$ induced by the complex structure is opposite the original orientation on $\calV$.  For example, if $\calV$ is a (complex) line bundle, $\Eul(\calV,2\pi \i X)=\epsilon (c_1(\calV)-2\pi \i\pair{\alpha^-}{X})$.  Set $z=-2\pi \i\pair{\alpha^-}{X}$ and expand the determinant
\[ \Eul(\calV,2\pi \i X)=\epsilon z^r\bigg(1+\sum_{m=1}^rc_m(\calV)z^{-m}\bigg),\]
where $c_m(\calV)$ is the $m^{th}$ Chern class of $\calV$.  Inverting and expanding in power series:
\[ \frac{1}{\Eul(\calV,2\pi \i X)}=\frac{\epsilon}{z^r} \sum_{\ell \ge 0} \bigg(-\sum_{m=1}^r c_m(\calV)z^{-m}\bigg)^\ell.\]
Thus
\[ \Eul_{\bbeta}^{-1}(\calV,2\pi \i X)=\lim_{s\rightarrow 0^+}\frac{\epsilon}{(-2\pi \i\pair{\alpha^-}{X-\i s\bbeta})^r} \sum_{\ell \ge 0} \bigg(-\sum_{m=1}^r \frac{c_m(\calV)}{(-2\pi \i\pair{\alpha^-}{X-\i s\bbeta})^m} \bigg)^\ell.\]
The sum over $\ell$ truncates since $c_m(\calV)^\ell=0$ for $\ell$ larger than half the dimension of the base.  Using the fact that $\pair{\alpha^-}{\bbeta} < 0$,
\[ \calF^{-1}\bigg( \lim_{s\rightarrow 0^+}\frac{1}{-2\pi \i\pair{\alpha^-}{X-\i s\bbeta}}\bigg) = H_{-\alpha^-}=H_{\alpha^+},\]
where $\alpha^+=-\alpha^-$ (note $\pair{\alpha^+}{\bbeta}>0$) and $H_{\alpha^+}$ denotes the Heaviside distribution defined by
\[ \pair{H_{\alpha^+}}{f}=\int_0^\infty dt f(t\alpha^+).\]
Since the Fourier transform of a product is the convolution of the Fourier transforms, we obtain
\[ \calF^{-1}(\Eul_{\bbeta}^{-1}(\calV,2\pi \i X))=\epsilon H_{\alpha^+}^r \star \sum_{\ell \ge 0} \bigg(-\sum_{m=1}^r c_m(\calV) H_{\alpha^+}^m \bigg)^\ell.\] 
Hence in general we have
\[ \calF^{-1}(\Eul_{\bbeta}^{-1}(\calV,2\pi \i X))=\epsilon \prod_{k=1}^n H_{\alpha_k^+}^{r_k} \star \sum_{\ell \ge 0} \bigg(-\sum_{m=1}^{r_k} c_m(\calV_k) H_{\alpha_k^+}^m \bigg)^\ell.\] 
The exponents and product over $k$ in this expression denote convolution, and $r_k$ is the complex rank of $\calV_k$.

\section{Hamiltonian cobordisms}
In this appendix, we recall the definition of Hamiltonian cobordisms (\cite{GinzburgGuilleminKarshon},\cite{HamiltonianCobordismBook},\cite{KarshonHarada}), and then outline proofs of some key results referred to in Section 2.

\begin{definition}
\label{HamiltonianCobordism}
Let $(N_i,\omega_i,\phi_i,\alpha_i)$, $i=0,1$ and $(\tilde{N},\tomega,\tphi,\talpha)$ be Hamiltonian $G$-spaces equipped with closed equivariant differential forms.  We say that $(\tilde{N},\tomega,\tphi,\talpha)$ is a \emph{proper Hamiltonian cobordism} between $(N_0,\omega_0,\phi_0,\alpha_0)$ and $(N_1,\omega_1,\phi_1,\alpha_1)$ if $\tphi$ is proper on the support of $\talpha$ and
\[ \partial \tN = N_0 \sqcup (-N_1), \hspace{0.5cm} \iota_{\partial \tN}^\ast (\tomega - \tphi) = (\omega_0-\phi_0)\sqcup (\omega_1-\phi_1), \hspace{0.5cm} \iota_{\partial \tN}^\ast \talpha = \alpha_0 \sqcup \alpha_1.\]
\end{definition}
\begin{remark}
Note that the definition implies $\tphi_i$ is proper on the support of $\alpha_i$, $i=0,1$.
\end{remark}
\begin{theorem}[\cite{KarshonHarada} Lemma 5.10]
\label{CobordismEquality}
In the setting of Definition \ref{HamiltonianCobordism}, 
\[\DH(N_0,\omega_0,\phi_0,\alpha_0)=\DH(N_1,\omega_1,\phi_1,\alpha_1).\]
\end{theorem}
\begin{proof}
Using the definition, the difference is a distribution $\frakm$ whose pairing with $f \in C^\infty_\comp(\lieg^\ast)$ is
\begin{align*} 
\pair{\frakm}{f} &= \int_{\partial \tN} \iota_{\partial \tN}^\ast (e^{\tomega} \talpha(-\partial)f \circ \tphi)\\
&= \int_{\tN} d(e^{\tomega} \talpha(-\partial)f\circ \tphi).
\end{align*}
Let $\{U_k \}$ be a locally finite open cover of $\lieg^\ast$ by relatively compact sets, and let $\{\rho_k \}$ be a partition of unity subordinate to this open cover.  Define distributions $\frakm_k$,
\[ \pair{\frakm_k}{f} = \int_{\tN} (\tphi^\ast \rho_k) \, d(e^{\tomega} \talpha(-\partial)f\circ \tphi).\]
Then $\supp(\frakm_k) \subset \overline{U_k}$, and it follows that $\frakm=\sum_k \frakm_k$ (the sum is locally finite) and each $\frakm_k$ is compactly supported.  To show $\frakm_k=0$ it suffices to show all of its Fourier coefficients vanish.  For $X \in \lieg$ fixed
\begin{align*}
\pair{\frakm_k}{e^{-2\pi \i \pair{-}{X}}}&= \int_{\tN} (\tphi^\ast \rho_k) \, d(e^{\tomega-2\pi \i \pair{\tphi}{X}} \talpha(2\pi \i X))\\
&=2\pi \i \int_{\tN} (\tphi^\ast \rho_k) \, \iota(X_N)(e^{\tomega-2\pi \i \pair{\tphi}{X}} \talpha(2\pi \i X)),
\end{align*}
where we have used $(d-2\pi \i \iota(X_N))(e^{\tomega-2\pi \i \pair{\tphi}{X}} \talpha(2\pi \i X))=0$.  The last line vanishes, because the integrand has no top-degree part.
\end{proof}

\noindent The next result can be compared to Theorem \ref{CohomologyClass}.
\begin{theorem}[\cite{KarshonHarada}]
\label{CohomologousTwoForms}
Let $N$ be a $G$-manifold, and $\alpha$ a closed equivariant differential form on $N$.  Let $v:N \rightarrow \lieg$ be a bounded taming map.  Let $\omega_i-\phi_i$, $i=0,1$ be two equivariant 2-forms on $N$.  Suppose that
\begin{enumerate}
\item $[\omega_0-\phi_0]=[\omega_1-\phi_1] \in H_G(N)$
\item $\pair{\phi_i}{v}$ is proper and bounded below on $\supp(\alpha)$, $i=0,1$.
\end{enumerate}
Then $\DH(N,\omega_0,\phi_0,\alpha)=\DH(N,\omega_1,\phi_1,\alpha)$.
\end{theorem}
\begin{proof}
Write $(\omega_1-\phi_1)-(\omega_0-\phi_0)=d_G \eta$ for some $\eta \in \Omega^1_G(N)$.  Let $W=N \times [0,1]$, and pull back all equivariant forms as well as $v$ by $\pr_1$ (pullbacks omitted from notation).  Let $t$ be the coordinate on $[0,1]$.  Define $\tomega_G=\tomega-\tphi$ by
\[ \tomega_G = (\omega_0)_G + d_G(t\eta).\]
We have $\tphi=(1-t)\phi_0+t\phi_1$, and thus
\[ \pair{\tphi}{v}=(1-t)\pair{\phi_0}{v}+t\pair{\phi_1}{v}.\]
Since a convex linear combination of proper, bounded below maps is proper and bounded below, $\pair{\tphi}{v}$ is proper and bounded below on $\supp(\alpha)$.  Since $v$ is bounded, $\tphi$ is proper (Prop. \ref{TameProper}).  $(W,\tomega,\tphi,\alpha)$ is the desired proper Hamiltonian cobordism.
\end{proof}
\begin{remark}
\label{SmoothCollapseCohomologous}
Recall (see remark \ref{SmoothCollapse}) that one situation in which the condition $[\omega_0-\phi_0]=[\omega_1-\phi_1] \in H_G(N)$ is satisfied, is if $N$ can be smoothly collapsed to part of a $G$-invariant submanifold $Z$, and $\iota_Z^\ast(\omega_0-\phi_0)=\iota_Z^\ast(\omega_1-\phi_1)$.
\end{remark}

The Harada-Karshon Theorem is proven by showing that $N$ is cobordant to a small neighbourhood of the critical set $Z$.  The main challenge in the cobordism approach (\cite{GinzburgGuilleminKarshon},\cite{HamiltonianCobordismBook},\cite{KarshonHarada}) is to construct a proper moment map on the cobording manifold.  As already pointed out in Lemma \ref{TameProper}, when $v$ is bounded, $\pair{\phi}{v}$ proper and bounded below implies that $\phi$ is proper.  The convenience of working with the condition that $\pair{\phi}{v}$ be proper and bounded below comes from two point-set topology facts:
\begin{enumerate}[(A)]
\item A finite collection of proper, bounded below functions $f_i$ can be patched together with a partition of unity, and the result is again proper and bounded below (\cite{KarshonHarada}, Lemma 3.5).
\item For $G$ compact, the $G$-average of a proper, bounded below function is again proper and bounded below (\cite{KarshonHarada}, Lemma 3.6).
\end{enumerate}
\noindent We next outline the proof of Lemma \ref{Hypotheses}, which we break into `existence' and `uniqueness' parts.
\begin{lemma}[`existence', \cite{KarshonHarada} Proposition 3.4]
\label{ExistencePart}
Let $(N,\omega,\phi)$ be a Hamiltonian $G$-manifold, possibly with boundary, and $\alpha$ a closed equivariant differential form.  Let $v:N \rightarrow \lieg$ be a bounded taming map with localizing set $Z$.\footnote{We \emph{do not} require that $Z$ be smooth here.}  Let $Y \supset Z$ be a $G$-invariant closed set.  Suppose that $\pair{\phi}{v}$ is proper and bounded below on $Y \cap \supp(\alpha)$.  Then there exists a $v$-polarized completion $(N,\tomega,\tphi)$ of $(N,\omega,\phi,\alpha)$ such that $\tomega-\tphi$ equals $\omega-\phi$ on a neighbourhood of $Y$.
\end{lemma}
\begin{proof}
Since $\pair{\phi}{v}$ is proper and bounded below on $Y \cap \supp(\alpha)$, it is possible to find a $G$-invariant open neighbourhood $U \supset Y \cap \supp(\alpha)$ such that $\pair{\phi}{v}$ is proper and bounded below on $\overline{U}$ (\cite{KarshonHarada}, Lemma 3.8).  Let $U_Y=(N \setminus \supp(\alpha)) \cup U$ (an open neighbourhood of $Y$).  Let $\rho_1,\rho_2$ be a partition of unity subordinate to the open cover $U_Y, N \setminus Y$, and let $f:N \rightarrow \mathbb{R}$ be any proper and bounded below smooth function.  Let
\[ \psi^\prime = \rho_1 \pair{\phi}{v}+\rho_2 f.\]
It's clear that $\psi^\prime$ agrees with $\pair{\phi}{v}$ on a neighbourhood of $Y$.  Note that $\psi^\prime$ is proper and bounded below on $\overline{U}$, by point-set topology fact (A).  Also, $\psi^\prime$ agrees with $f$ on $N \setminus U_Y$, and so $\psi^\prime|_{N \setminus U_Y}$ is proper and bounded below.  Since $\supp(\alpha)\subset \overline{U}\cup (N \setminus U_Y)$ (a union of two closed sets), this proves that $\psi^\prime$ is proper and bounded below on $\supp(\alpha)$.  Now set $\psi$ to be the $G$-average of $\psi^\prime$.  By point-set topology fact (B), $\psi$ is proper and bounded below on $\supp(\alpha)$.  Moreover $\psi$ equals $\pair{\phi}{v}$ near $Y$.

Choose a $G$-invariant Riemannian metric $g$ on $N$, and define a 1-form $\Theta$ on $N \setminus Z$,
\[ \Theta = \frac{g(v_N,-)}{g(v_N,v_N)}.\]
Finally, let
\[ \tomega-\tphi = \tomega_G := \omega_G + d_G\big( (\psi-\pair{\phi}{v})\Theta \big).\]
Since $\psi$ and $\pair{\phi}{v}$ are equal near $Y$, $\tomega_G$ is defined on all of $N$ and equals $\omega_G$ on a neighbourhood of $Y$.  And
\[ \pair{\tphi}{v} = \psi \]
which is proper and bounded below on $\supp(\alpha)$ by construction.  $(N,\tomega,\tphi)$ is the desired $v$-polarized completion.
\end{proof}

\begin{lemma}[`uniqueness', \cite{KarshonHarada} Lemmas 4.12, 4.17]
\label{UniquenessPart}
Let $(N,\omega,\phi)$ be an oriented Hamiltonian $G$-manifold, and $\alpha$ a closed equivariant differential form.  Let $v_i, Z_i, (U_i,\tomega_i,\tphi_i)$, $i=0,1$ be two sets of data satisfying the following conditions:
\begin{enumerate}
\item $v_i:N \rightarrow \lieg$ is a bounded taming map with smooth localizing set $Z_i$,
\item $U_i$ is a $G$-invariant open set that can be smoothly collapsed to part of $Z_i$,
\item $\pair{\tphi_i}{v_i}$ is proper and bounded below on $Z_i \cap \supp(\alpha)$.
\end{enumerate}
Suppose further that $v_0$, $v_1$ agree on $\supp(\alpha)$.  Then
\[ \DH(U_1,\tomega_1,\tphi_1,\alpha)=\DH(U_0,\tomega_0,\tphi_0,\alpha).\]
\end{lemma}
\begin{proof}
Put $U=U_0 \cup U_1$ and
\[ W = (U \times [0,1]) \setminus \big( (U \setminus U_0) \times \{0 \} \cup (U \setminus U_1) \times \{1 \} \big).\]
Pull back $\omega$, $\phi$, $v_0,v_1$, $\alpha$ along the projection $U \times [0,1] \rightarrow U$, and then restrict to $W$ (pullbacks omitted from notation).  Using $t$ as a coordinate on $[0,1]$, let
\[ v=(1-t)v_0+tv_1,\]
and let $Z=\{p|v_W(p)=0 \}$ be the corresponding localizing set (it need not be smooth).  Since $v_1,v_0$ agree on $\supp(\alpha)$,
\[v|_{\supp(\alpha)}=v_1|_{\supp(\alpha)}=v_0|_{\supp(\alpha)}.\]
It follows that $\pair{\phi}{v}$ is proper and bounded below on $\supp(\alpha) \cap Z$.  Let $(W,\tomega,\tphi)$ be a $v$-polarized completion of $(W,\omega,\phi,\alpha)$ (using Lemma \ref{ExistencePart}).  Since $v$ is bounded, $\tphi$ is proper on the support of $\alpha$.  Thus $W$ is a proper Hamiltonian cobordism between the two boundary components $U_0$, $U_1$.  By Theorem \ref{CobordismEquality}
\begin{equation}
\label{EqualityOfParticularDH}
\DH(U_0,\tomega|_{U_0},\tphi|_{U_0},\alpha)=\DH(U_1,\tomega|_{U_1},\tphi|_{U_1},\alpha).
\end{equation}
For $i=0,1$, $(U_i,\tomega|_{U_i},\tphi|_{U_i})$ is a $v|_{U_i}=v_i$-polarized completion of $(U_i,\omega,\phi,\alpha)$.  However, recall that $U_i$ is a $G$-invariant open set that can be smoothly collapsed to part of $Z_i$.  Remark \ref{SmoothCollapseCohomologous} shows that in this case, any $v_i$-polarized completion has the \emph{same} DH distribution.  Hence, the result actually follows from \eqref{EqualityOfParticularDH}.
\end{proof}

\noindent Finally, we outline the proof of Theorem \ref{HKTheorem}.
\begin{theorem}[\cite{KarshonHarada} Theorem 5.20]
\label{BasicCobordismResult}
Consider the setting of Lemma \ref{Hypotheses}.  Suppose further that $\pair{\phi}{v}$ is proper and bounded below on the support of $\alpha$.  Then
\[ \DH(N,\omega,\phi,\alpha) = \DH^{v}_Z(N,\omega,\phi,\alpha). \]
\end{theorem}
\begin{proof}
Let $U$ be a $G$-invariant tubular neighbourhood of $Z$, and put
\[ W = (N \times [0,1]) \setminus \big((N\setminus U) \times \{1 \}\big), \hspace{1cm} Y=(Z \times [0,1]) \cup (N \times \{0 \}).\]
Pull back the equivariant forms and $v$ to $N \times [0,1]$ and restrict to $W$ (pullbacks omitted from notation).  By assumption, $\pair{\phi}{v}$ is proper and bounded below on $Y \cap \supp(\alpha)$.  Apply Lemma \ref{ExistencePart} to obtain a $v$-polarized completion $(W,\tomega,\tphi)$ of $(W,\omega,\phi,\alpha)$, with $\tomega-\tphi$ being equal to $\omega-\phi$ in a neighbourhood of $Y$ (in particular, near $N \times \{0 \}$).  Then $W$ is a proper Hamiltonian cobordism between $(N,\omega,\phi,\alpha)$ and $(U,\tomega|_{U \times \{1 \}},\tphi|_{U \times \{1 \}}, \alpha|_U)$.  The result follows from Theorem \ref{CobordismEquality} and Remark \ref{SmoothCollapseCohomologous}.
\end{proof}

\section{Smoothness of $Z$}
Here we discuss the role of the assumption (Lemma \ref{Hypotheses}) that the localizing set $Z$ is a smooth submanifold.  The smoothness assumption leads to a particularly simple description of the contribution of a component $Z_i \subset Z$ to \eqref{HKFormula}: it is the twisted DH distribution of \emph{any} $v$-polarized completion of a tubular neighbourhood of $Z_i$.

However, in general $Z=\{v_N=0 \}$ is not smooth.  On the other hand, the cobordism used in the H-K Theorem (see Theorem \ref{BasicCobordismResult}, or \cite{KarshonHarada} Theorem 5.20) does not use the smoothness assumption.  Neither does the `existence' part of Lemma \ref{Hypotheses}, since, for example, for the `polarized completion' $(U,\tomega,\tphi)$ which is constructed (see Lemma \ref{ExistencePart}, or \cite{KarshonHarada} Proposition 3.4), $\tomega-\tphi$ agrees with $\omega-\phi$ on an open neighbourhood of $Z$, and so also makes sense when $Z$ is singular.  The assumption only plays a role in the `uniqueness' part of Lemma \ref{Hypotheses} (see Lemma \ref{UniquenessPart}, or \cite{KarshonHarada} Lemmas 4.12, 4.17).  Thus, one obtains a localization formula for the Duistermaat-Heckman distribution quite generally, but the contribution $\DH(U_i,\tomega,\tphi,\alpha)$ of a component $Z_i$ is not as simple to describe.

In the general case, one can use the polarized completion appearing in the proof of Lemma \ref{ExistencePart}.  Alternatively, to get hold of the contribution $\DH(U_i,\tomega,\tphi,\alpha)$, one can try to choose a new taming map $v^\prime$ on $N^\prime :=U_i$ now with a smooth localizing set, and then apply the H-K Theorem to $N^\prime$ obtaining a new Hamiltonian cobordism from $N^\prime$ to a collection of smaller open sets $U^\prime_{j}$ around the components $Z_{j}^\prime$ of the localizing set for $v^\prime$.  Then
\[ \DH(N^\prime,\tomega,\tphi,\alpha)=\sum_j \DH^{v^\prime}_{Z_{j}^\prime}(N^\prime,\omega^\prime,\phi^\prime,\alpha)\]
and the uniqueness part of Lemma \ref{Hypotheses} applies to the terms on the right side of the equation.  In this way one obtains a formula for $\DH(N,\omega,\phi,\alpha)$ using a composition of Hamiltonian cobordisms.

In the case $v=\phi$ where $\phi:N \rightarrow \liet^\ast \simeq \liet$ is the moment map for the action of a torus $T$, a suitable $v^\prime$ on $N^\prime=U_i$ can be obtained by perturbing $v$ to $v^\prime=v-\gamma$, where $\gamma$ is a small `generic' element of $\liet$ (see Section 4.3).  The perturbation $\gamma$ can even be chosen independently for each component $Z_i$ of $Z$.  This is similar to the perturbation used by Paradan in \cite{Paradan98}.

To see that the Harada-Karshon Theorem applies to $(N^\prime=U_i,\tomega,\tphi,\alpha)$ equipped with taming map $v^\prime=v-\gamma$ (with $\gamma$ sufficiently small), one needs to check that $\pair{\tphi}{v^\prime}$ is proper and bounded below on $N^\prime \cap \supp(\alpha)$.  The moment map $\tphi$ of the polarized completion $(N^\prime,\tomega,\tphi)$ can be obtained from the proof of Lemma \ref{ExistencePart}:
\[ \tphi = \phi - (\psi-\|\phi\|^2)\Theta_N, \hspace{0.5cm} \pair{\Theta_N}{X}=\frac{g(v_N,X_N)}{g(v_N,v_N)}, \hspace{0.5cm} X \in \liet\]
where $g$ is a $T$-invariant Riemannian metric on $N$ and $\psi \ge \|\phi\|^2$ is a smooth $T$-invariant function on $N^\prime$ which is proper and bounded below on $\supp(\alpha)$, and agrees with $\pair{\phi}{v}=\|\phi\|^2$ on a neighbourhood of $Z_i$.  (Note that $\psi - \|\phi\|^2=0$ on a neighbourhood of $Z_i=\{v_N=0 \} \cap N^\prime$, so it is not a problem that $\Theta_N$ is not defined there.)  Then
\[ \pair{\tphi}{v^\prime}=\psi - \pair{\phi}{\gamma}+(\psi-\|\phi\|^2)\frac{g(v_N,\gamma_N)}{g(v_N,v_N)}.\]
In our case, the neighbourhood $N^\prime=U_i$ is such that $\supp(\alpha)\cap \overline{N^\prime}$ is compact and $\overline{N^\prime} \cap \{v_N=0 \}=Z_i$.  This implies that $|\pair{\phi}{\gamma}|\le C_1$ and $|g(v_{N},\gamma_{N})|\le C_2$ are bounded on $\supp(\alpha) \cap \overline{N^\prime}$, and that $g(v_N,v_N)$ can be bounded below by a positive constant $0<C_3\le 1$ on $(\overline{N^\prime} \cap \supp(\alpha))\setminus \supp(\psi-\|\phi\|^2)$.  Replacing $\gamma$ with $C_3(C_2+1)^{-1}\gamma$, we have
\[ \pair{\tphi}{v^\prime} \ge \psi - C_1 - \frac{C_2}{C_2+1}\psi=\frac{1}{C_2+1}\psi -C_1\]
on $N^\prime \cap \supp(\alpha)$.  This implies that $\pair{\tphi}{v^\prime}$ is proper and bounded below on $N^\prime \cap \supp(\alpha)$.

Because of the `uniqueness' part of Lemma \ref{Hypotheses}, the formula obtained from the above composition of Hamiltonian cobordisms is the same as that obtained from perturbing the taming map $v^\prime=v-\gamma$ from the beginning, as we did in the main part of the paper.  From this one can see that the terms $\frakm_\beta$ of \eqref{FormulaNotPerturbed} are the same as the contributions obtained using the unperturbed taming map $v=\hPhi_{\a}$.

\begin{remark}
Harada-Karshon work with a weaker condition: that a neighbourhood of the localizing set $Z$ admit a \emph{smooth equivariant weak deformation retract} to $Z$.  (This is the same as what we have shortened to `smooth collapse' (Definition \ref{DefinitionSmoothCollapse}), except with $Z^\prime$ not required to be smooth.)  This condition is appealing in that it guarantees the `uniqueness' part of Lemma \ref{Hypotheses}.  An interesting question left open in \cite{KarshonHarada} is whether the localizing set $Z$ in the case $v=\phi$ always satisfies this weaker condition.
\end{remark}

{\small \bibliographystyle{plain} %Putting ``small'' around this, shrinks the font in the references section

}
\end{document}